\documentclass[11pt, a4paper, oneside]{article}
\usepackage[body={16.5cm, 21.0cm},left=1.5cm,right=1.5cm]{geometry}
\usepackage{natbib}
\usepackage{amsmath, amssymb}
\usepackage{charter}
\usepackage{graphicx}

\usepackage{setspace}
\usepackage{pifont}
\usepackage{fancyhdr}
\newtheorem{thm}{Theorem}
\newtheorem{lem}[thm]{Lemma}
\newtheorem{rem}[thm]{Remark}
\newtheorem{cor}[thm]{Corollary}

\newcommand{\keywords}[1]{{\scriptsize \noindent \textbf{KEY WORDS AND PHRASES:}} {#1}\\}
\newenvironment{proof}{\noindent{\bf Proof:}}{\hfill \ding{111} \\}

\newcommand{\id}{\ensuremath{\displaystyle{\mathop {=} ^d}}}

\newcommand{\field}[1]{\mathbb{#1}}
\newcommand{\real}{\ensuremath{{\field{R}}}}
\newcommand{\mc}[1]{{\ensuremath{\mathcal{#1}}}}

\newcommand{\sumab}[2]{\ensuremath{\sum\limits_{#1}^{#2}}}

\newcommand{\intunit}{\ensuremath{\int_{0}^{1}}}

\newcommand{\arrowf}[1]{\ensuremath{\displaystyle {\mathop {\longrightarrow}_{#1 \rightarrow \infty}\,}}}
\newcommand{\limit}[1]{\ensuremath{\displaystyle {\lim_{#1 \rightarrow{\infty}}}}}

\newcommand{\argmin}[1]{\ensuremath{\displaystyle {\mbox{\footnotesize arg}\,\min_{#1}}}}
\newcommand{\conv}[1]{\ensuremath{\, \displaystyle {\mathop {\longrightarrow}_{n \rightarrow \infty} ^{#1}}}\, }

\newcommand{\kdivn}{\frac{k}{n}}
\newcommand{\ndivk}{\frac{n}{k}}

\title{On tail trend detection: modeling relative risk\thanks{Research partially supported by ENES -- Extremes in Space, project
PTDC/MAT/112770/2009 and by national funds through the Funda\c c\~ao Nacional para a Ci\^encia e Tecnologia, Portugal -- FCT under the project PEst-OE/MAT/UI0006/2011.}}

\author{Laurens de Haan \\  
 {\small University of Lisbon}\\ {\small Erasmus University Rotterdam} \and {Albert Klein Tank}\\  {\small KNMI, Royal Netherlands}\\ {\small Meteorological Institute} \\
  \and {Cl\'{a}udia Neves}\\ {\small CEAUL and University of Aveiro}}
\date{}

\begin{document}

\maketitle


\abstract{The climate change dispute is about  changes over time of environmental characteristics (such as rainfall). Some people say that a possible change is not so much in the mean but rather in the extreme phenomena (that is, the average rainfall may not change much but heavy storms may become more or less frequent). The paper studies changes over time in the probability that some high threshold is exceeded. The model is such that the threshold does not need to be specified, the results hold for any high threshold. For simplicity a certain linear trend is studied depending on one real parameter. Estimation and testing procedures (is there a trend?) are developed. Simulation results are presented. The method is applied to trends in heavy rainfall at 18 gauging stations across Germany and The Netherlands. A tentative conclusion is that the trend seems to depend on whether or not a station is close to the sea.}

\vspace{0.2cm}



\section{Introduction}
\label{SecIntro}

In the climate change dispute some people suggest
(\citet{KleinTankKonnen:03,Groismanetal:05,Alexanderetal:06,Zolinaetal:09}) that
perhaps there is no or little change in the mean of the probability distribution of daily rainfall over time but there is a change in the tail that
is, more extreme events occur more frequently. The present paper --
like \citet{Smith:89,HallTajvidi:00,Haneletal:09} -- considers a trend in extremes from the point
of view of extreme value theory.

If one wants to concentrate on a trend connected with extreme events rather than with the central part of the probability distribution function $F$, one should look at a high quantile  $F^{\leftarrow}(p)$ (i.e. the inverse of $F$) for $p$ close to one  or at the exceedance probability $1-F(x)$ at a high level $x$. Hence we consider the limit behavior of $F^{\leftarrow}(p)$ as $p\uparrow 1$ or $F(x)$ as $x \uparrow x^*$, which is the right end point of the probability distribution ($x^*:=\sup\{x:\, F(x)<1\}$). Since the limit relation for $F$ is simpler than for $F^{\leftarrow}$, we concentrate on the behavior of $F(x)$ as $x\uparrow x^*$.

Consider random variables $X(s)$ where $s\geq 0$ is time. Write $F_s(x):= P\{X(s)\leq x\}$ for $x \in \real$. We assume that for all $s>0$
\begin{equation*}
    \frac{1-F_s(x)}{1-F_0(x)}
\end{equation*}
tends to a positive constant for all $s>0$ when $x$ tends to the
right endpoint $x^*$ of $F_0$. Hence the exceedance probability at time $s$ is
systematically a factor times the exceedance probability at time
zero. We consider a simple model for relative risk and assume that for some real trend constant $c$ and all $s\geq 0$ 
\begin{equation}\label{TrendPropag}
    \displaystyle {\lim_{x \uparrow x^*}}\,
    \frac{1-F_s(x)}{1-F_0(x)}= e^{cs}.
\end{equation}
This means that for example (with $s=1$ and $c=1$) that the probability of any extreme event taking place at time $s$ is $e$ times the corresponding probability at time zero. For $c\,s$ small the limit function is approximately linear.

For our analysis we shall need the context of extreme value theory that is, we assume that the distribution function $F_0$ is in the domain of attraction of some extreme value distribution $G_{\gamma}$ i.e., there exist sequences of constants $a_n>0$ and $b_n$ ($n=1,2,\ldots$) such that the normalized maximum of a sample from $F_0$ converges to $G_\gamma$ for some $\gamma \in \real$:
\begin{equation}\label{DOA}
\limit{n} \, F^n(a_n x+b_n)= G_{\gamma}(x):= \exp\{-(1+\gamma x)^{-1/\gamma}\}
\end{equation}
for all $x$ for which $1+\gamma x>0$ (notation $F_0 \in \mc{D}(G_{\gamma})$). This condition is really a condition on the tail $1-F_0$ of the probability distribution since it is equivalent to
\begin{equation}\label{DOAalt}
\limit{t} \, t\bigl\{1-F_0\bigl(a_0(t)x+b_0(t)\bigr)\bigr\}=(1+\gamma x)^{-1/\gamma}
\end{equation}
for $x>0$ where $a_0(t):= a_{[t]}$ and $b_0(t):= b_{[t]}$ (and $[t]$ is the integral part of $t$) \citep[cf. e.g.][]{Coles:01,deHaanFerreira:06}.

We return to condition \eqref{TrendPropag}. Let us first translate this condition into a condition for high quantiles as we discussed before.

Condition \eqref{TrendPropag} is equivalent to the following condition on the quantile function $F^{\leftarrow}$:
\begin{equation}\label{FsInv2F0Inv}
\lim_{p \uparrow 1} \, \frac{F_s^{\leftarrow}(p)-F_0^{\leftarrow}(p)}{a_0\bigl(\frac{1}{1-p}\bigr)}= \frac{e^{c\gamma s}-1}{\gamma}.
\end{equation}
It is convenient (and sometimes usual) to change the notation a bit at this point. Consider the functions $U_s$ defined by
\begin{equation*}
U_s(t)= F_s^{\leftarrow} \Bigl( 1-\frac{1}{t}\Bigr)= \biggl(\frac{1}{1-F_s}\biggr)^{\leftarrow}(t) \quad \mbox{ for } t>1.
\end{equation*}
Relation \eqref{TrendPropag} is seen to be equivalent to
\begin{equation}\label{UsU0Dif}
    \limit{t} \frac{U_s(t)-U_0(t)}{a_0(t)}=
    \frac{e^{ c\gamma s}-1}{\gamma}.
\end{equation}
For $\gamma>0$ this relation can be simplified. In this case it is equivalent to
\begin{equation}\label{Us2U0}
    \limit{t}\, \frac{U_s(t)}{U_0(t)}= e^{c\gamma s}.
\end{equation}
The condition $F_0 \in \mc{D}(G_\gamma)$ in conjunction with \eqref{TrendPropag} implies that the extreme value condition holds for any $F_s$ with the same limit distribution i.e. (cf. \eqref{DOAalt})
\begin{equation}\label{DOAesses}
\limit{t} \, t\bigl\{1-F_s\bigl(a_s( t)x+b_s(t)\bigr)\bigr\}=(1+\gamma x)^{-1/\gamma}
\end{equation}
where $a_s>0$ and $b_s$ are appropriately chosen functions. We shall use the well-known facts that \eqref{DOAesses} is equivalent to
\begin{equation}\label{ERVUs}
    \limit{t}\frac{U_s(tx)-U_s(t)}{a_s(t)}=
    \frac{x^{\gamma}-1}{\gamma} \quad \mbox{ for } x>0
\end{equation}
and that 
\begin{equation}\label{Aux}
b_s(t)-U_s(t)= o\bigl(a_s(t)\bigr), \; t \rightarrow \infty\, ; \quad \limit{t} \frac{a_s(tx)}{a_s(t)}= x^{\gamma} \; \mbox{ for } x>0.
\end{equation}
Moreover for $s>0$
\begin{equation}\label{ScaleFct}
	\limit{t}\, \frac{a_s(t)}{a_0(t)}= e^{c\gamma s}.
\end{equation}
All the mentioned implications will be proved in Appendix \ref{AppRel}.

We have restricted ourselves to the model in \eqref{TrendPropag} that is, a trend function of the form $e^{cs}$, since we are interested in a monotone trend and also because a more general change would have been more difficult to detect. Figure \ref{Fig.Moms} gives some insight into the difficulty of detecting more complex trend functions, namely a temporal trend in the extreme value index $\gamma$. In the future we shall study more general (not monotone) changes in a similar manner, possibly with adjustments enabling other appropriate estimation procedures.

\begin{figure}
\caption{\footnotesize Estimates $\hat{\gamma}(s_j)$ for each one of the $m+1=18$ time points $s_j=j/m$, $j=0,1,2,\ldots, m$, at one particular location, a gauging station in The Netherlands. The number $k$ corresponds to the number of observations above the hight random threshold $X_{n-k,n}(s_j)$ (Left). Corresponding estimates $\hat{\gamma}(s_j)$ with $k=30$ (Right).} 
\begin{center}
\includegraphics*[scale=0.45]{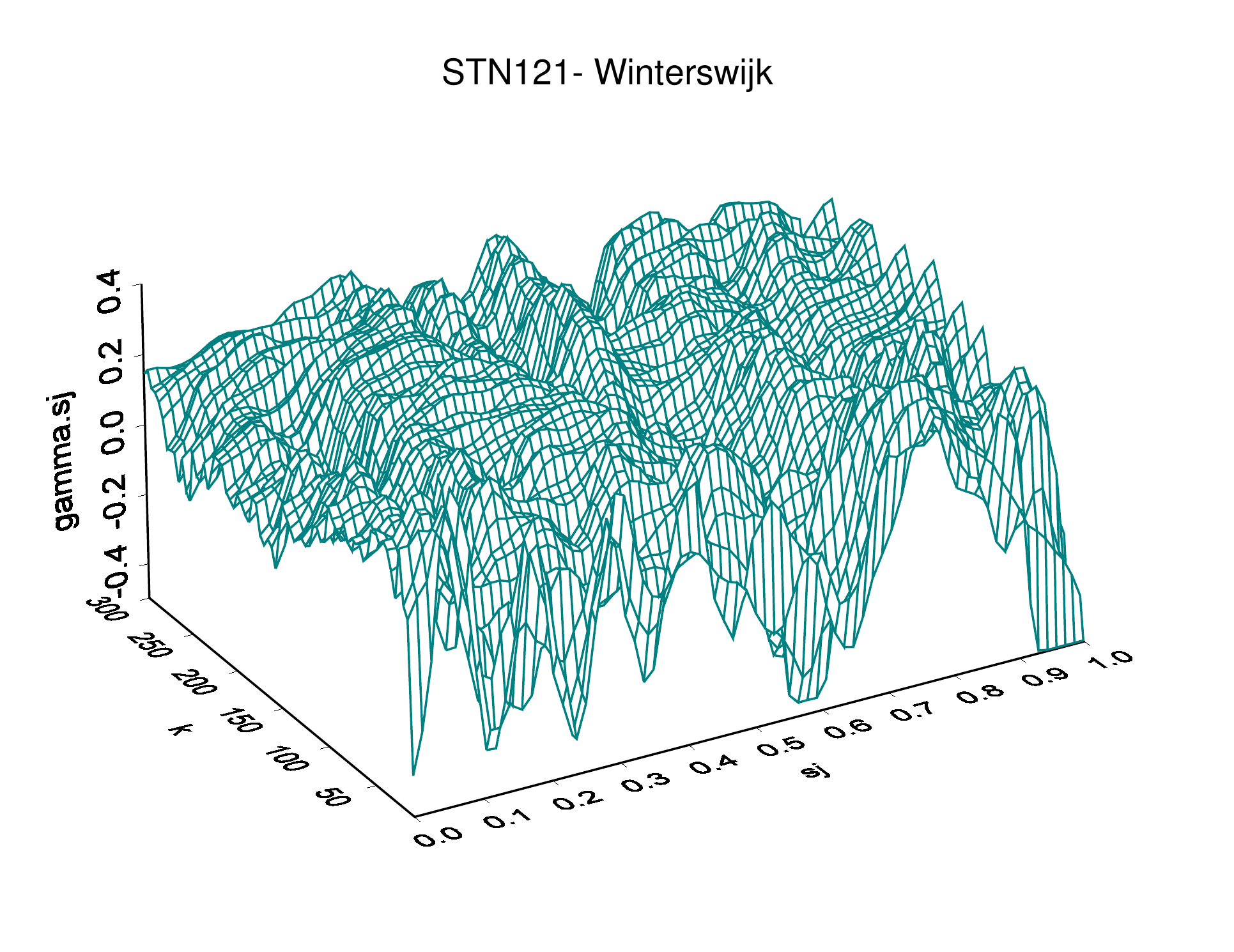}
\includegraphics*[scale=0.4]{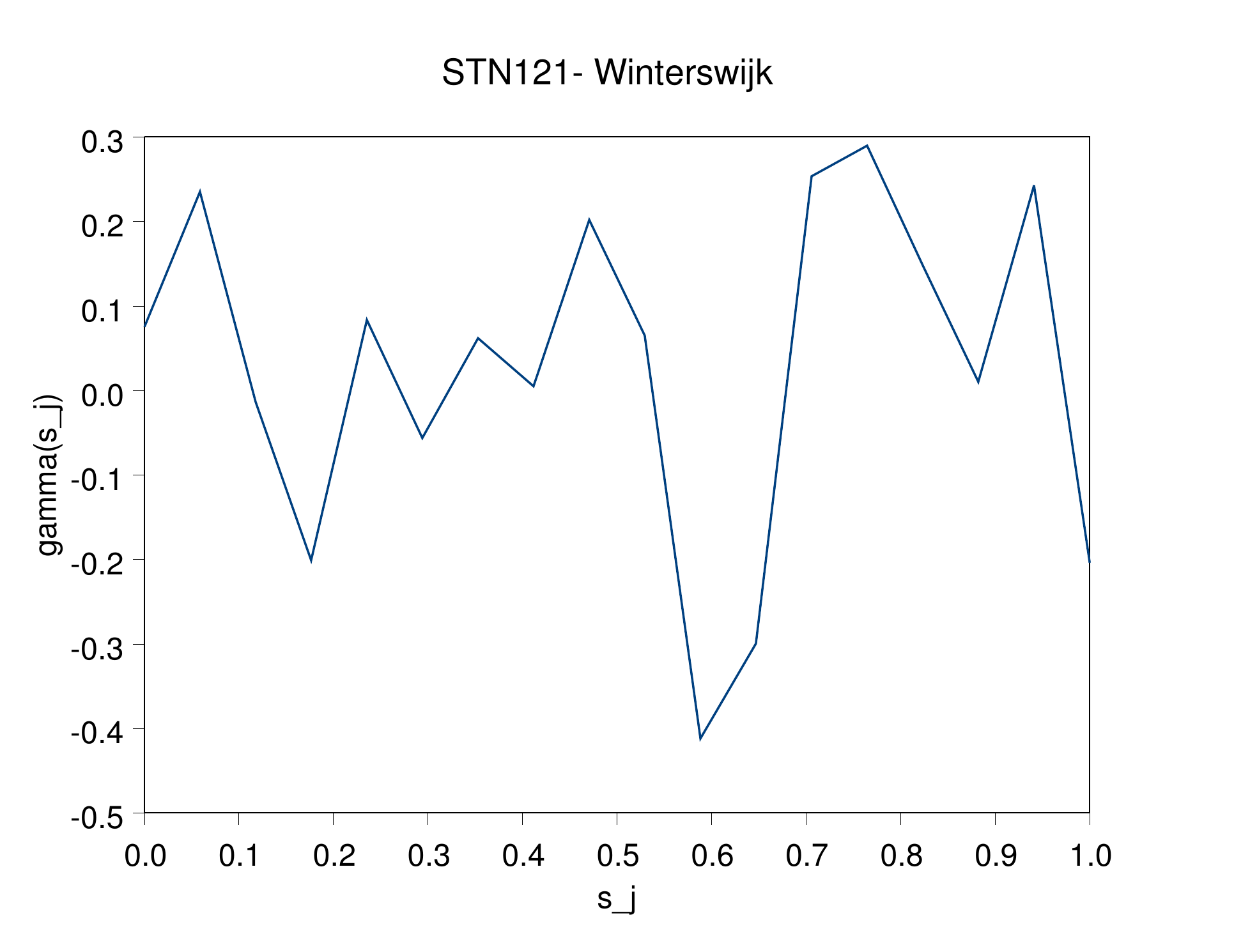}
\label{Fig.Moms}
\end{center}
\end{figure}

The aim of the paper is to develop estimators and testing procedures for the parameter $c$. This will be done in a semi-parametric way on the basis  of the limit relations \eqref{TrendPropag}, \eqref{UsU0Dif} and \eqref{Us2U0}.

Suppose that we have repeated observations at discrete time points 
$0=s_0<s_1<s_2<\ldots<s_m$. It is assumed that
$\bigl\{X_i(s_j)\bigr\}^{\;n\quad m}_{i=1\, j=1}$ are all
independent and that $X_1(s_j),\, X_2(s_j), \ldots, X_n(s_j)$ have
the same distribution function $F_{s_j}$ for all $j$. Let $X_{1,n}(s_j) \leq X_{2,n}(s_j) \leq \ldots \leq X_{n,n}(s_j)$ be their order statistics.

Since in the limit relations (e.g. relation \eqref{FsInv2F0Inv}) only high quantiles play a role, we can expect that the impact of such a condition can be detected only among the higher order statistics. Our estimators will be based on the set of $k$ highest order statistics $\bigl(X_{n-k,n}(s_j), X_{n-k+1,n}(s_j), \ldots, X_{n,n}(s_j)\bigr)$  for each $j$. In order to be able to apply the law of large numbers and the like we need to let $k$ depend on $n$, $k=k_n$ and $\lim_{n\rightarrow \infty} k_n= \infty$. On the other hand we want to determine $k$ in such a way that we deal with the tail of the distribution only. That leads to the condition $k_n= o(n)$, $n\rightarrow \infty$. For the asymptotic normality of the estimators a further restriction will be imposed on the sequence $k_n$.

Now we are ready to construct estimators for $c$.

\begin{description}
  \item[(i)] Consider first the simplest case, $\gamma >0$. Relation \eqref{Us2U0} implies $(s\geq 0)$
\begin{equation}\label{LimLogs}
\limit{n} \, \log U_s\bigl( \frac{n}{k}\bigr) -  \log U_0\bigl( \frac{n}{k}\bigr)= c\gamma s.
\end{equation}
We are going to replace the quantities at the left hand side by their sample analogs. It will be proved (Appendix \ref{AppParetos}) that for $j=1,2,\ldots,m$
\begin{equation}\label{LimLogsEmp}
\log X_{n-k,n}(s_j)-  \log U_{s_j}\bigl( \frac{n}{k}\bigr) \conv{P} 0.
\end{equation}
as $n\rightarrow \infty$. Hence under condition \eqref{LimLogs}
\begin{equation*}
\sumab{j=1}{m} \Bigl(\log X_{n-k,n}(s_j)-\log X_{n-k,n}(0)- c\gamma s_j \Bigr)^2
\end{equation*}
should be small. This leads to the least squares estimator
\begin{equation}\label{Est1}
     \hat{c}^{(1)}:= \frac{\sumab{j=1}{m}s_j\,\bigl(\log X_{n-k,n}(s_j)-\log X_{n-k,n}(0)\bigr)}{\hat{\gamma}^+_{n,k}
     \sumab{j=1}{m}s_j^2},
\end{equation}
where $\hat{\gamma}^+_{n,k}$ is some estimator of $\max (\gamma,0)$ in section \ref{SecResults}. We shall discuss specific estimators of $\max (\gamma,0)$ in section \ref{SecResults} and the simulations section \ref{SecSims}.

 \item[(ii)] For $\gamma$ not restricted to be positive \eqref{UsU0Dif} leads to an estimator
  for $c$. Intuitively relation \eqref{UsU0Dif} means that
  \begin{equation*}
    \biggl( 1+\hat{\gamma}_{n,k}\frac{X_{n-k,n}(s_j)-X_{n-k,n}(0)}{\hat{a}_0\bigl( \ndivk \bigr)}\biggr)^{\frac{1}{\hat{\gamma}_{n,k}}}\approx e^{cs_j},
  \end{equation*}
  where $\hat{\gamma}_{n,k}$ is an estimator for $\gamma$ and $\hat{a}_0\bigl( n/k \bigr)$ is an estimator for $a_0\bigl(n/k \bigr)$.
  Define $\hat{c}^{(2)}$ by
  \begin{equation*}
    \argmin{c} \sumab{j=1}{m}\Bigl\{\log \Bigl( 1+\hat{\gamma}_{n,k}\frac{X_{n-k,n}(s_j)-X_{n-k,n}(0)}{\hat{a}_0\bigl( \ndivk
    \bigr)}\Bigr)^{\frac{1}{\hat{\gamma}_{n,k}}}-cs_j\Bigr\}^2
  \end{equation*}
  i.e.,
    \begin{equation}\label{Est2}
    \hat{c}^{(2)}:= \frac{\sumab{j=1}{m}s_j\,\log \Bigl( 1+\hat{\gamma}_{n,k}\,\frac{ X_{n-k,n}(s_j)-X_{n-k,n}(0)}{\hat{a}_{0}(n/k)}\Bigr)^{\frac{1}{\hat{\gamma}_{n,k}}}}{\sumab{j=1}{m} s_j^2}.
\end{equation}
For $\hat{\gamma}_{n,k}=0$, the estimator is defined by continuity. Again, specific (well-known) estimators $\hat{\gamma}_{n,k}$ and $\hat{a}_{0}(n/k)$ will be discussed in section \ref{SecResults}.

 \item[(iii)] Finally relation \eqref{TrendPropag} also leads to an
  estimator for $c$. Intuitively relation \eqref{TrendPropag} means that
    \begin{equation*}
        \log\,
        \frac{1-\widehat{F}_s\bigl(X_{n-k,n}(0)\bigr)}{1-\widehat{F}_0\bigl(X_{n-k,n}(0)\bigr)}
        \approx cs
    \end{equation*}
    where $\widehat{F}_s$ is the empirical distribution function at
    time $s$. Note that $1-\widehat{F}_0\bigl(X_{n-k,n}(0)\bigr)\,\approx\, k/n$. Hence the estimator:
    \begin{equation}\label{Est3}
    \hat{c}^{(3)}:= \frac{\sumab{j=1}{m}\log\Bigl(\frac{1}{k}\sumab{i=1}{n} I_{\{X_i(s_j)>X_{n-k,n}(0)\}}\Bigr)}{\sumab{j=1}{m}s_j}.
\end{equation}
\end{description}

The problem of defining and estimating a trend in extreme value theory has been considered by a number of authors including \citet{Smith:89, HallTajvidi:00, Coles:01, YeeStephenson:07} and more recently addressed by \citet{Mannshardtetal:10}. What distinguishes our approach from the traditional ones is (among others):
\begin{itemize}
\item The results are directly interpretable (it is about how probabilities of extreme events change over time).
\item Asymptotic justification: we prove that our estimators are not only valid when the observations come from an extreme value distribution but also under the more realistic assumption that they come from a distribution in the domain of attraction.
\item Some existing proposals are not completely satisfactory. A review and discussion of existing results is given in Appendix \ref{app}.
\end{itemize}

The outline of this paper is as follows. In section \ref{SecResults} consistency
and asymptotic normality of the estimators introduced in \eqref{Est1}, \eqref{Est2} and \eqref{Est3} is discussed. Proofs are postponed to
section \ref{SecProofs}. In section \ref{SecSims} we collect some
simulation results for illustrating and assessing finite sample performance of the
various estimators for the trend. In section \ref{SecData} we apply the
methods to daily rainfall at 18 stations across Germany and The
Netherlands and give a tentative interpretation of
the results. Indeed for some stations the probability of extreme rainfall has increased by about $2\%$ in each decade.

\section{Results}
\label{SecResults}

{\bf (i)} Let us consider $\hat{c}^{(1)}$ first and suppose that $\gamma>0$. For part of our results we need a second order strengthening of conditions \eqref{TrendPropag} and \eqref{DOA}.

{\sc Condition A}\hspace{0.2cm}  Suppose there exists a
positive or negative function $\beta$ with $\lim_{t\rightarrow \infty}
\beta(t)=0$ such that for $x>0$
\begin{equation}\label{2ndRVU0}
    \limit{t}\frac{\frac{U_0(tx)}{U_0(t)}-x^{\gamma^+}}{\beta(t)}=
    x^{\gamma^+}\,\frac{x^{\widetilde{\rho}}-1}{\widetilde{\rho}}
\end{equation}
with $\widetilde{\rho}$ a non-positive parameter. Further we need a second order
strengthening of condition \eqref{Us2U0}: suppose that for all $j$
\begin{equation*}
    \limit{t}\frac{\frac{U_{s_j}(t)}{U_0(t)}-e^{c\gamma^+ s_j}}{\beta(t)}=
    e^{c\gamma^+ s_j}\frac{e^{c\widetilde{\rho} s_j}-1}{\widetilde{\rho}}.
\end{equation*}
Equivalently
\begin{equation}\label{2ndUs2U0}
	\limit{t} \frac{\log U_{s_j}(t)- \log U_0(t)-c\gamma^+ s_j}{\beta(t)}=\frac{e^{c\widetilde{\rho} s_j}-1}{\widetilde{\rho}}.
\end{equation}

We consider an estimator $\hat{\gamma}^+_{n,k}(s_j)$ of $\gamma^+$ that is
consistent i.e., $\hat{\gamma}^+_{n,k}(s_j)\conv{P}\gamma^+$ provided $k=k_n \rightarrow \infty$,
$k_n/n \rightarrow 0$. Furthermore we require that under condition A
\begin{equation}\label{BivANgammaPlus}
    \sqrt{k}\Bigl( \hat{\gamma}^+_{n,k}(s_j)-\gamma^+,\; \log X_{n-k,n}(s_j)-\log U_{s_j}\bigl(\ndivk \bigr)\Bigr) \conv{d} \bigl(\Gamma^+(s_j), \, B^+(s_j)
    \bigr),
\end{equation}
say, for all $j$, where $\bigl(\Gamma^+(s_j), \, B^+(s_j)\bigr)$ has
a multivariate normal distribution provided $k$ (the number of upper
order statistics used in $\hat{\gamma}^+_{n,k}(s_j)$ for all $j$)
satisfies $k= k_n \rightarrow \infty$ and
\begin{equation}\label{ConstBias}
    \limit{n} \sqrt{k_n}\, \beta\Bigl(\frac{n}{k_n} \Bigr) =: \lambda
\end{equation}
exists finite. Various estimators  $\hat{\gamma}^+_{n,k}(s_j)$
are known with this property, notably Hill's estimator (\citet{Hill:75}) as explained now:

\begin{rem}
\label{RemHills}
	In sections \ref{SecSims} and \ref{SecData} we shall choose
	\begin{equation*}
		\hat{\gamma}_{n,k}(s_j):= \frac{1}{k} \sumab{i=0}{k-1} \log X_{n-i,n}(s_j)- \log X_{n-k,n}(s_j)
	\end{equation*}
(Hill's estimator). In that case
\begin{equation*}
\Bigl(\Gamma^+(s_j), \, B^+(s_j) \Bigr)\, \id \, \Bigl(\gamma \intunit \bigl(s^{-1}W(s)-W(1)\bigr)\, ds +\frac{\lambda}{1-\rho},\,W(1) \Bigr)
\end{equation*}
with $W$ Brownian motion, hence $\Gamma^+(s_j)$ and $B^+(s_j)$ are independent, $Var\bigl(\Gamma^+(s_j)\bigr)=\gamma^2$, $Var\bigl(B^+(s_j)\bigr)= 1$ \citep[][pages 52 and 76]{deHaanFerreira:06}.
\end{rem}

{\bf (ii)} Next we consider $\hat{c}^{(2)}$ (and $\hat{c}^{(3)}$). Again we need a second order strengthening of conditions \eqref{TrendPropag} and \eqref{DOA} for part of the results.

{\sc Condition B}\hspace{0.2cm}  Suppose there exists a positive or negative
  function $\alpha_0$ with $\lim_{t \rightarrow \infty}\alpha_0(t)=0$ such that
for each $x>0$
  \begin{equation}\label{2ndERVU0}
    \limit{t}\frac{\frac{U_{0}(tx)-U_{0}(t)}{a_0(t)}-\frac{x^{\gamma}-1}{\gamma}}{\alpha_0(t)}= \frac{1}{\rho}\biggl(
    \frac{x^{\gamma+\rho}-1}{\gamma+\rho}-\frac{x^{\gamma}-1}{\gamma}\biggr)=:H_{\gamma,\rho}(x)
\end{equation}
where $\rho$ is a non-positive parameter. For $\gamma=0$ and/or $\rho=0$ the limit function is defined by continuity. Further we need a strengthening of condition \eqref{TrendPropag} or rather \eqref{UsU0Dif}:
\begin{equation}\label{2ndUs2U0Dif}
    \limit{t}\frac{\frac{U_{s_j}(t)-U_{0}(t)}{a_0(t)}-\frac{e^{c\gamma s_j}-1}{\gamma}}{\alpha_0(t)}=H_{\gamma,\rho}(e^{c
    s_j}).
\end{equation}
Relations \eqref{2ndERVU0} and \eqref{2ndUs2U0Dif} imply that all functions $U_{s_j}$ satisfy a second order relation (cf. Lemma \ref{LemERVUs} below).
We consider estimators
$\hat{\gamma}_{n,k}(s_j)$ and $\hat{a}_{s_j}(n/k)$ that are consistent i.e.,
  \begin{equation}\label{ConsistHillandScale}
    \hat{\gamma}_{n,k}(s_j) \conv{P} \gamma, \quad
    \frac{\hat{a}_{s_j}\bigl(\ndivk \bigr)}{a_{s_j}\bigl(\ndivk \bigr)} \conv{P}
    1
  \end{equation}
provided $k=k_n \rightarrow \infty$, $k_n/n \rightarrow 0$,
$n\rightarrow \infty$.
Furthermore we require that under condition B
\begin{eqnarray*}
   & &  \sqrt{k}\biggl( \hat{\gamma}_{n,k}(s_j)-\gamma,\; \frac{\hat{a}_{s_j}\bigl(\ndivk \bigr)}{a_{s_j}\bigl(\ndivk \bigr)}-1,\;\frac{X_{n-k,n}(s_j)-U_{s_j}(n/k)}{a_{s_j}\bigl(\ndivk\bigr)}\biggr)\\
   &\conv{d}& \bigl(\Gamma(s_j), \, A(s_j),\,B(s_j)\bigr),
\end{eqnarray*}
say, where $\bigl(\Gamma(s_j),A(s_j),\,B(s_j)\bigr)$, $j=1,2,\ldots,m$, are independent random vectors and have a
multivariate normal distribution for each $j$  provided $k=k_n \rightarrow
\infty$, and
\begin{equation*}
    \limit{n} \sqrt{k_n}\, \alpha_0\Bigl(\frac{n}{k_n} \Bigr) =: \lambda
\end{equation*}
exists finite. Various estimators are known with these properties, in particular the ones we mention now:

\begin{rem}\label{RemCovariances}
	In sections \ref{SecSims} and \ref{SecData} we shall choose the moment estimator for $\hat{\gamma}_{n,k}$ and the associated scale estimator \citep[(3.5.9) p.102 and (4.2.4) p.130][]{deHaanFerreira:06}  for $\hat{a}_{s_j}(n/k)$. In this case, if $\lim_{n\rightarrow \infty} \sqrt{k}\, \beta(n/k)=0$, $B(s_j)$ and $\bigl(\Gamma(s_j),\,A(s_j)\bigr)$ are independent, $Var\bigl(B(s_j)\bigr)= 1$,
\begin{equation*}
Var\bigl(\Gamma(s_j)\bigr)= \left\{
                    \begin{array}{ll}
                       \gamma^2+1, & \mbox{ }\gamma \geq 0    \\
                      \frac{(1-\gamma)^2(1-2\gamma)(1-\gamma+6\gamma^2)}{(1-3\gamma)(1-4\gamma)} , & \mbox{ } \gamma <0
                    \end{array}
                  \right.
\end{equation*}
\begin{equation*}
Var\bigl(A(s_j)\bigr)= \left\{
                    \begin{array}{ll}
                       \gamma^2+2, & \mbox{ } \gamma \geq 0   \\
                       \frac{2-16\gamma+51\gamma^2-69\gamma^3+50\gamma^4-24\gamma^5}{(1-2\gamma)(1-3\gamma)(1-4\gamma)}, & \mbox{ } \gamma<0
                    \end{array}
                  \right.
\end{equation*}
and
\begin{equation*}
Cov\bigl(\Gamma(s_j), \, A(s_j)\bigr)= \left\{
                    \begin{array}{ll}
                       \gamma-1, & \mbox{ }\gamma \geq 0\\
                       \frac{(1-\gamma)^2(-1+4\gamma-12\gamma^2)}{(1-3\gamma)(1-4\gamma)}, & \mbox{ } \gamma <0
                    \end{array}
                  \right.
\end{equation*}
(pages 104, 131, 133 respectively \citet{deHaanFerreira:06}; the asymptotic biases -- in case $\lambda \neq 0$ in \eqref{ConstBias} -- can be found on the same pages).
\end{rem}

We have the following results.

\begin{thm}\label{ThmMain}
  1. Under conditions \eqref{TrendPropag} and \eqref{DOA}
   \begin{equation*}
    \hat{c}^{(1)} \conv{P} c.
  \end{equation*}
    Under condition A
        \begin{equation*}
        \sqrt{k}\,\bigl(\hat{c}^{(1)}-c \bigr) \conv{d}  \frac{\sumab{j=1}{m}s_j \Bigl(\frac{ B^+(s_j)-B^+(0)}{\gamma^+}+\frac{e^{c\widetilde{\rho} s_j}-1}{\widetilde{\rho} \gamma^+}\,\lambda\Bigr)}{\sumab{j=1}{m}s^2_j}-\frac{c}{
        \gamma^+}\,\frac{1}{m}\sumab{j=1}{m}\Gamma^+(s_j).
    \end{equation*}
  2. Under conditions \eqref{TrendPropag} and \eqref{DOA}
  \begin{equation*}
    \hat{c}^{(r)} \conv{P} c \quad \mbox{ for }\; r=2,\,3.
  \end{equation*}
    Under condition B
    \begin{eqnarray*}
        & &\sqrt{k}\,\bigl(\hat{c}^{(2)}-c \bigr)\conv{d} \sumab{j=1}{m}s_j\Bigl\{\frac{1-e^{-c\gamma s_j}-c \gamma s_j}{\gamma^2}\, \frac{1}{m}\sumab{i=1}{m}\Gamma(s_i) \\
        & & \mbox{\hspace{0.1cm} }+B(s_j)-e^{-c \gamma s_j}B(0)-\frac{1-e^{-c \gamma s_j}}{\gamma}\,A(0) + \lambda \,e^{-c \gamma s_j} H_{\gamma, \rho}(e^{c s_j})\Bigr\}\Big/\sumab{j=1}{m}s^2_j
    \end{eqnarray*}
    and
    \begin{equation*}
        \sqrt{k}\,\bigl(\hat{c}^{(3)}-c \bigr) \conv{d} \sumab{j=1}{m}\Bigl\{e^{-cs_j}W^{(s_j)}(e^{cs_j})- W^{(0)}(1)+ \lambda\,b_3(s_j)\Bigr\}\Big/\sumab{j=1}{m}s_j,
    \end{equation*}
    where $\{W^{(s_j)}(t)\}_{t \geq 0}$ are independent standard Brownian motions
    $(j=1,2,\ldots,m)$ and 
    \begin{equation*}
        b_3(s_j)= \left\{
                    \begin{array}{ll}
                       -\frac{e^{-cs_j\rho}+1}{\rho(\gamma +\rho)}, & \mbox{ } \gamma +\rho \neq 0, \,\rho<0,\\
                       -cs_j\frac{e^{cs_j\gamma}+1}{\gamma}, & \mbox{ } \gamma +\rho = 0, \,\rho<0,\\
                      \frac{2cs_j}{ \gamma}, & \mbox{ } \rho=0\neq \gamma,\\
                       -(cs_j)^2, & \mbox{ } \gamma=\rho=0.
                    \end{array}
                  \right.
\end{equation*}

\end{thm}

\begin{rem}\label{RemAsymptVar}
When choosing the estimators of the index and scale according to Remarks \ref{RemHills} and \ref{RemCovariances} we get:

\noindent the variance of the limit distribution of $\sqrt{k}\bigl(\hat{c}^{(1)}-c\bigr)$ is
\begin{equation*}
	\frac{1}{\Bigl( \sumab{j=1}{m}s_j^2\Bigr)^2}\,\frac{1}{(\gamma^+)^2}\biggl\{ \sumab{j=1}{m}s_j^2+\Bigl( \sumab{j=1}{m}s_j\Bigr)^2\biggr\}+ \frac{c^2}{m},
\end{equation*}
the variance of the limit distribution of $\sqrt{k}\bigl(\hat{c}^{(3)}-c\bigr)$ is
\begin{equation*}
	\frac{\sumab{j=1}{m}\bigl( 1+e^{-cs_j}\bigr)}{\Bigl(\sumab{j=1}{m}s_j\Bigr)^2}
\end{equation*}
and the variance of the limit distribution of $\sqrt{k}\bigl(\hat{c}^{(2)}-c\bigr)$ is
\begin{equation*}
	\frac{\Bigl(\sumab{j=1}{m}s_j\,\frac{1-e^{-c\gamma s_j}-c\gamma s_j}{\gamma^2}\Bigr)^2\frac{Var \bigl(\Gamma(0)\bigr)}{m}+ \sumab{j=1}{m}s_j^2+\Bigl(\sumab{j=1}{m}s_je^{- c\gamma s_j}\Bigr)^2+\Bigl(\sumab{j=1}{m}s_j\,\frac{1-e^{-c\gamma s_j}}{\gamma}\Bigr)^2 Var\bigl(A(0)\bigr)}{\Bigl(\sumab{j=1}{m}s_j^2\Bigr)^2}
\end{equation*}
with the variances of $\Gamma(0)$ and $A(0)$ as in Remark \ref{RemCovariances}.
Figure \ref{Fig.Var1} and the close-up Figure \ref{Fig.Var2} offer a comparison of the three variances for $\gamma=0.1$ which is close to the value of $\gamma$ that plays a role in the application section \ref{SecData}. Figure  \ref{Fig.Var3} depicts the three variances for $\gamma=0.5$.
\end{rem}
\begin{figure}[h]
\caption{\footnotesize Asymptotic variances from Remark \ref{RemAsymptVar} with equal lengths, i.e. $s_j=j/m$, $j=1,2,\ldots,m$, and $m=17$ for several values of $c$.} 
\begin{center}
\includegraphics*[scale=0.42]{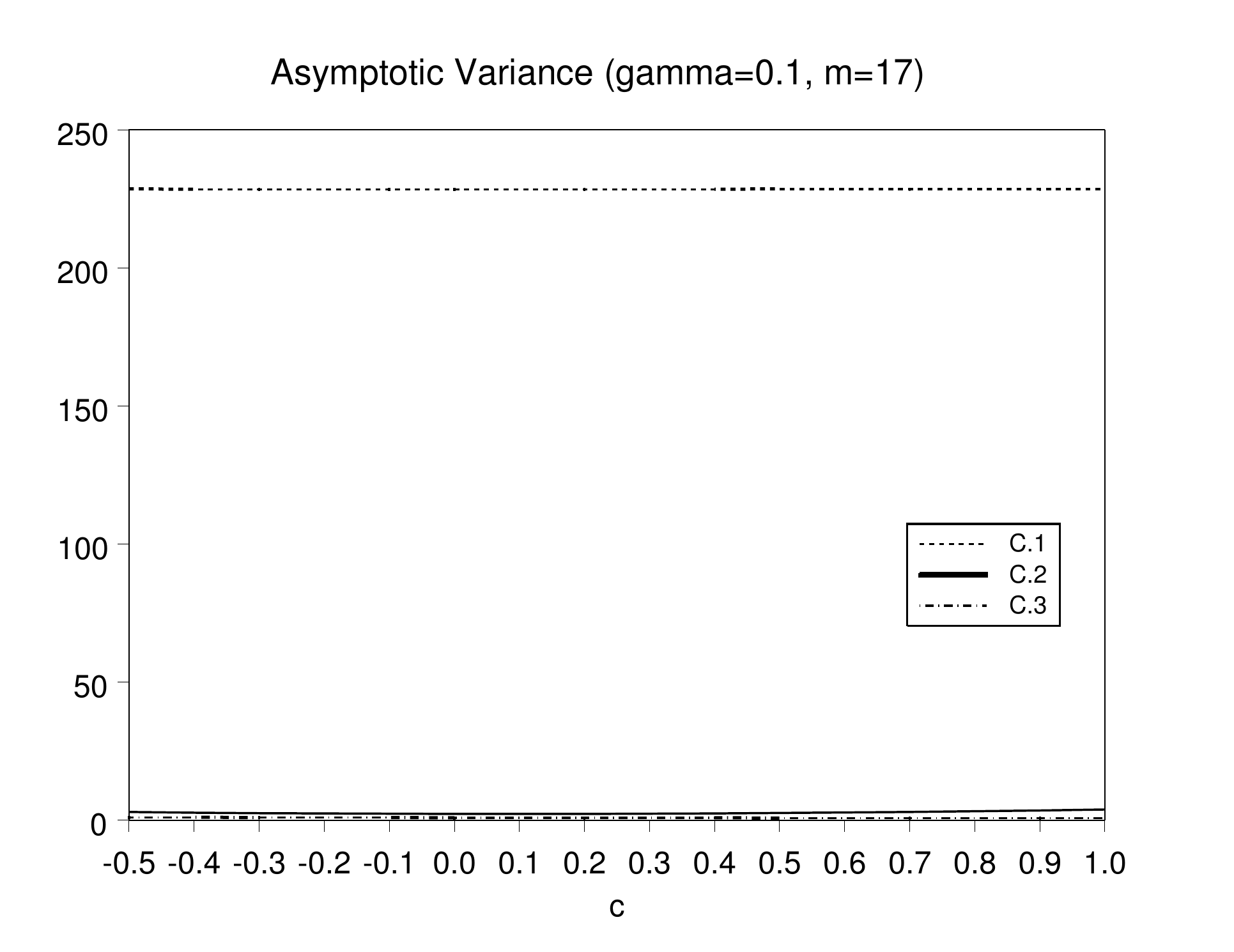}
\label{Fig.Var1}
\end{center}
\end{figure}
\begin{figure}[h]
\caption{\footnotesize Close-up of the asymptotic variances in Figure \ref{Fig.Var1}.} 
\begin{center}
\includegraphics*[scale=0.42]{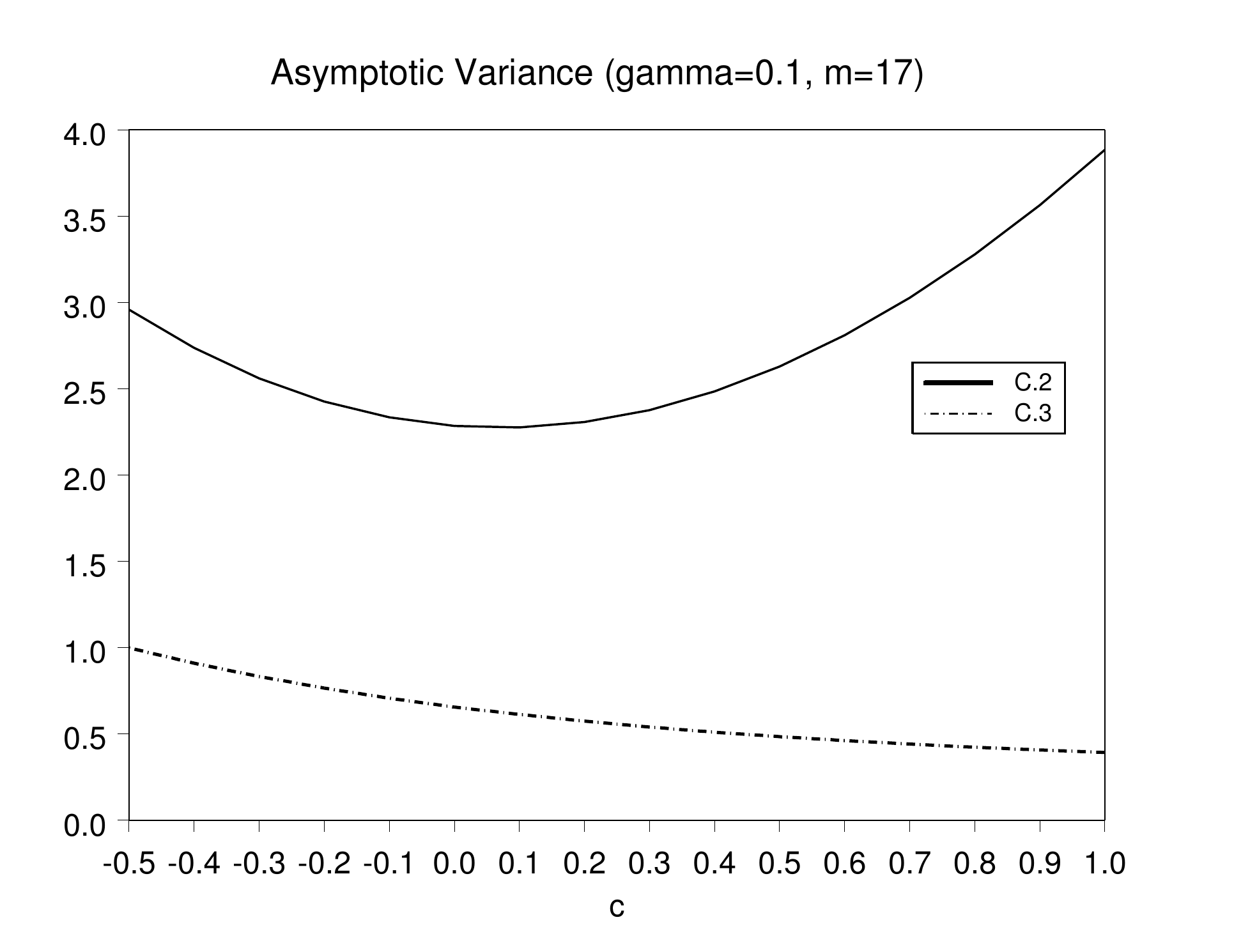}
\label{Fig.Var2}
\end{center}
\end{figure}
\begin{figure}
\caption{\footnotesize Asymptotic variances from Remark \ref{RemAsymptVar} with equal lengths, i.e. $s_j=j/m$, $j=1,2,\ldots,m$, and $m=17$ for several values of $c$.} 
\begin{center}
\includegraphics*[scale=0.42]{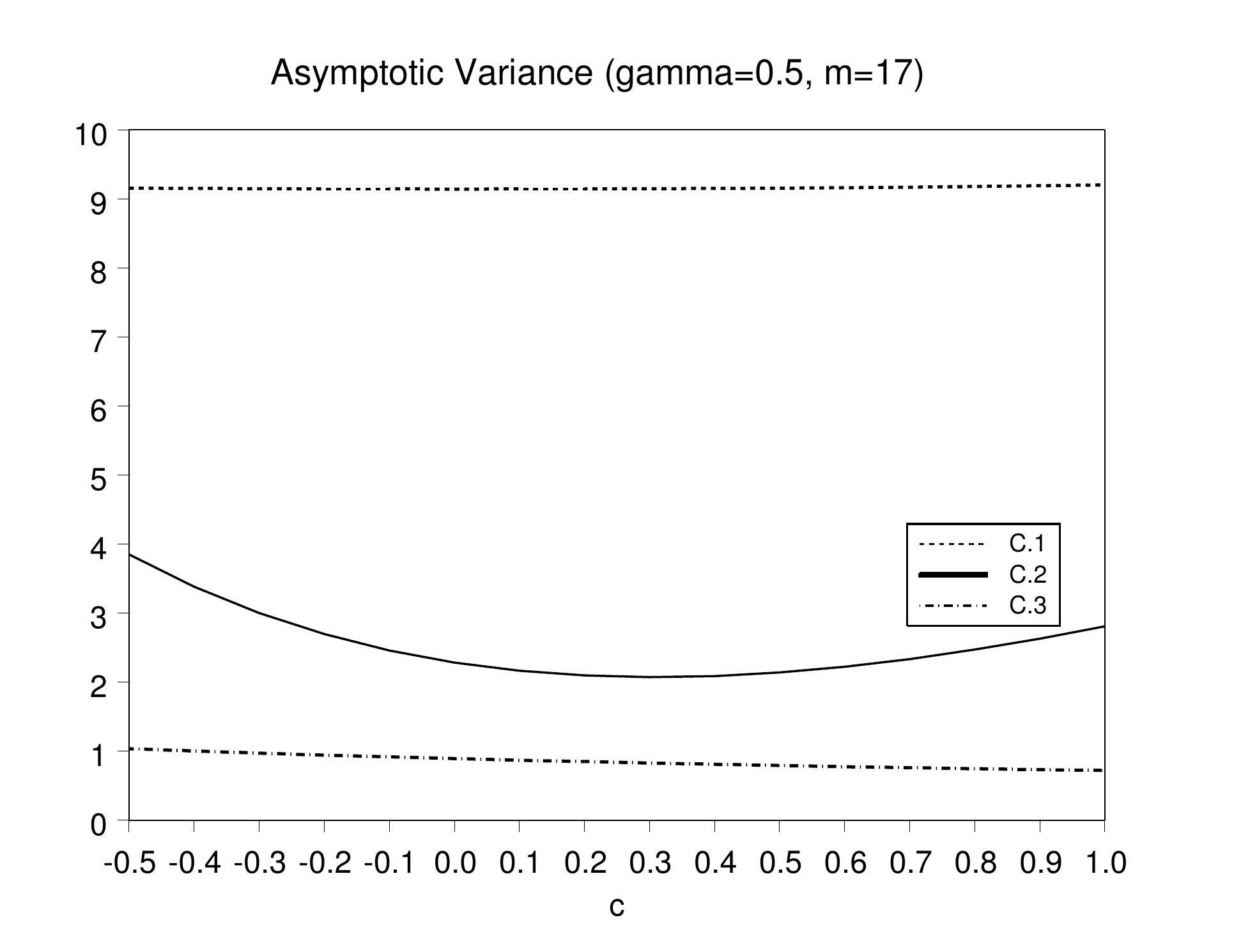}
\label{Fig.Var3}
\end{center}
\end{figure}

\begin{cor}\label{CorTest}
Assume $c=0$. Under the conditions of the Theorem,
\begin{enumerate}
  \item if $k=k_n$ is
such that $\sqrt{k}\,\beta(n/k) \rightarrow 0$, as $n\rightarrow \infty$
then
    \begin{equation}\label{TestStat1}
        Q^{(1)}_{m,n}:= \sumab{j=1}{m}\frac{k}{2}\, \biggl\{ \frac{\log X_{n-k,n}(s_j)-\log
        X_{n-k,n}(0)}{\hat{\gamma}^+_{n,k}}\biggr\}^2 \conv{d}
        \chi^2(m);
    \end{equation}
  \item if $k=k_n$ is
such that $\sqrt{k}\,\alpha_0(n/k) \rightarrow 0$, as $n\rightarrow \infty$
then
  \begin{equation}\label{TestStat2}
    Q^{(2)}_{m,n}:= \sumab{j=1}{m}\frac{k}{2}\, \biggl\{\frac{1}{k} \sumab{i=1}{n}I_{\{X_i(s_j)>X_{n-k,n}(0)\}}-1\biggr\}^2 \conv{d}
    \chi^2(m).
  \end{equation}
\end{enumerate}
Here $\chi^2(m)$ is a standard chi-squared distributed random variable
with $m$ degrees of freedom.
\end{cor}
Corollary \ref{CorTest} gives rise to a testing procedure for
detecting the presence of a trend in the tail of the underlying
distribution functions $F_s$ all lying in the same domain of
attraction. That is, $Q^{(r)}_{m,n}$, $r=1,\,2$, defined above can be used as test
statistics to evaluate the null hypothesis $H_0:\,c=0$ against the
alternative $H_1:\,c\neq 0$. Whence $H_0$ should be
rejected at a significance level $\alpha \in (0,1)$ for any
observed value of $Q^{(r)}_{m,n}$ verifying $Q^{(r)}_{obs}>q_{1-\alpha}(m)$, the
latter being the $(1-\alpha)$-quantile pertaining to the chi-squared
distribution with $m$ degrees of freedom.


\section{Simulations}
\label{SecSims}

Simulations have been carried out for three distributions: (i) the generalized Pareto distribution; (ii) the ordinary Pareto distribution and (iii) the Cauchy distribution. The number of locations is $200$ (i.e.\ $m=200$) with $s_j=j/m$, $j=1,2,\ldots,m$. At each location there are $500$ i.i.d. observations ($n=500$). Then there are $1000$ replications which serve to obtain the means of $\hat{c}^{(r)}$, $r=1,2,3$ as a function of the number $(k+1)$ of tail related observations.

By definition, the finite sample behavior of $\hat{c}^{(1)}$ and $\hat{c}^{(2)}$ is inexorably attached to the estimation of the extreme value index $\gamma$. The parameter $\gamma$, which can be seen as a gauge of tail heaviness of the underlying distribution function $F_s$ is thus an important design factor in the present numerical study. Since $\gamma$ does not depend on $s$, we use combined estimators
\begin{equation}\label{GammaOverall}
	\widehat{\gamma^+}=\hat{\gamma}^+_{n,k}:= \frac{1}{m}\sumab{j=1}{m}\hat{\gamma}^+_{n,k}(s_j) \quad \mbox{ and }\quad \hat{\gamma}=\hat{\gamma}_{n,k}:= \frac{1}{m}\sumab{j=1}{m}\hat{\gamma}_{n,k}(s_j)
\end{equation}
in accordance with Remarks \ref{RemHills} and \ref{RemCovariances}. Estimator $\hat{c}^{(3)}$ is a shift invariant estimator not depending on $\gamma$. But the second order parameter $\rho\leq 0$ also plays a relevant in the performance evaluation of the three estimators since it contributes for the (second order) dominant component of the asymptotic bias. In the present framework, providing a full array of combinations of parameter values $c,\, \gamma$ and $\rho$, for a various number of time points $m$, would make a simulation study quite cumbersome and ultimately of unenthusiastic reading. For the sake of brevity, we have settled with $\gamma=-0.1, 0.1$ and $0.5$ (since $0.1$ is a typical value for rainfall, the application topic).

{\bf (i)} The generalized Pareto distribution (GPD) with distribution function $1-(1+\gamma x)^{-1/\gamma}$ for those $x$ for which $1+\gamma x >0$ has been considered. Relation \eqref{2ndERVU0} holds with limit zero since the left hand-side is zero (exact fit). In this case, the values $c=-0.1$ and $c=0.1$ have been considered.

The starting point is a r.v. $X$ from the GPD distribution. For each location $s_j$ we then take \linebreak $X(s_j)\, \id \, e^{cs_j\gamma}X + \bigl(e^{cs_j\gamma}-1\bigr)/\gamma$. That way the relations \eqref{TrendPropag} and  \eqref{2ndUs2U0Dif} hold.

Figure \ref{Fig.Sim1} displays the average (over $1000$ replications) values of the estimators $\hat{c}^{(1)}$ (only for positive $\gamma$), $\hat{c}^{(2)}$ and $\hat{c}^{(3)}$ as functions of the number $k$ of upper order statistics above $X_{n-k,n}(s_j)$ for all $s_j$. As usual in graphs of this type there is a stretch of the graph that is more or less straight; the idea is that in that part both the variance and bias are not too high.
We note that if, by the one hand, only a very tiny sample fraction $(k/n)$ is selected then huge variance arises; on the other hand, if we get further into the original sample by selecting a very large number $k$ of upper order statistics then bias increases. This sort of bias/variance trade-off is a common requirement in extreme value statistics.
\begin{figure}
\caption{\footnotesize Estimated means of $\hat{c}^{(r)}$, $r=1,2,3$, plotted against the same number $k$ of top observations on each location $s_j=j/m$, $j=1,2,\ldots,m$, with underlying Generalized Pareto distribution, in either
case of true value $c=-0.1$ or $c=0.1$.} 
\begin{center}
\includegraphics*[scale=0.42]{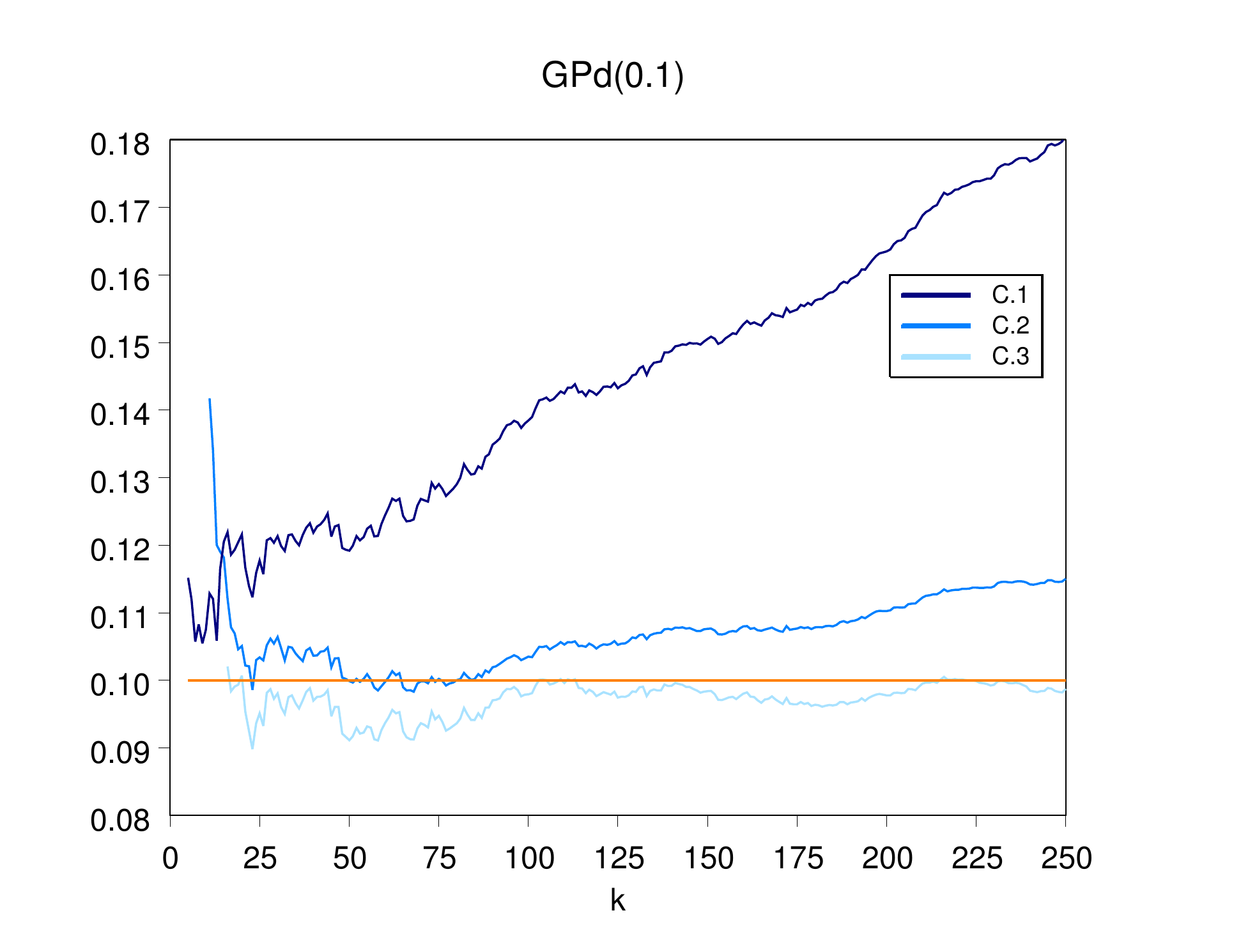}\hfill
\includegraphics*[scale=0.42]{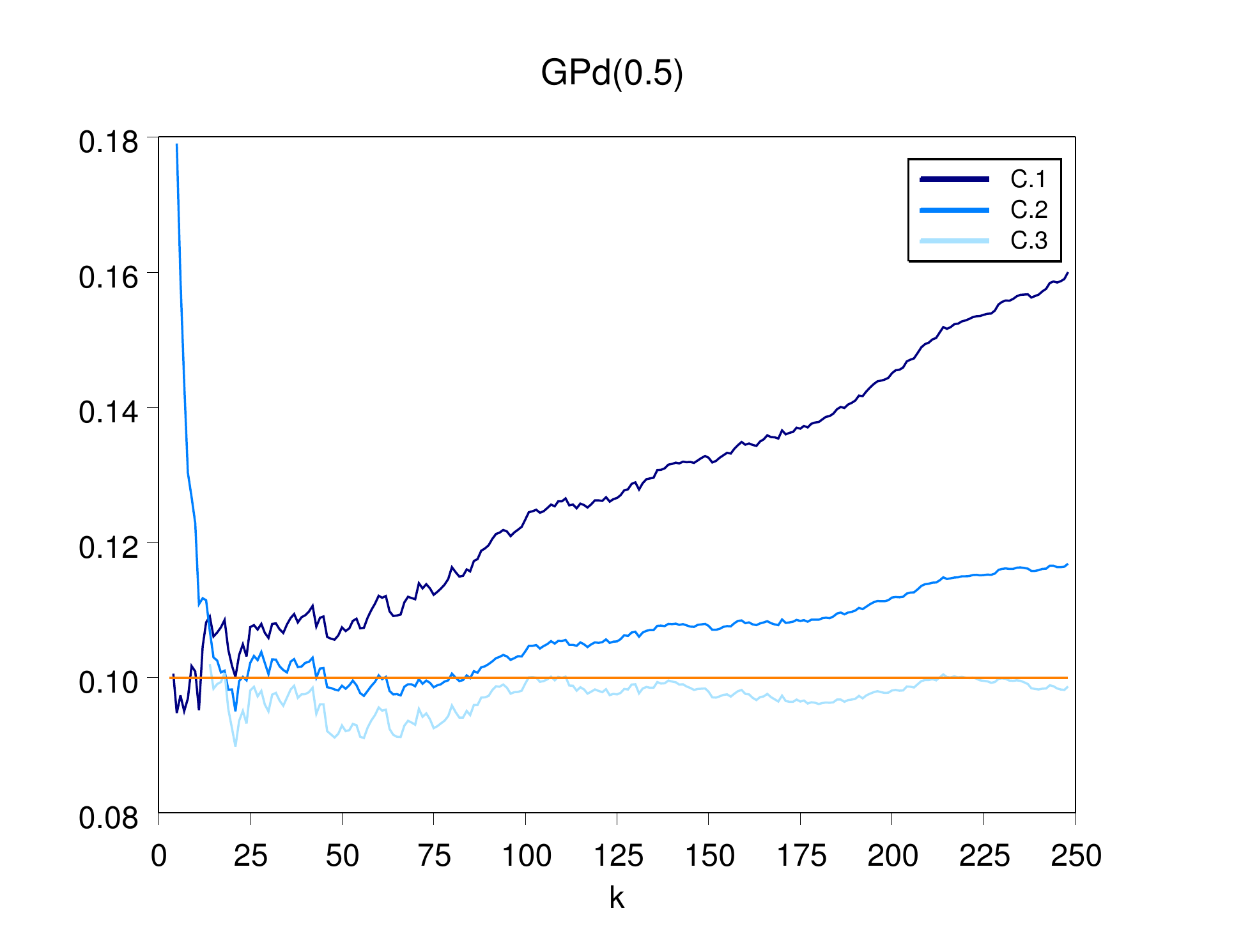}\hfill
\includegraphics*[scale=0.42]{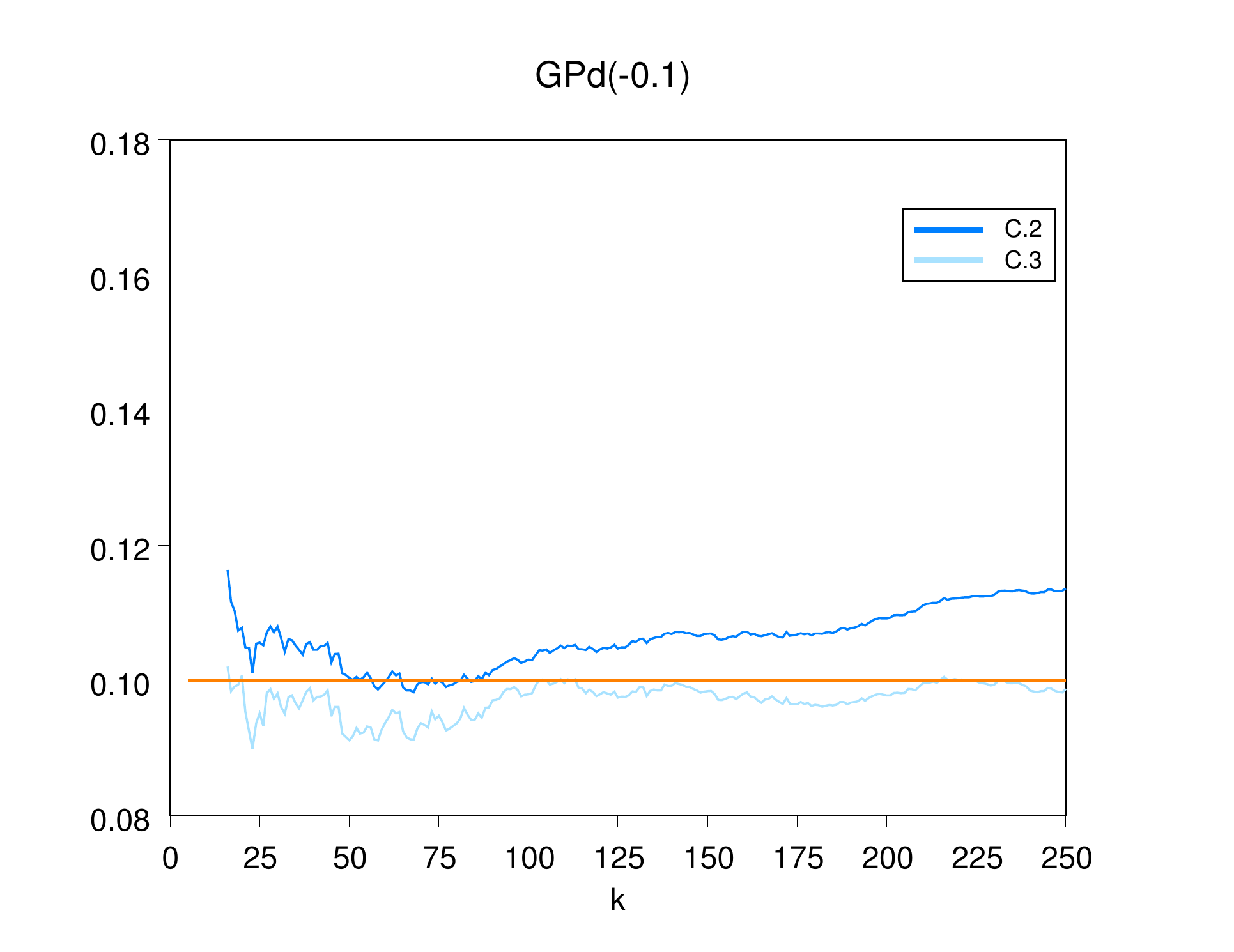}\hspace{2.0cm}\hfill
\includegraphics*[scale=0.42]{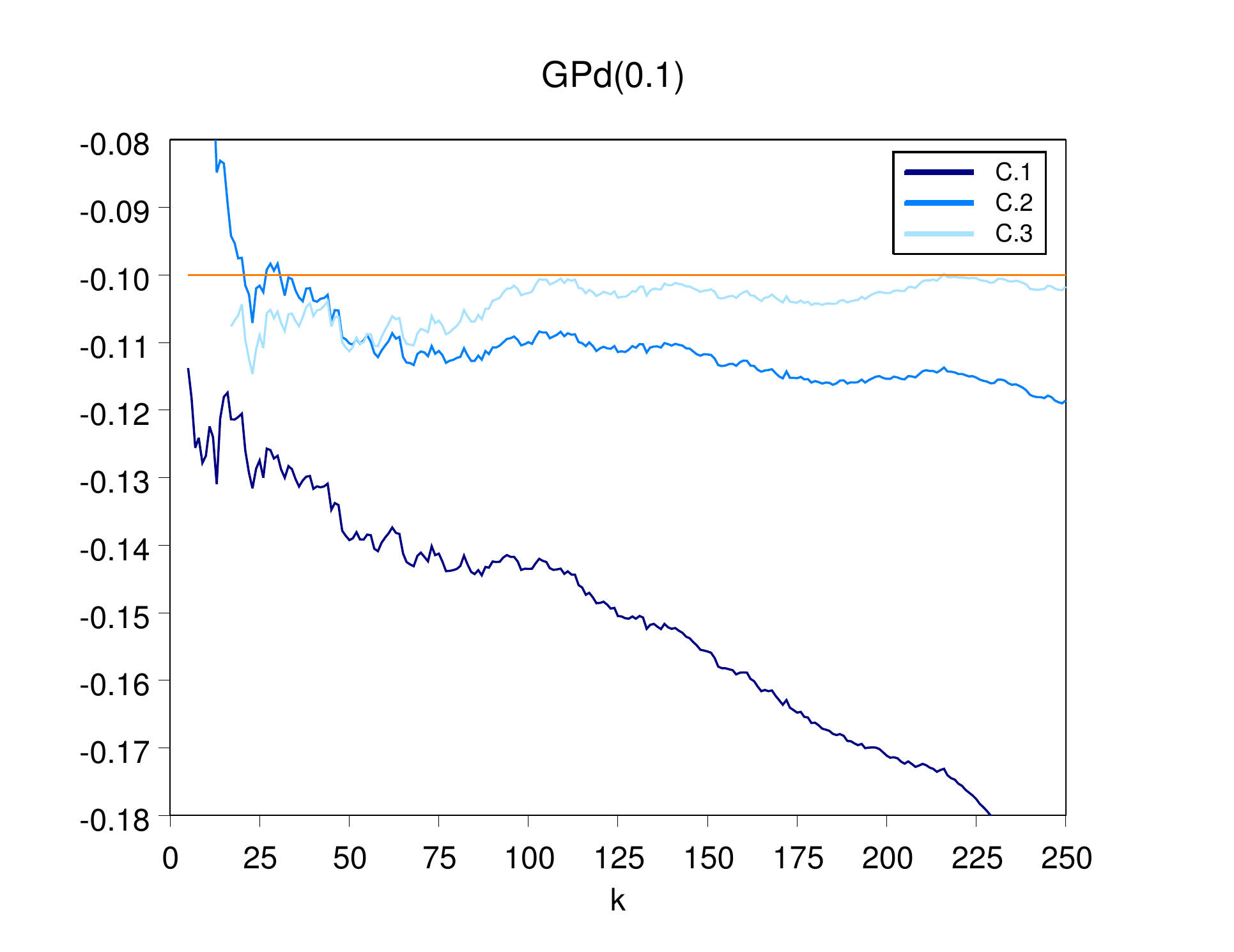}\hfill
\includegraphics*[scale=0.42]{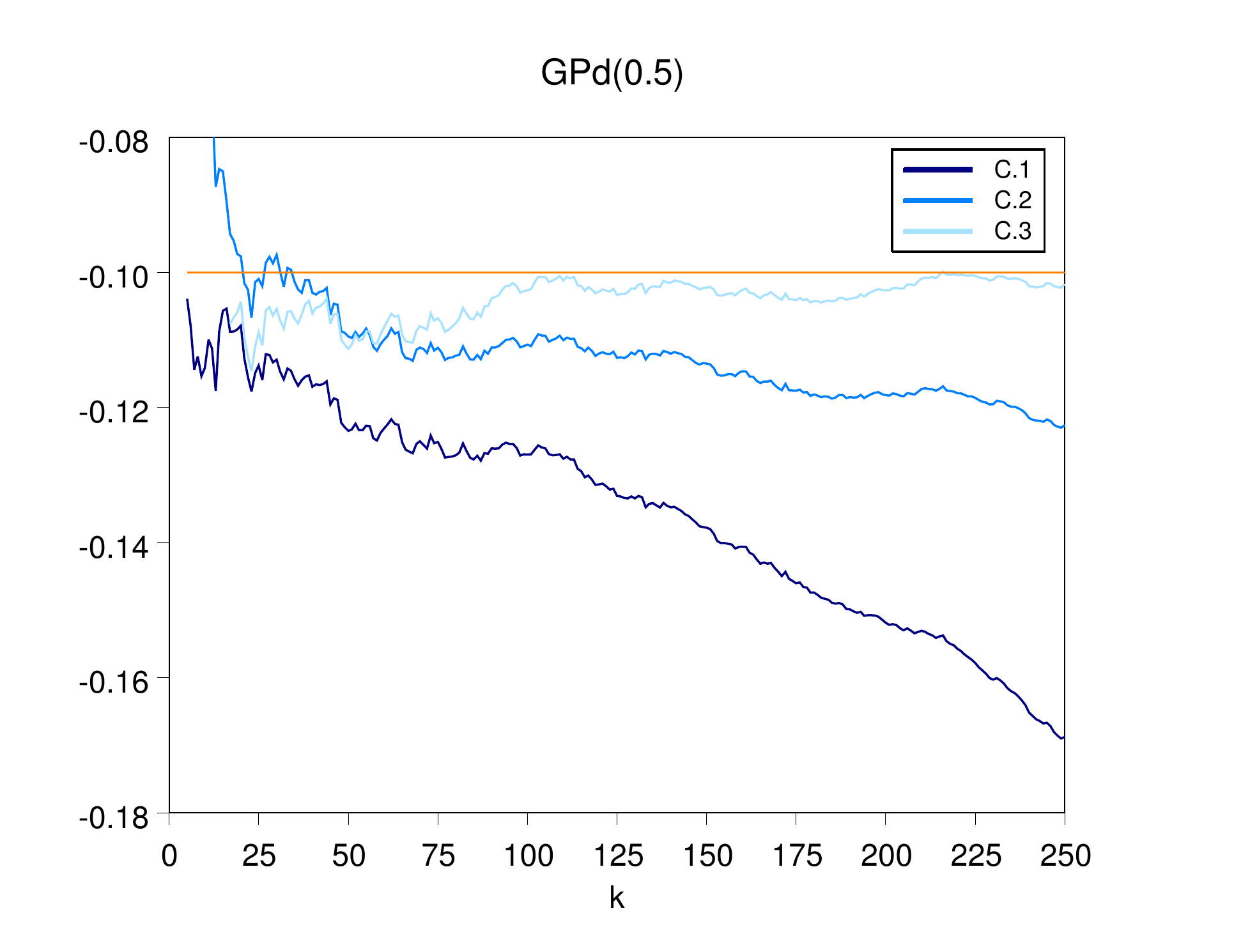}
\label{Fig.Sim1}
\end{center}
\end{figure}

In Figure  \ref{Fig.Sim1} the extreme value index $\gamma$ and scale $a_0$ have been estimated by the estimators prescribed in Remarks \ref{RemHills} and \ref{RemCovariances}. Note that in these and later graphs a realistic choice of the number of upper order statistics seems to be in the range between $k=25$ and $k= 50$. This is the main reason why we say that the second estimator, $\hat{c}^{(2)}$, has the best performance among the three estimators, despite the estimator $\hat{c}^{(3)}$ begins to return estimates in a close vicinity of the true value $c$ as $k$ is approaching the whole sample size $n$. The latter occurs because the estimator $\hat{c}^{(3)}$ is the most direct empirical counterpart of relation \eqref{TrendPropag} and the underlying Generalized Pareto distribution is the precise limit distribution in condition \eqref{DOAalt} regarding tail distribution.

Overall, not the estimator $\hat{c}^{(1)}$, but
the other two estimators of $c$ seem to return stable trajectories
in a close vicinity of the actual $c$-value quite often. Because the estimator $\hat{c}^{(1)}$ is
subject to $\gamma$ positive,  one should expect that $\hat{c}^{(1)}$ is more
prone to bias and/or variance inflations due to the presence of a true $\gamma$ near
zero. This is verified by the simulations: the panels on the left in Figure
\ref{Fig.Sim1} show that setting $\gamma=0.1$ results in considerable bias displayed by $\hat{c}^{(1)}$. Estimators $\hat{c}^{(2)}$ and $\hat{c}^{(3)}$ come out with the best performance for intermediate values of $k$.

Figure  \ref{Fig.Sim5} gives a comparison with the maximum likelihood estimators, valid for $\gamma\geq -1$. The plot on the left panel of Figure \ref{Fig.Sim5} epitomizes the behavior of $\hat{c}^{(2)}$, either with the maximum likelihood estimator or with the moment estimator, because in other simulations we have conducted the moment estimator has been recognized so as to instill less bias in $\hat{c}^{(2)}$ while pertaining to moderate values of $k$, which are the most adequate in the context of extreme value theory. Moreover, the fact that the maximum likelihood estimates for shape and scale, $\gamma$ and $a_0$, have to be numerically obtained can pose a practical difficulty to our trend estimation procedure. The convergence of appropriate numerical procedures may be rather poor when the true value of $\gamma$ is close to zero. For $\gamma=0.1$, the number of times the algorithm has converged thus returning feasible estimates of $\gamma$ and $a_0$ is depicted on the right hand-side of Figure \ref{Fig.Sim5}.

\begin{figure}
\caption{\footnotesize Estimated means of $\hat{c}^{(2)}$ either with Moment estimator and ML estimator for the Generalized Pareto distribution with $\gamma=0.1$ and $c=0.1$. The number of samples amongst the 1000 replicates that have produced valid ML-estimates is presented on the right hand-side.} 
\begin{center}
\includegraphics*[scale=0.4]{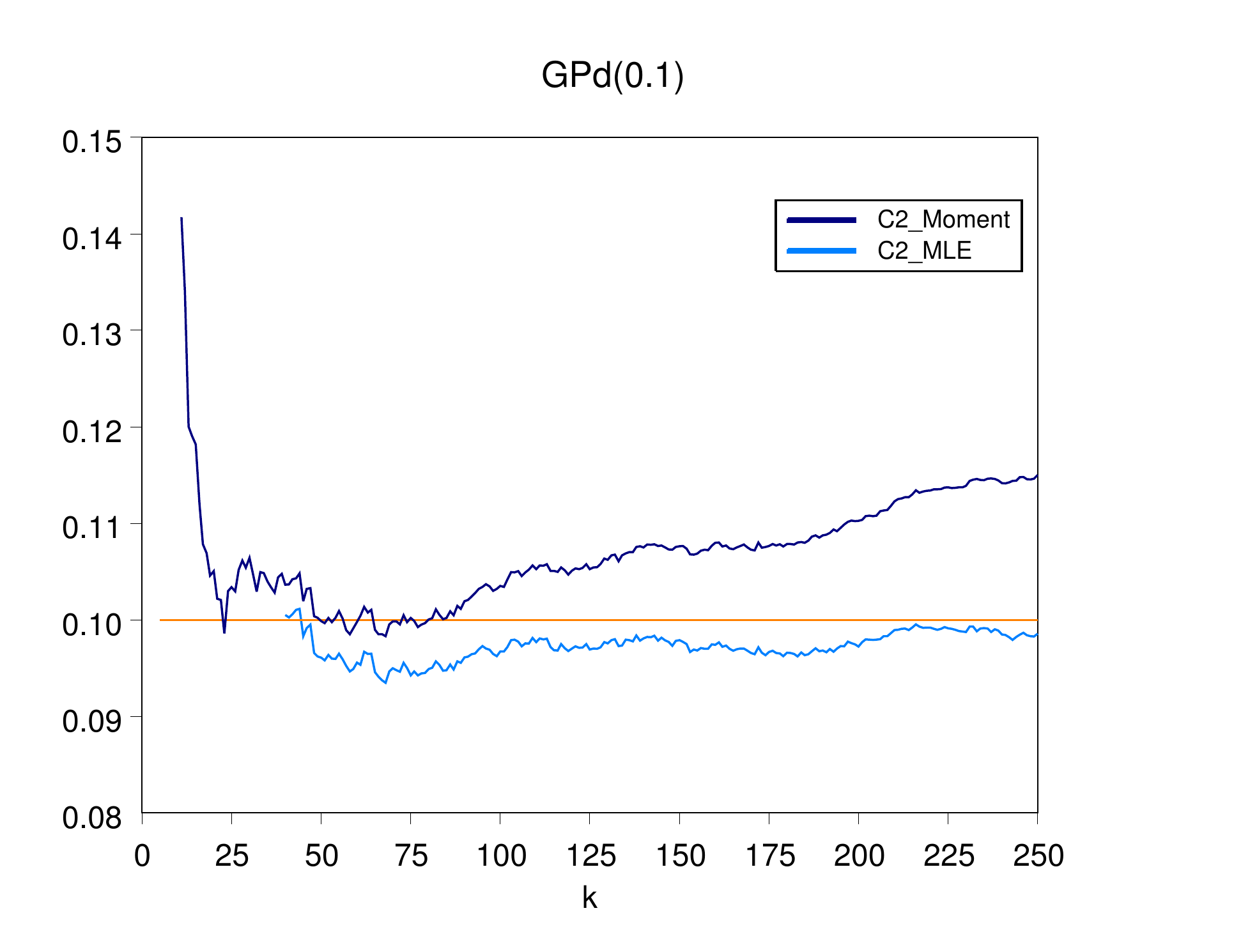}\hfill
\includegraphics*[scale=0.4]{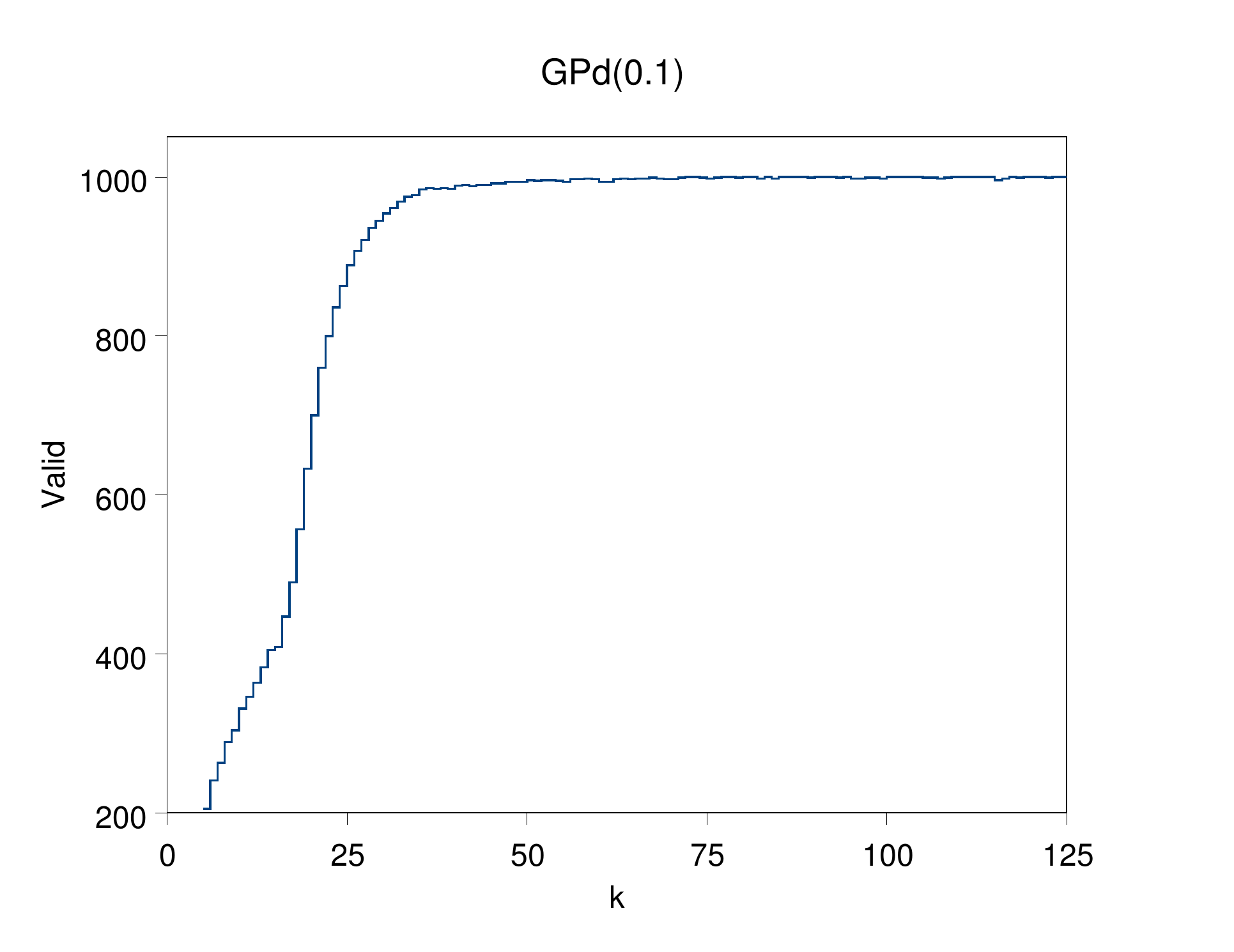}
\label{Fig.Sim5}
\end{center}
\end{figure}

{\bf (ii)} For the ordinary Pareto distribution with distribution function $1-x^{-1/\gamma}$, $x\geq 1$, $\gamma>0$, we simulate the trend by taking $X(s_j)\id \,e^{cs_j\gamma}X$ where $X$ follows the Pareto distribution. Again we have an exact fit in view of condition \eqref{Us2U0} for $\gamma>0$.

{\bf (iii)} For the Cauchy distribution again the trend is simulated by taking 
$X(s_j)\id \,e^{cs_j}X$ (since $\gamma$ is $1$). Relations \eqref{2ndERVU0} and  \eqref{2ndUs2U0Dif} hold. In this case $|\alpha_0|$ is a regularly varying function with index $-2$ (entails a rather fast convergence).

Figure \ref{Fig.Sim2} displays the simulation results for the Pareto  and Cauchy distributions with $c=0.1$. Again, the extreme value index $\gamma$ and scale $a_0$ have been estimated by the moment estimator. In the particular case of the Pareto distribution, the estimates process of $\hat{c}^{(3)}$ is a repeat of the previous exact model described in (i), due to its invariance towards a shift in location and/or changes to in the scale of the observed data. Although $\hat{c}^{(1)}$ is only valid for $\gamma >0$, it remains a matching competitor against $\hat{c}^{(2)}$ and $\hat{c}^{(3)}$ under the three parent distribution functions considered in Figure \ref{Fig.Sim2}. The estimators $\hat{c}^{(r)}$, $r=1,2$ and $3$, are quite close to the real value if one chooses $k$ close to 30; the graph in that area is relatively flat.
\begin{figure}
\caption{\footnotesize Estimated means of $\hat{c}^{(r)}$, $r=1,2,3$, plotted
against the same number $k$ of top observations on each location $s_j=j/m$, $j=1,2,\ldots,m$, in
case of true value $c=0.1$, with Pareto and Cauchy parent distributions.}
\begin{center}
\includegraphics*[scale=0.4]{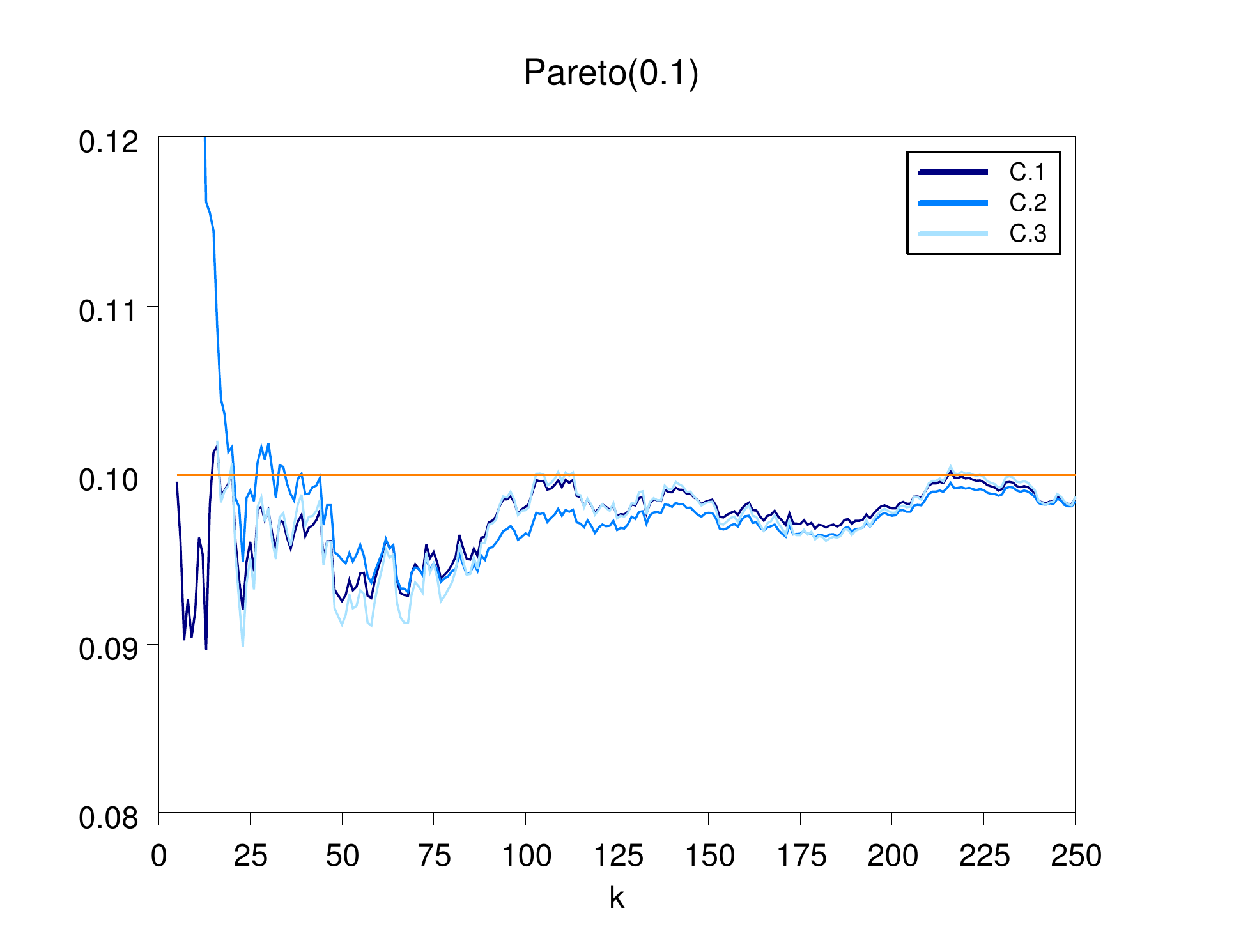}\hfill
\includegraphics*[scale=0.4]{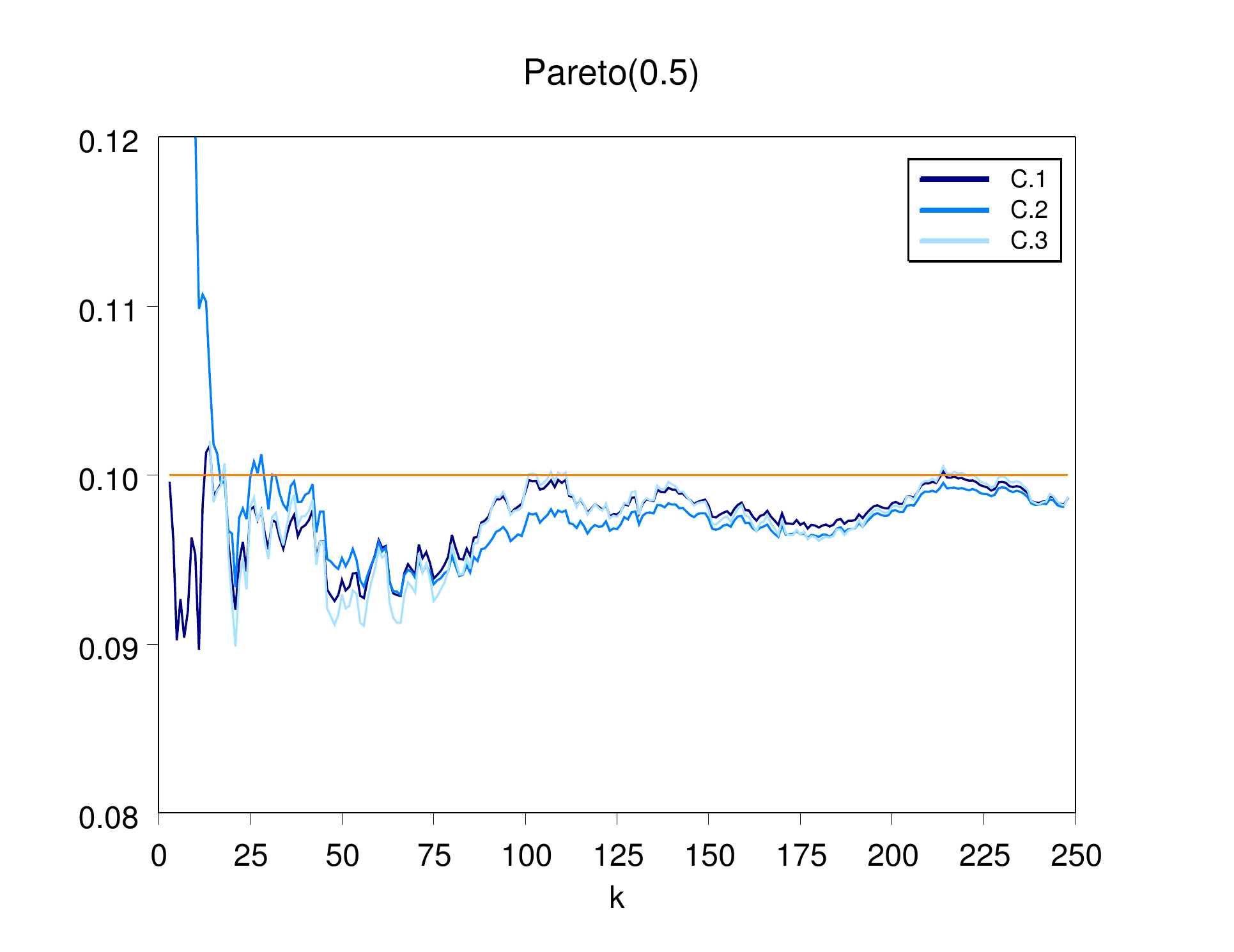}\hfill
\includegraphics*[scale=0.4]{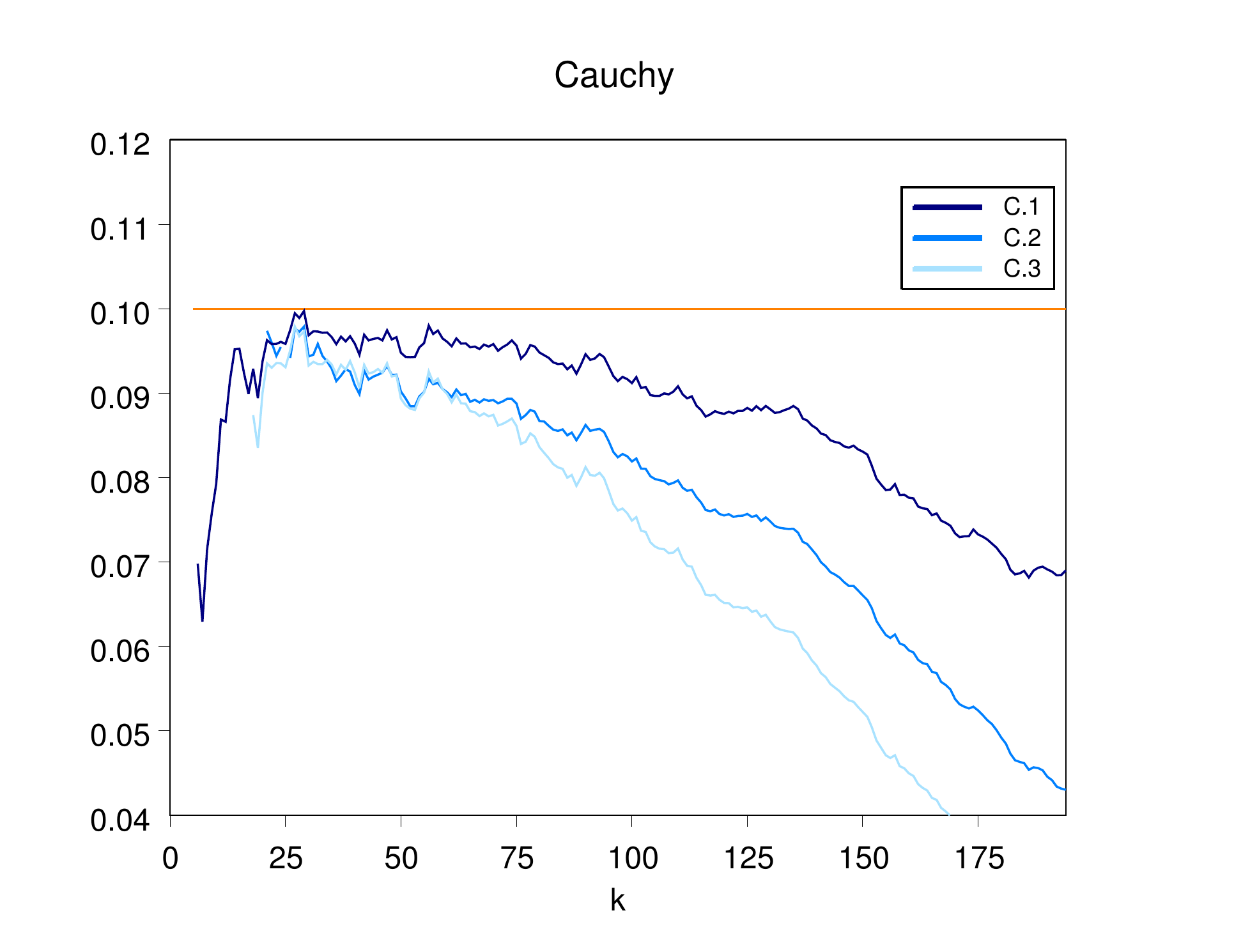}
 \label{Fig.Sim2}
\end{center}
\end{figure}

We emphasize that a much more comprehensive simulation study has been performed. From the described simulations and the other ones we conclude that the estimators perform reasonably well. The main conclusion is that estimators $\hat{c}^{(2)}$ and $\hat{c}^{(3)}$ seem to behave better than $\hat{c}^{(1)}$. In the next (application) section we shall adopt estimator $\hat{c}^{(2)}$. Nevertheless, in general applications, a possible focus standing on $\hat{c}^{(3)}$ alone could be justified by its implicit detachment from the extreme value index $\gamma$ thus making  $\hat{c}^{(3)}$ more versatile in the estimation of $c$ for a wider range of underlying models pertaining to diverse values of $\gamma$.

\clearpage

\section{Data Analysis}
\label{SecData}

\begin{figure}
\caption{\footnotesize Selected gauging stations in The Netherlands
(\emph{left}); selected gauging stations in Germany and Dutch
stations near the borderline (\emph{right}).}
\begin{center}
\includegraphics*[scale=0.45]{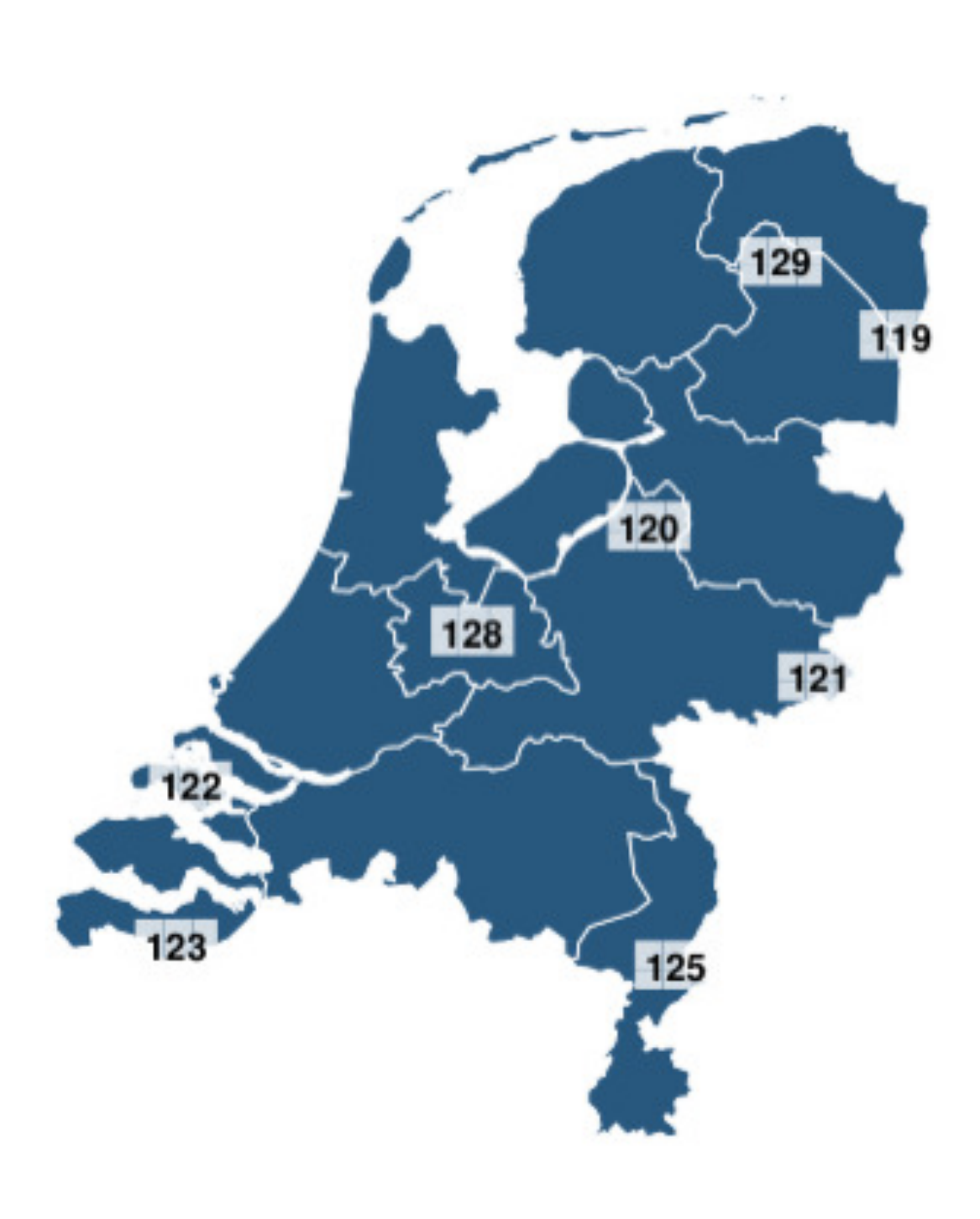}
\includegraphics*[scale=0.45]{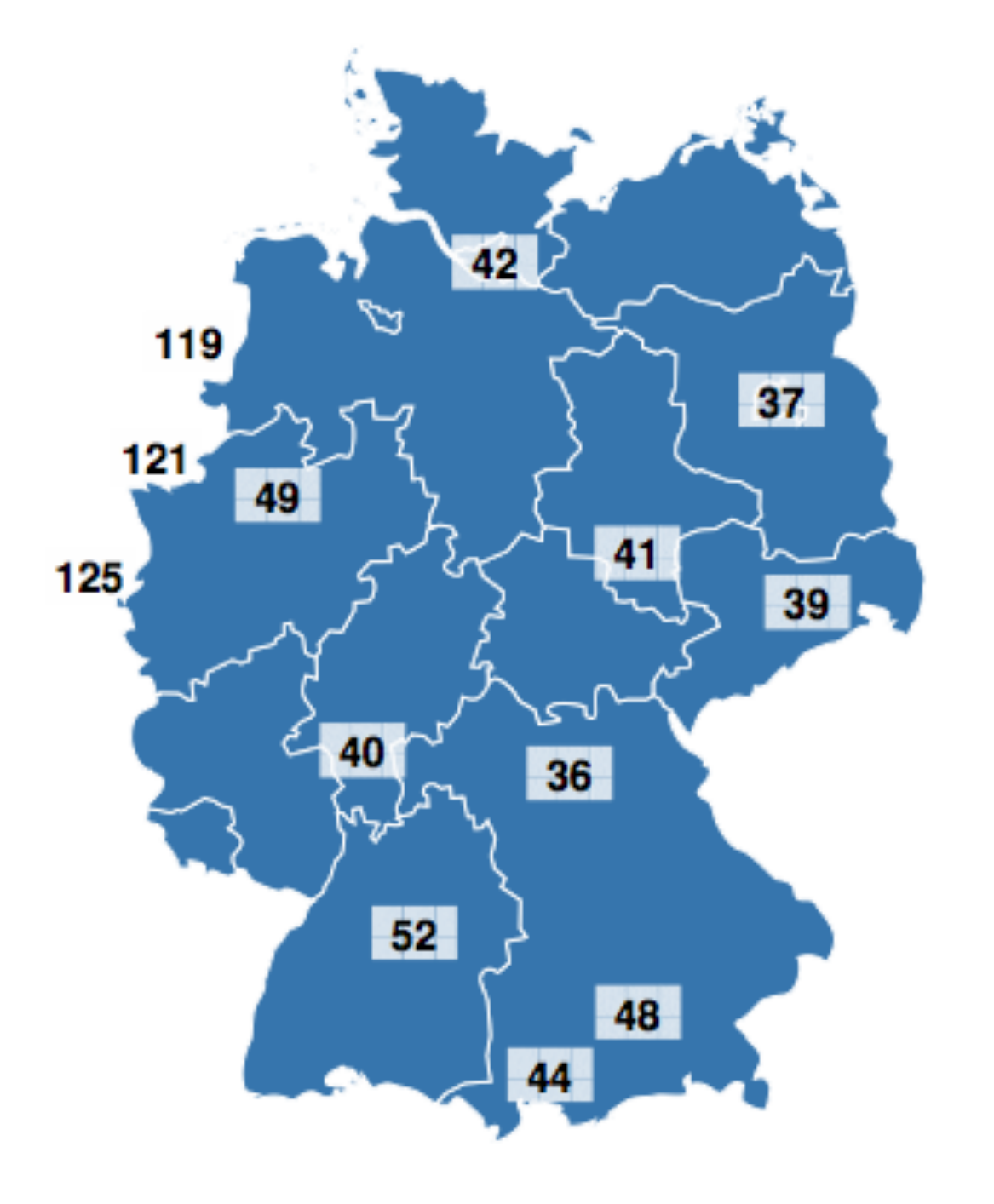}
 \label{Fig.STNDandNL}
\end{center}
\end{figure}

As an application of the tail trend assessment 
methodology developed
in this paper, we will look at daily rainfall totals collected in
$18$ gauging stations across Germany and The Netherlands, comprising
latitude 47N-53N and longitude 5E-13E. The geographic location of the stations is displayed in Figure \ref{Fig.STNDandNL}. Rainfall data are from the European Climate
Assessment and Dataset. We note that different
stations suffer from different coverage in time in the sense that
not all stations have started regular recording of data at the same
year. Moreover there are some stations with missing values. All of them however
meet the basic criterium for completeness that there is less than 10
days missing per year which leaves us with  $90$ years of
complete data from 1918 up to 2007.

\subsection{Trend estimation in the extreme relative risk model}
Figure \ref{Fig.Series} displays yearly maxima plots for
several stations on the basis of
available data. We get a mixed picture. In \emph{STN41-Halle}, for instance, precipitation does
not seem to be as severe now as in the first half of the $20$th century
anymore, whereas \emph{STN39-Dresden} shows increasingly annual maxima
with the largest peak of $158\, mm$ of rain, spot on the
catastrophic event of 12 August 2002.

In what follows we shall assume that as long as there is at least one day in between, there is not much dependence in the amount of rainfall on two different days. For each station we select first the highest observation. Then we remove the observations on the day before and after. Next we select the highest observation from the remaining data, etc. This goes on until we have selected 70 days or the threshold of $1\, mm$ is reached. That way we get a sequence of higher order statistics from i.i.d. data. Table \ref{Table} displays the number of rain days (i.e.\ with at least $1\, mm$ of rain) per station.

\begin{table}\centering
\caption{\footnotesize Total number of rain days in the period from 1918 to 2007 for each selected station.}\label{Table}
\begin{footnotesize}
\begin{tabular}{|c|c|c|c|c|c|}
  \hline
  STN & Name & Country & Lat. & Lon. & Rain days \\
    \hline
  36 & Bamberg          & D & 48$^\circ$49'N & 11$^\circ$33'E & 6276 \\
  37 & Berlin           & D & 52$^\circ$31'N & 13$^\circ$20'E & 6270 \\
  39 & Dresden          & D & 51$^\circ$31'N & 13$^\circ$44'E & 6293 \\
  40 & Frankfurt        & D & 50$^\circ$6'N  & 08$^\circ$40'E & 6069 \\
  41 & Halle            & D & 51$^\circ$28'N & 11$^\circ$57'E & 6230 \\
  42 & Hamburg          & D & 53$^\circ$33'N & 09$^\circ$59'E & 6125 \\
  44 & Hohenpeissenberg & D & 47$^\circ$48'N & 10$^\circ$59'E & 6230 \\
  48 & M\"{u}nchen      & D & 48$^\circ$08'N & 11$^\circ$34'E & 6153 \\
  49 & M\"{u}nster      & D & 52$^\circ$59'N & 07$^\circ$41'E  & 5904 \\
  52 & Stuttgart        & D & 48$^\circ$46'N & 09$^\circ$46'E & 6289 \\
  119 & Ter Apel        & NL & 52$^\circ$53'N & 07$^\circ$04'E & 6069 \\
  120 & Heerde          & NL & 52$^\circ$24'N & 06$^\circ$03'E & 6300 \\
  121 & Winterswijk     & NL & 51$^\circ$59'N & 06$^\circ$42'E & 6298 \\
  122 & Kerkwerve       & NL & 51$^\circ$40'N & 03$^\circ$51'E & 6276 \\
  123 & Westdorpe       & NL & 51$^\circ$13'N & 03$^\circ$52'E & 6300 \\
  125 & Roermond        & NL & 51$^\circ$11'N & 05$^\circ$58'E & 6300 \\
  128 & De Bilt         & NL & 51$^\circ$06'N & 05$^\circ$11'E & 6299 \\
  129 & Eelde           & NL & 53$^\circ$08'N & 06$^\circ$35'E & 6300\\
\hline
\end{tabular}
\end{footnotesize}
\end{table}

\begin{figure}
\caption{\footnotesize Some yearly maxima of daily rainfall.}
\begin{center}
\includegraphics*[scale=0.37]{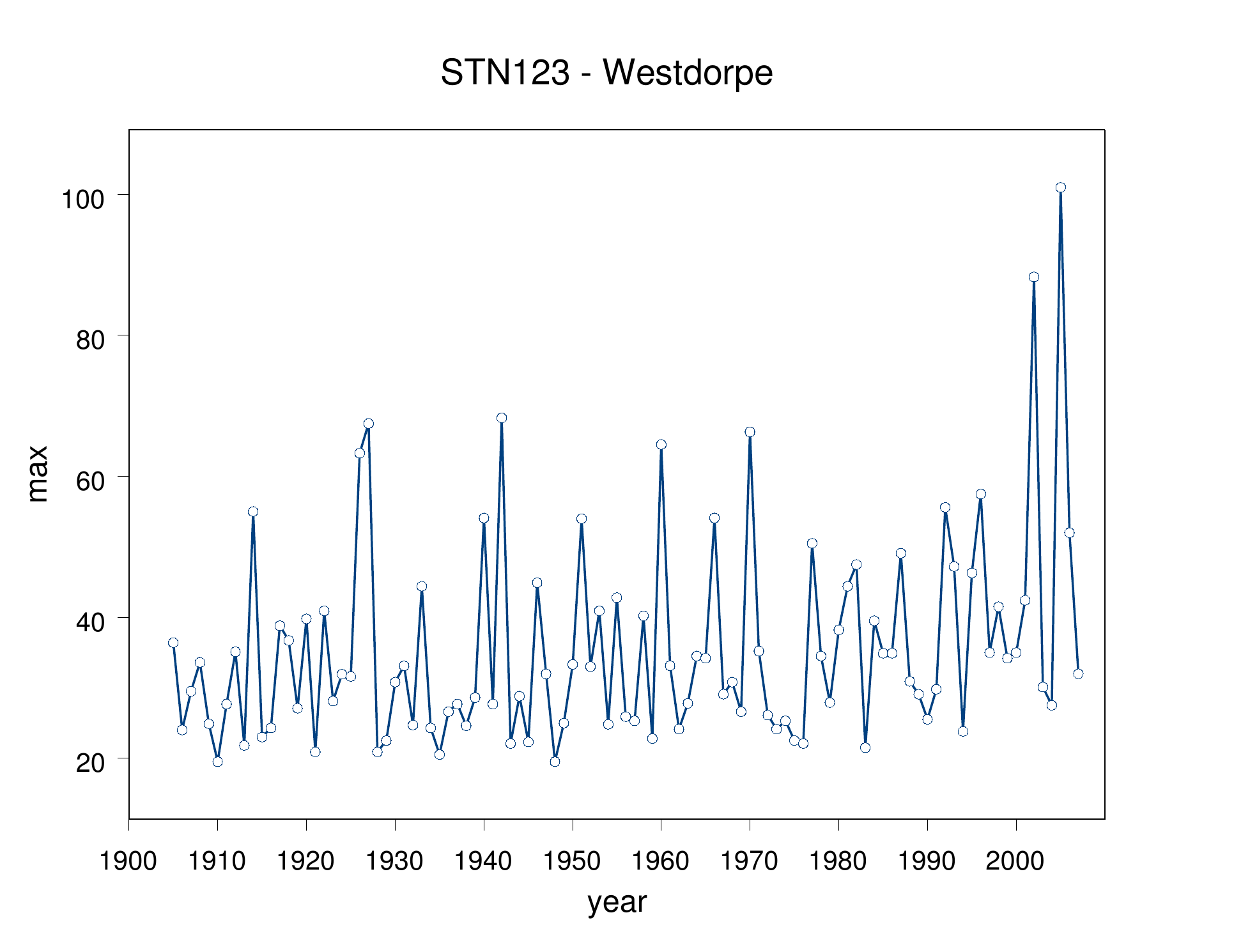}
\includegraphics*[scale=0.37]{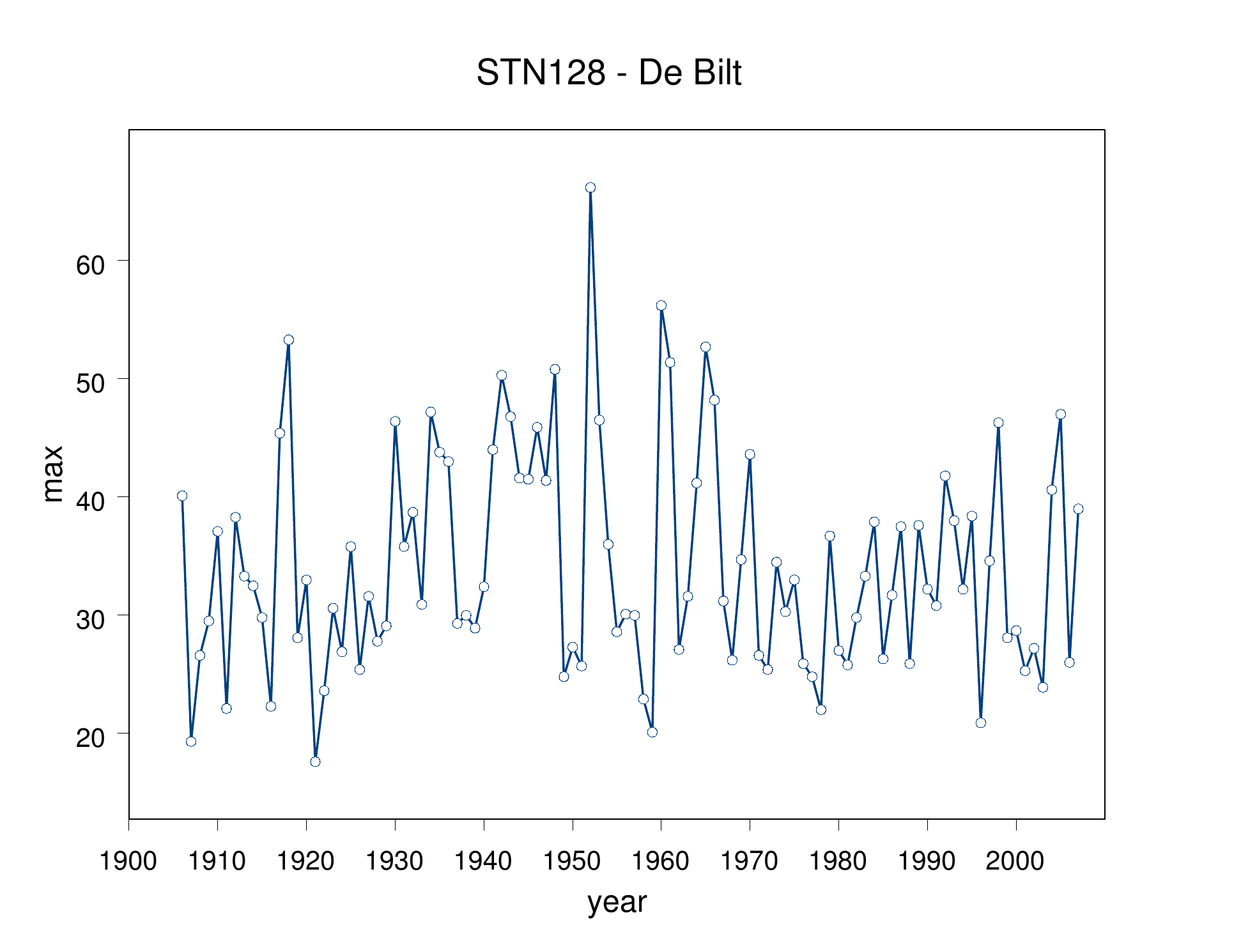}
\includegraphics*[scale=0.37]{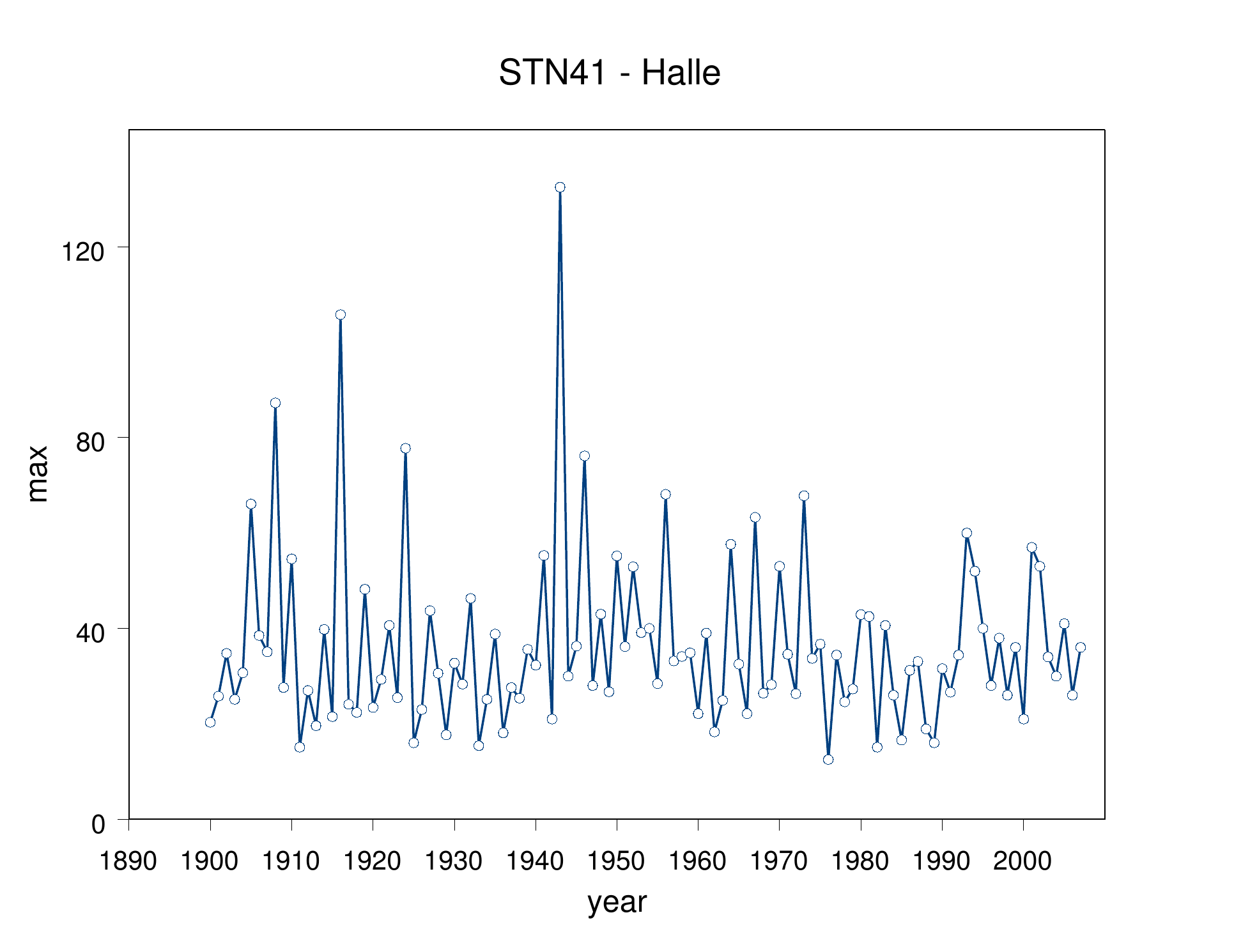}
\includegraphics*[scale=0.37]{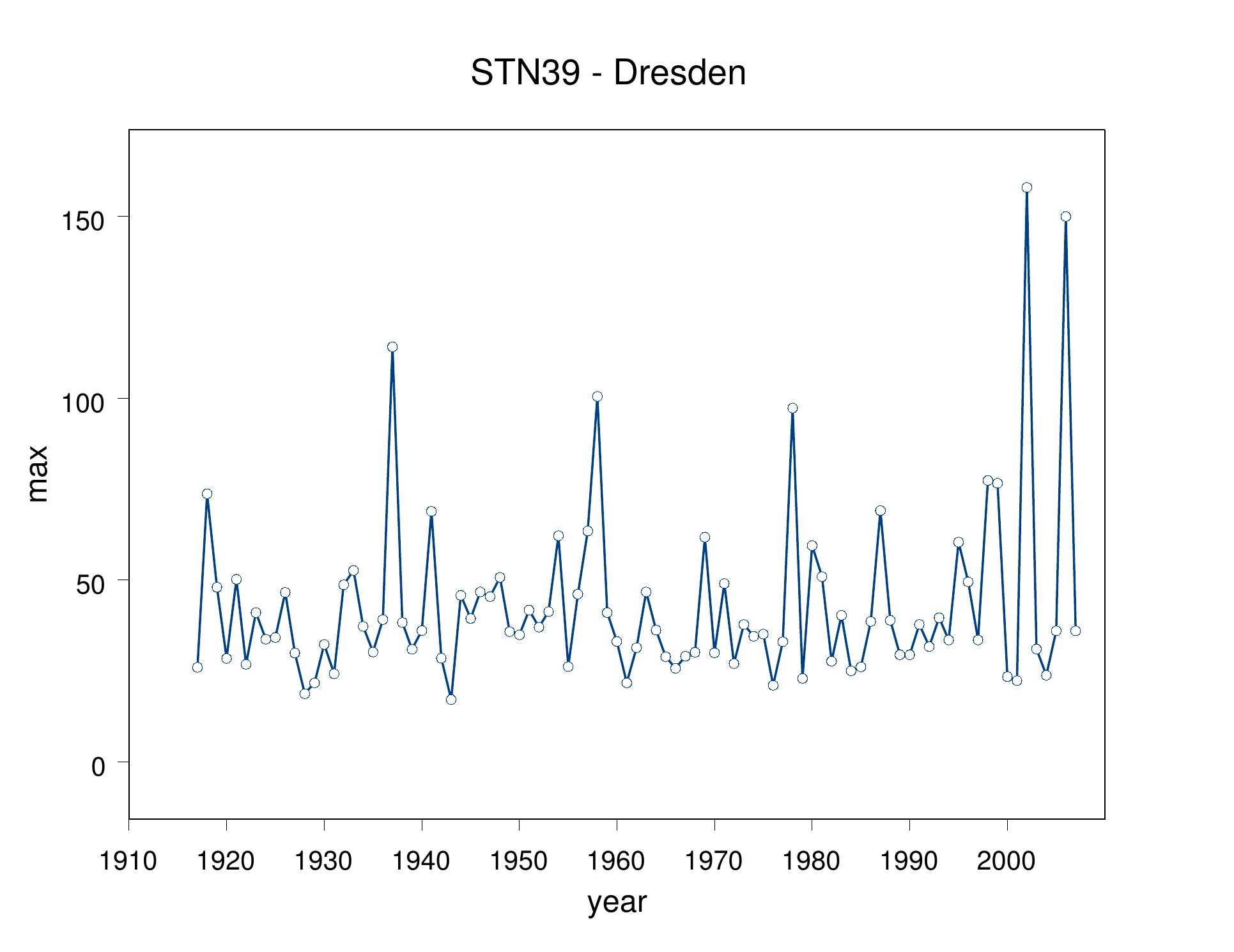}
\label{Fig.Series}
\end{center}
\end{figure}

Since we are not looking for a spatial trend now we shall make a study
of the highest daily rainfall amounts in the 90 year period for each
station separately. At each location, for $\hat{\gamma}^+$ (in connection with $\hat{c}^{(1)}$) we use Hill's estimator and for $\hat{\gamma}$ and $\hat{a}_0$ (for $\hat{c}^{(2)}$) we use the moment estimator (cf. sections 3.5 and 4.2 of
\citet{deHaanFerreira:06}) 

The point estimation of the extreme value index $\gamma$
and trend estimation is conducted with the same  number of
upper order statistics $k$, just as prescribed in each definition of
$\hat{c}^{(r)}$, $r=1,2,3$, introduced in \eqref{Est1},
\eqref{Est2}, and \eqref{Est3}, respectively. 

In order to have enough tail related rain measurements per time point
we found reasonable to take consecutive intervals of 5 years over the
90-year span. The disjoint intervals serve as our time points indexed by $j=0,1,2,\,\ldots,17=m$. 

For the purpose of data analysis, the gauging stations have been divided  into two
groups, determined by their alignments in the general climate
characteristics (according to the K\"{o}ppen–-Geiger climate classification system, see e.g. \citet{Kotteketal:06}). Each selected station in Germany was classified as either
\emph{humid oceanic} or \emph{humid continental}. All stations
across The Netherlands are classified as \emph{humid oceanic}.
Although we are not looking for spatial coherence we hope to benefit
from this information to get a more systematic presentation of our
results.

Estimates of $\gamma$ in a vicinity of $0.1$ often emerge in
connection with the extremal behavior of distributions underlying
rainfall records (see e.g. \citet{Buishandetal:08}, p.239; also \citet{Mannshardtetal:10}, p.492). This seems to hold for most of the considered stations although there is a lot of variation.

Figure \ref{Fig.Stations}  includes sample paths of
the three proposals for estimating the tail trend parameter $c\in \real$ for some typical gauging stations. As
already discussed, we shall handle estimation of $c$ by screening
 plots as in Figure \ref{Fig.Stations} for
plateaus of stability in the early part of the graphs pertaining to
the smoother estimator $\hat{c}^{(2)}$, coherent with the path patterns of $\hat{c}^{(1)}$ and
$\hat{c}^{(3)}$ whenever possible.

\begin{figure}
\caption{\footnotesize Sample path of the overall moment estimator
for $\gamma$ and sample trajectories of $\hat{c}^{(r)}$, $r=1,2,3$,
as a function of the same number $k$ of top observations for each
5-year interval between 1918 and 2007 for several stations.}
\begin{center}
\includegraphics*[scale=0.4]{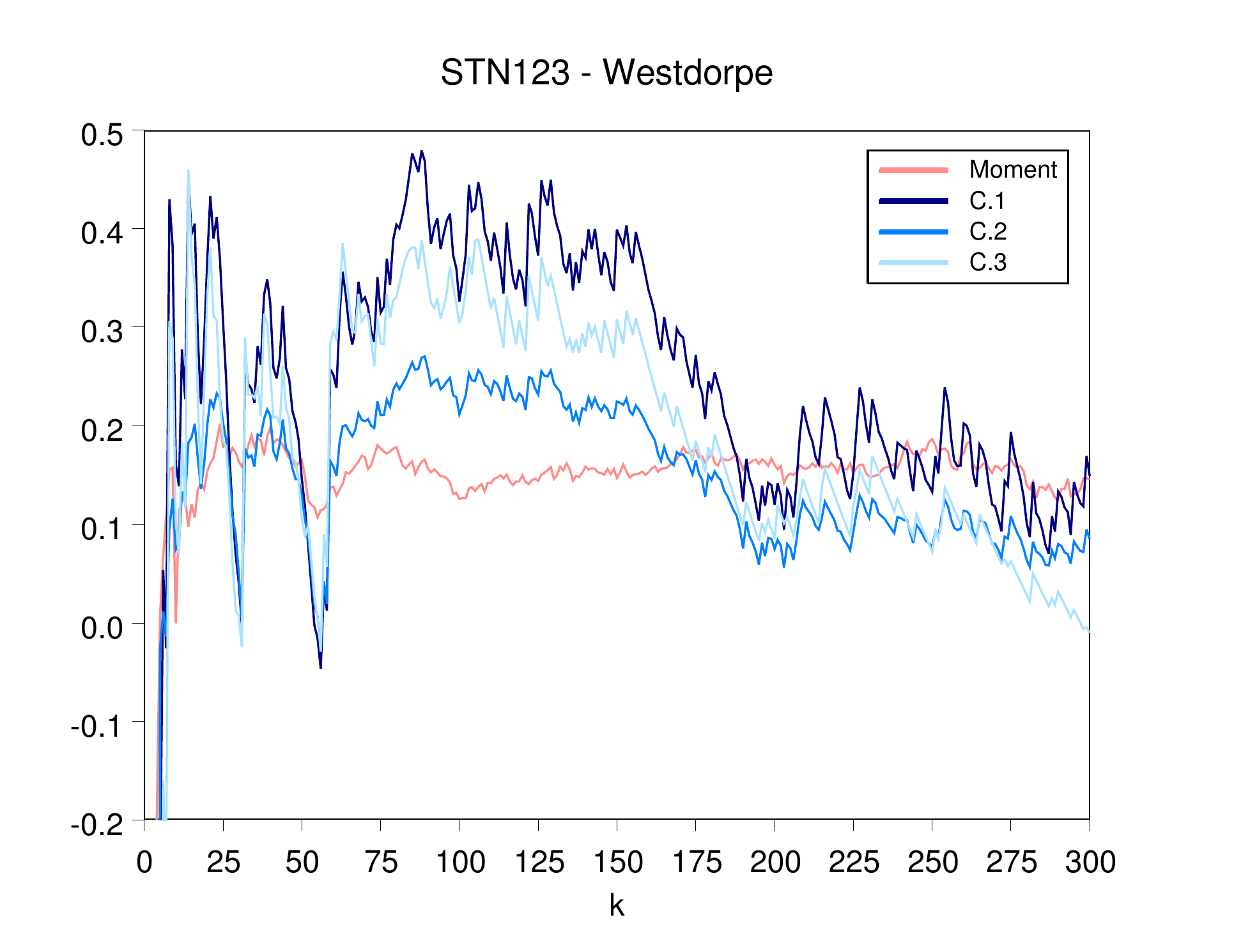}
\includegraphics*[scale=0.4]{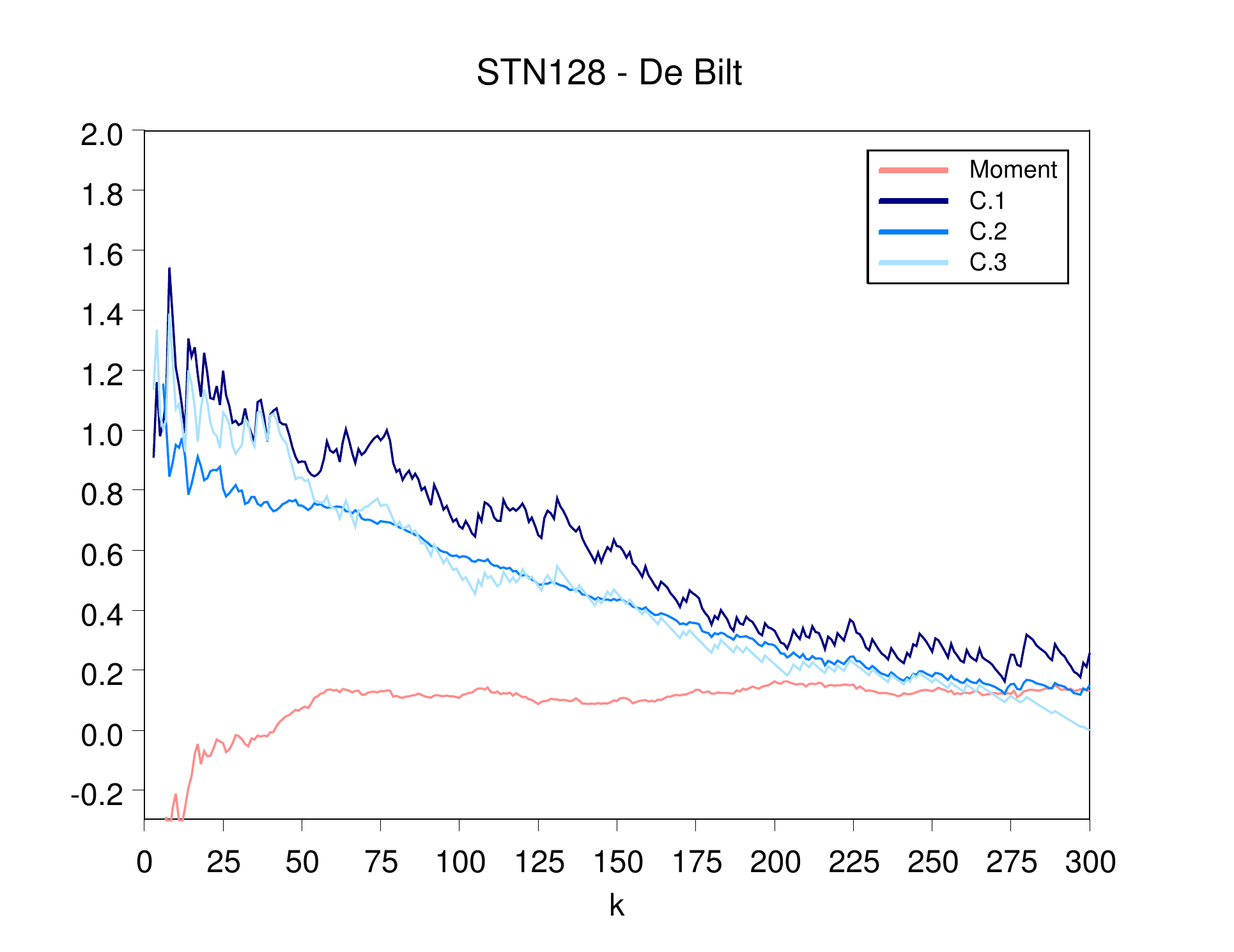}
\includegraphics*[scale=0.4]{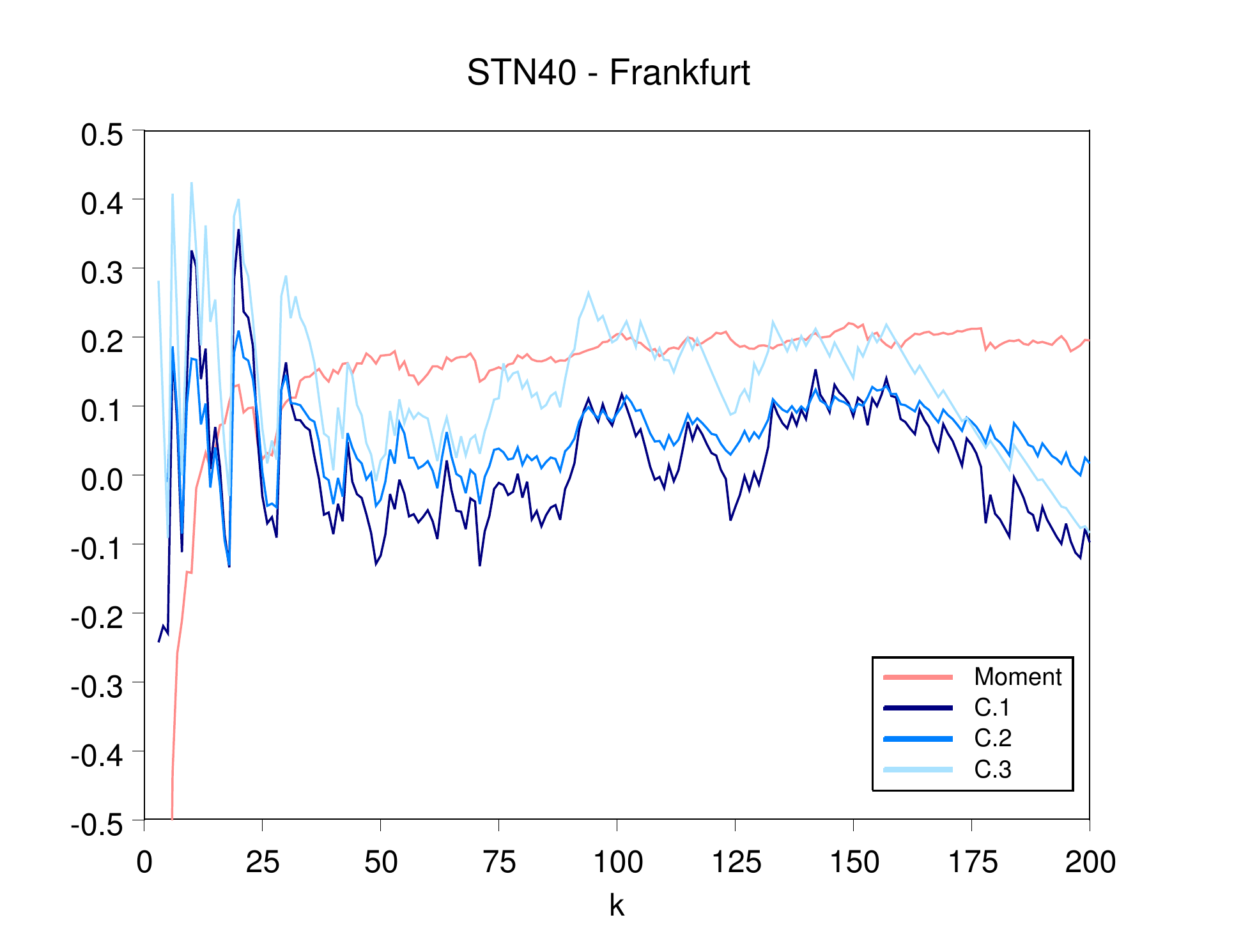}\vspace{0.5cm}
\includegraphics*[scale=0.4]{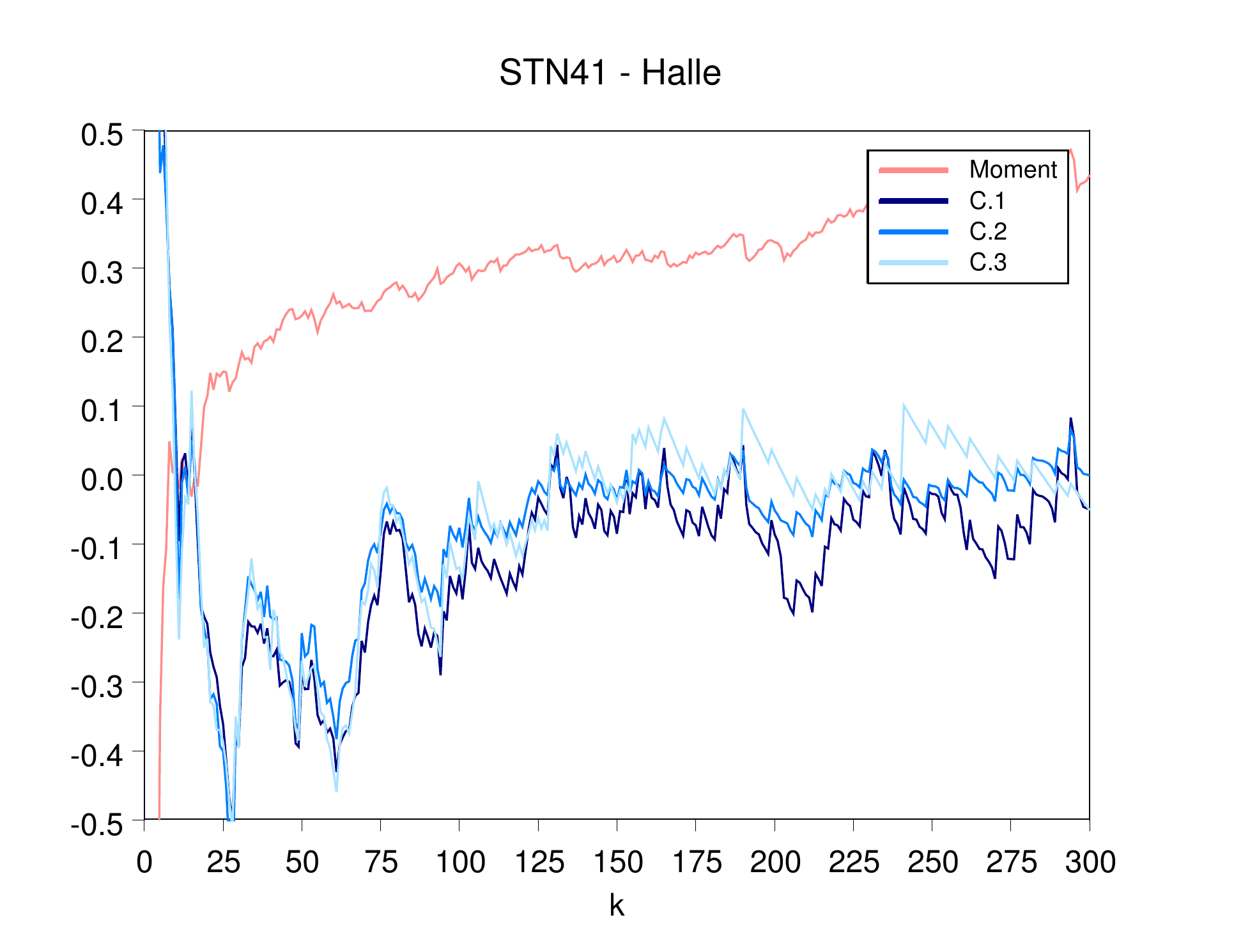}
\includegraphics*[scale=0.4]{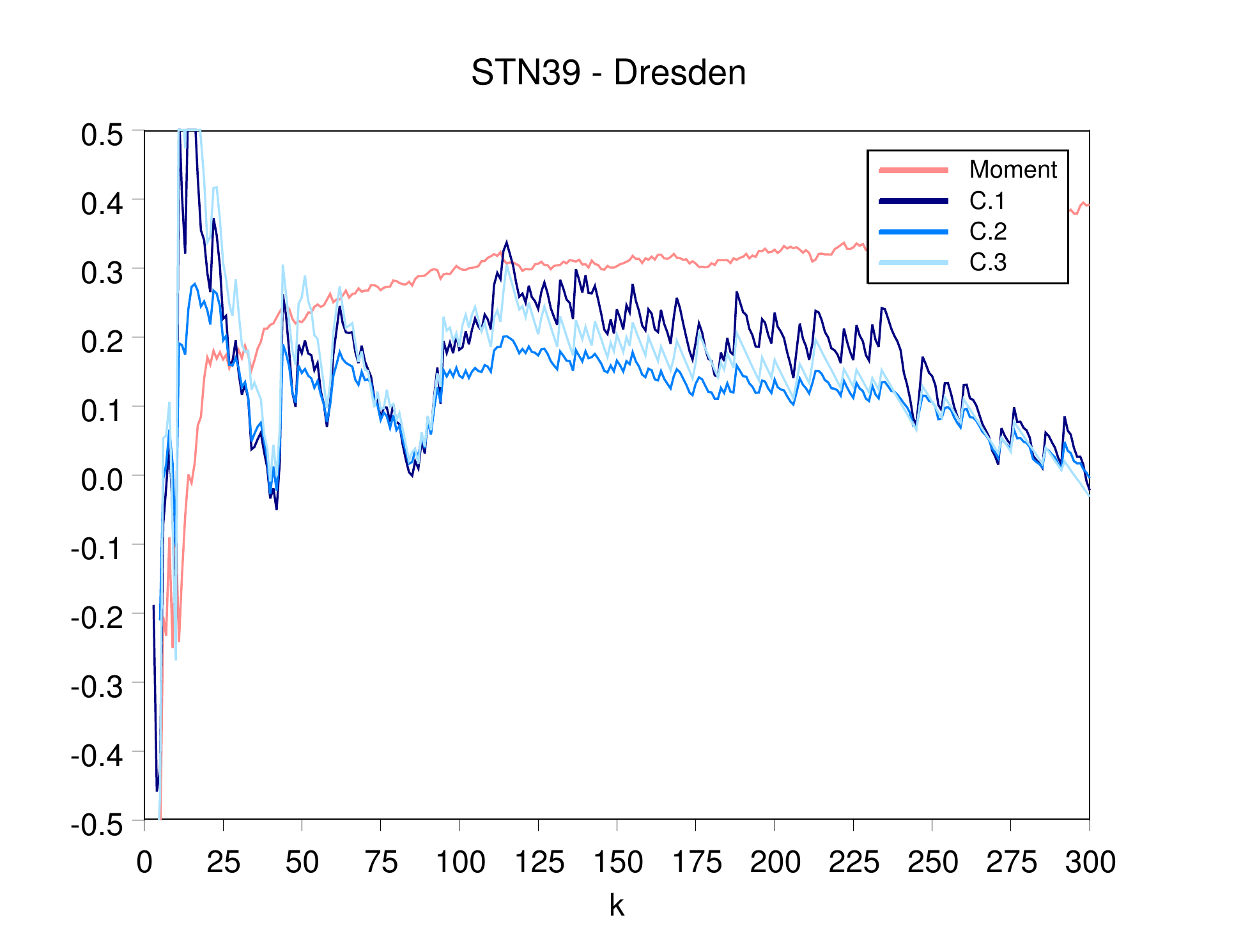}
\label{Fig.Stations}
\end{center}
\end{figure}

Table \ref{TableEstTrendDNL} contains the estimated values
of $c$ for each station by their increasing order of magnitude. Standard errors are also provided.
Bearing in mind the simulation results from section \ref{SecSims}, here we shall confine attention to the leading estimator $\hat{c}^{(2)}$.  Since the combined moment estimator $\hat{\gamma}=\hat{\gamma}_{n,k}$ as defined in \eqref{GammaOverall} is a consistent estimator for $\gamma$, the asymptotic standard error of $\hat{c}^{(2)}=\hat{c}$ can be estimated by
\begin{eqnarray*}
    \widehat{\mbox{s.e.}}(\hat{c}^{(2)})&:=& \frac{1}{\sqrt{k}}\, \frac{1}{\sumab{j=1}{m}s_j^2}\Bigl\{\Bigl(\sumab{j=1}{m}s_j\,\frac{1-e^{-\hat{c}\hat{\gamma} s_j}-\hat{c}\hat{\gamma} s_j}{\hat{\gamma}^2}\Bigr)^2\frac{ \sigma^2_{\Gamma}(\hat{\gamma})}{m}+ \sumab{j=1}{m}s_j^2\\
    & & \mbox{\hspace{2cm}}+\Bigl(\sumab{j=1}{m}s_je^{- \hat{c}\hat{\gamma} s_j}\Bigr)^2+\Bigl(\sumab{j=1}{m}s_j\,\frac{1-e^{-\hat{c}\hat{\gamma} s_j}}{\hat{\gamma}}\Bigr)^2 \sigma^2_{A_0}(\hat{\gamma})
\Bigr\}^{\frac{1}{2}},
\end{eqnarray*}
where
\begin{equation*}
    \sigma^2_{\Gamma}(\gamma):= \left\{
                                \begin{array}{ll}
                                  1+ \gamma^2, & \mbox{ } \gamma \geq 0, \\
                                  \frac{(1-\gamma)^2(1-2\gamma)(1-\gamma+6\gamma^2)}{(1-3\gamma)(1-4\gamma)}, & \mbox{ } \gamma<0
                                \end{array}
                              \right.
                              \end{equation*}
and
\begin{equation*}
\sigma^2_{A_0}(\gamma):= \left\{
                                \begin{array}{ll}
                                  2+\gamma^2, & \mbox{ } \gamma \geq 0,\\
                                  \frac{2-16\gamma+51\gamma^2-69\gamma^3+50\gamma^4-24\gamma^5}{(1-2\gamma)(1-3\gamma)(1-4\gamma)}, & \mbox{ } \gamma<0
                                \end{array}
                              \right.
\end{equation*}
(cf. Remark \ref{RemAsymptVar} in section \ref{SecResults}).

\begin{table}
\centering \caption{\footnotesize Station-wise estimates for the trend
parameter $c$ by adopting estimator $\hat{c}^{(2)}$.}\label{TableEstTrendDNL}
\medskip
\begin{footnotesize}
\begin{tabular}{|c|c|c|c|}
 \multicolumn{2}{c}{\emph{Continental Stations}}\\
\hline
  Station &$\hat{c}$& $\widehat{\mbox{s.e.}}(\hat{c})$ &$\hat{\gamma}$\\
  \hline
   STN 41 & $-0.2$ & $0.256$ &$0.2 $   \\
   STN 44 & $-0.1$ & $0.271$ &$0.15$  \\
   STN 48 & $-0.1$ & $0.289$ &$0.13$    \\
   STN 37 & $-0.06$& $0.325$ &$0.2$  \\
   STN 39 & $0.25$ & $0.308$ &$0.19$  \\
   STN 36 & $0.32$ & $0.311$ &$0.18$   \\
  \hline
\end{tabular}
\hspace{0.5cm}
\begin{tabular}{|c|c|c|c|}
 \multicolumn{2}{c}{\emph{Oceanic Stations}}\\
\hline
  Station &$\hat{c}$& $\widehat{\mbox{s.e.}}(\hat{c})$ &$\hat{\gamma}$  \\
  \hline
   STN 49  & $0.16$ & $0.322$ &$0.12 $  \\
   STN 40  & $0.2$ & $0.339$ &$0.1 $  \\
   STN 123 & $0.2$ & $0.300$ &$0.19$  \\
   STN 52  & $0.25$& $0.300$ &$ 0.08$  \\
   STN 122 & $0.55$& $0.311$ &$0.1$     \\
   STN 119 & $0.7$ & $0.323$ & $0.15$ \\
   STN 120 & $0.7$ & $0.349$ &$0.15$     \\
   STN 125 & $0.8$ & $0.408$ &$0.08$\\
   STN 128 & $0.8$ & $0.384$ &$-0.05$  \\
   STN 121 & $0.9$ & $0.386$ &$0.05$ \\
   STN 42  & $0.9$ & $0.432$ &$ 0.05$   \\
   STN 129 & $1.0$ & $0.400$ &$0.07$  \\
    \hline
\end{tabular}
\end{footnotesize}
\end{table}

\subsection{Detecting a trend in extreme rainfall}

It remains to assess whether the stations with near zero estimates in fact have  a null trend. Examples are
\emph{STN37--Berlin},  \emph{STN48--M\"{u}nchen}, \emph{STN49--M\"{u}nster}, \emph{STN52--Stuttgart} and \emph{STN39--Dresden}.  The last site we refer to is \emph{STN40--Frankfurt},
where testing for the presence of a trend is also of practical importance given  the poor circumstances involving the estimation of the
parameter $c$:  the erratic sample paths
displayed by the three estimators often cross the $c=0$ line
(cf. Figure \ref{Fig.Stations}). In the case of \emph{STN40--Frankfurt} it seems difficult to find a ``plateau of stability" in Figure \ref{Fig.Stations}; the estimate $\hat{c}=0.2$ is rather uncertain. Therefore, we aim at a more definite decision on
the value of $c$ by means of a testing procedure upon
\emph{STN40--Frankfurt} in particular.

In order to tackle the problem of testing the presence of a trend in time, i.e.\ the problem of testing hypothesis
\begin{equation}\label{H0andH1}
    H_0:\, c=0 \quad versus \quad H_1:\,c\neq 0,
\end{equation}
we shall use $Q^{(r)}_{m,n}$ from corollary \ref{CorTest}  as our test statistics.  Hence, for $r=1,\,2$, the null hypothesis $H_0:c=0$ is rejected in favor of the two-sided alternative $H_1: c\neq0$  if $Q^{(r)}_{m,n} > q_{m,1-\alpha}$, where $q_{m,1-\alpha}$ denotes the $(1-\alpha)$-quantile of the chi-square distribution with $m$ degrees of freedom.

Figure \ref{Fig.STNQmn} depicts the sample trajectories pertaining
to the two-sided test statistics $Q^{(1)}_{m,n}$ and $Q^{(2)}_{m,n}$
in companion with critical values at the nominal size $\alpha=0.05$ and with
respect to the referenced stations \emph{STN37--Berlin},
\emph{STN48--M\"{u}nchen}, \emph{STN49--M\"{u}nster} and \emph{STN40--Frankfurt}.  It seems that the hypothesis of no trend has to be rejected in case of \emph{STN40--Frankfurt}. The two tests also ascertain a non-null trend for \emph{STN48--M\"{u}nchen} and \emph{STN49--M\"{u}nster}. There is no evidence of a particular trend at \emph{STN37--Berlin}.

\begin{figure}
\vspace{6pc}
\caption{\footnotesize Sample paths of $Q^{(1)}_{18,n}$ and $Q^{(2)}_{18,n}$ test statistics and corresponding critical values for the two-sided test at a significance level $\alpha=0.05$ $(\chi^{-1}_{0.95}(17)=8.67)$ for several stations across Germany.}
\begin{center}
\includegraphics*[scale=0.4]{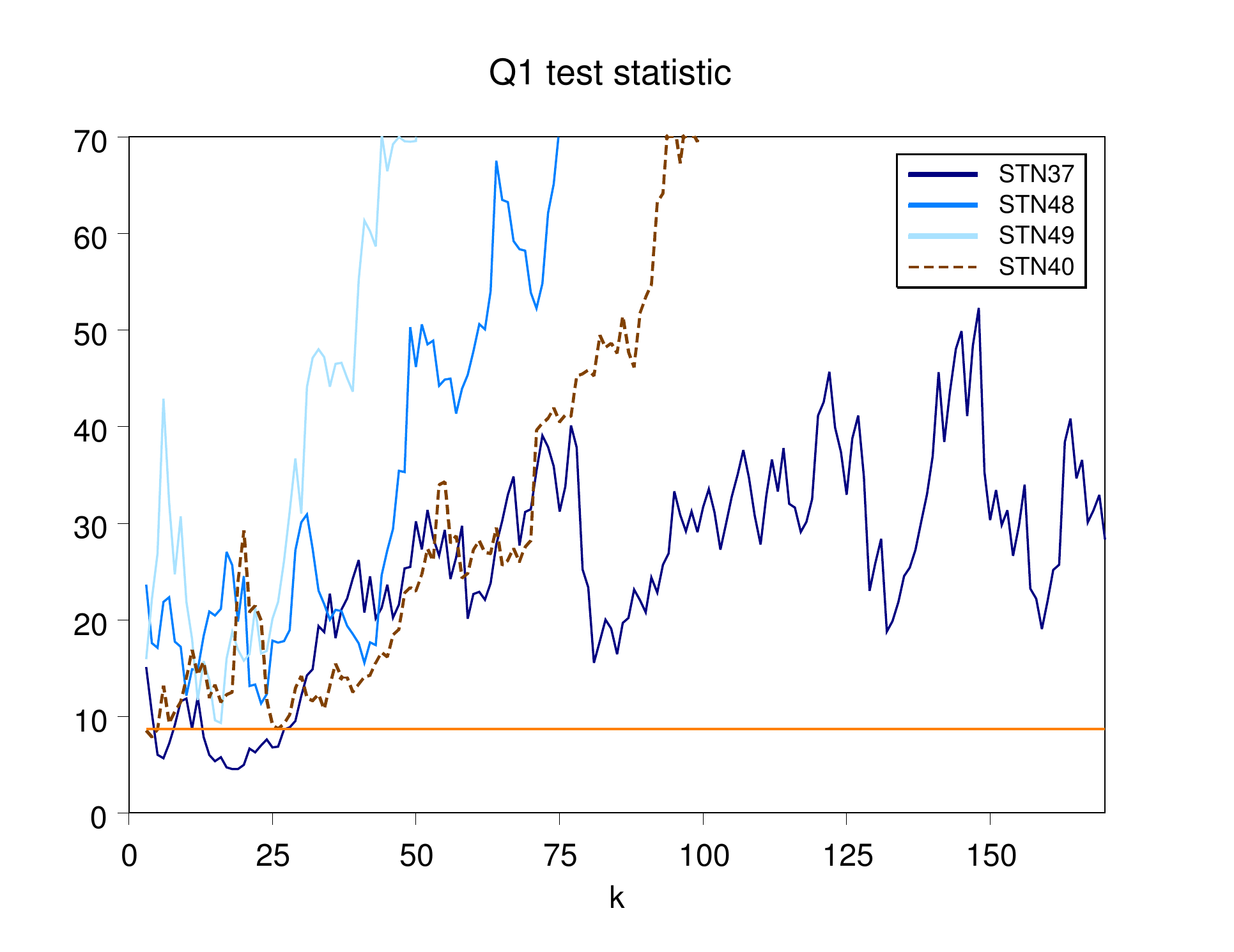}
\includegraphics*[scale=0.4]{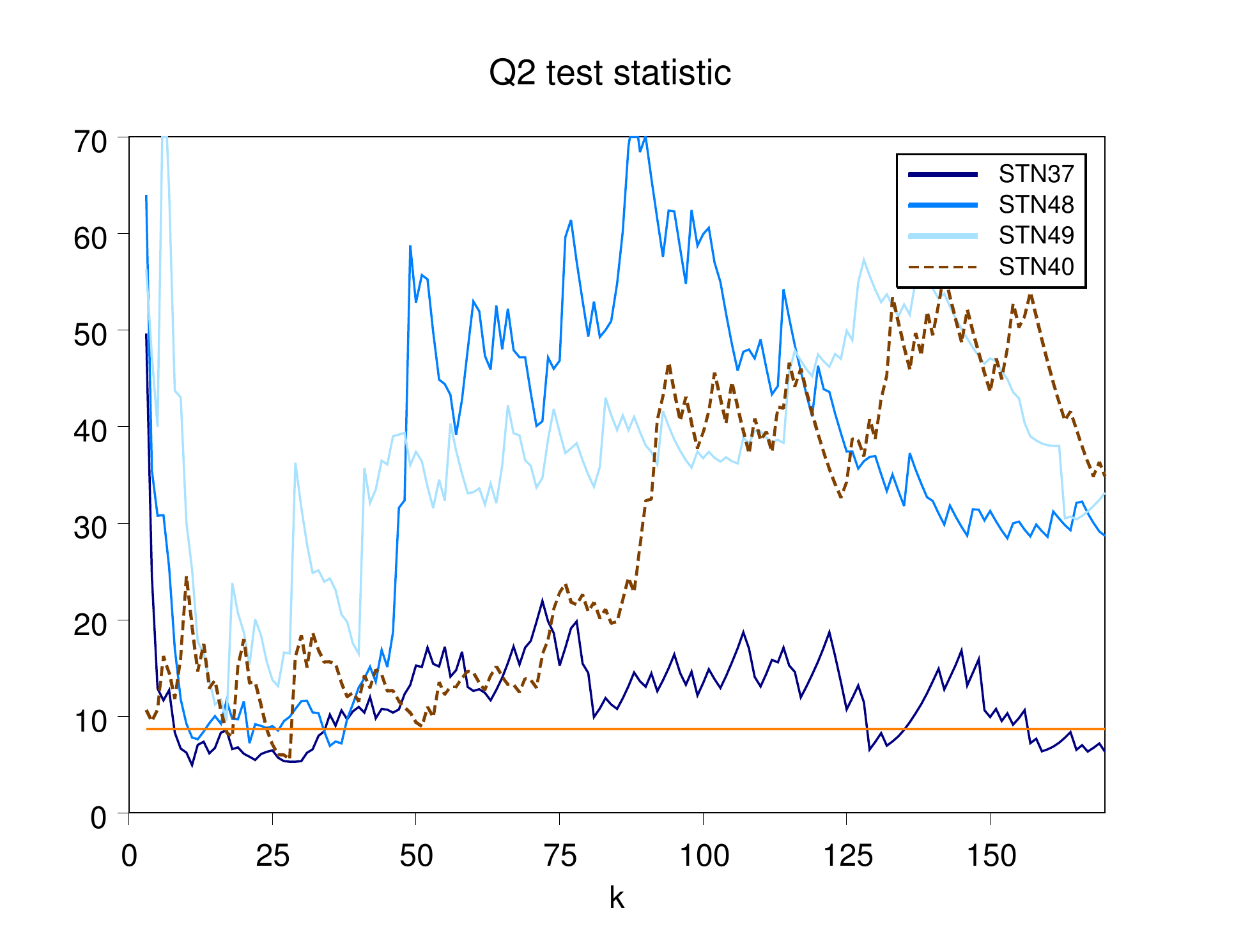}
 \label{Fig.STNQmn}
\end{center}
\end{figure}

\subsection{Discussion}
Broadly speaking, the stations in the
Oceanic group present higher positive values of $c$ (but not always smaller values of $\gamma$) than the Continental stations. Overall, fairly positive values of
$c$ may be interpreted as being influenced by the ocean.

\citet{AllenIngram:02} describe how the intensity of extreme rainfall
events depends on the availability of moisture. Because moisture
availability is constraint on temperature (through the
Clausius-Clapeyron relationship), an increase in rainfall extremes is
expected in a warming climate. \citet{Lenderinketal:09} show that higher
increases can be expected at locations that are under the influence of
the sea. Increasing sea surface temperatures contribute to higher
rainfall amounts. The results obtained in this study for the behavior
of extreme rainfall at locations in Germany and The Netherlands are
consistent with these findings. Overall, the Oceanic group of stations
shows a stronger increase in extreme rainfall than the Continental group
of stations.

Next we discuss the interpretation of $c$. There is no evidence of a significant trend in the extreme relative risk at \emph{STN37--Berlin}, meaning that the value $c=0$ can be assigned to this gauging station. If $c = 1.0$, like the estimated value at \emph{STN129--Eelde} (see Table \ref{TableEstTrendDNL}), then in view of the fact
that $s$ is measured in periods of 5 years, the probability of really heavy
rainfall increases approximately by $12.5\%$ during each decade. Figure \ref{Fig.LogRR} may help to clarify the contrast in these values by plotting the empirical log-relative risk 
\begin{equation*}
        \log\,
        \frac{1-\widehat{F}_s\bigl(X_{n-k,n}(0)\bigr)}{1-\widehat{F}_0\bigl(X_{n-k,n}(0)\bigr)}
        \approx cs
    \end{equation*}
against several values of $k$ and for $s_j=j/m \in [0,1]$. At \emph{STN129--Eelde}, there is no clear evidence of a trend, whereas  at \emph{STN129--Eelde} the estimated log-relative risk seems to increase as $s_j$ approaches 1.

\begin{figure}
\vspace{6pc}
\caption{\footnotesize Estimated $\log$-relative risk for $s_j= j/17$, with $j=1,2,\ldots, 17$ marking periods of 5 years. Equal distances between $s_j$ are chosen to reflect time periods of equal lenght.}
\begin{center}
\includegraphics*[scale=0.4]{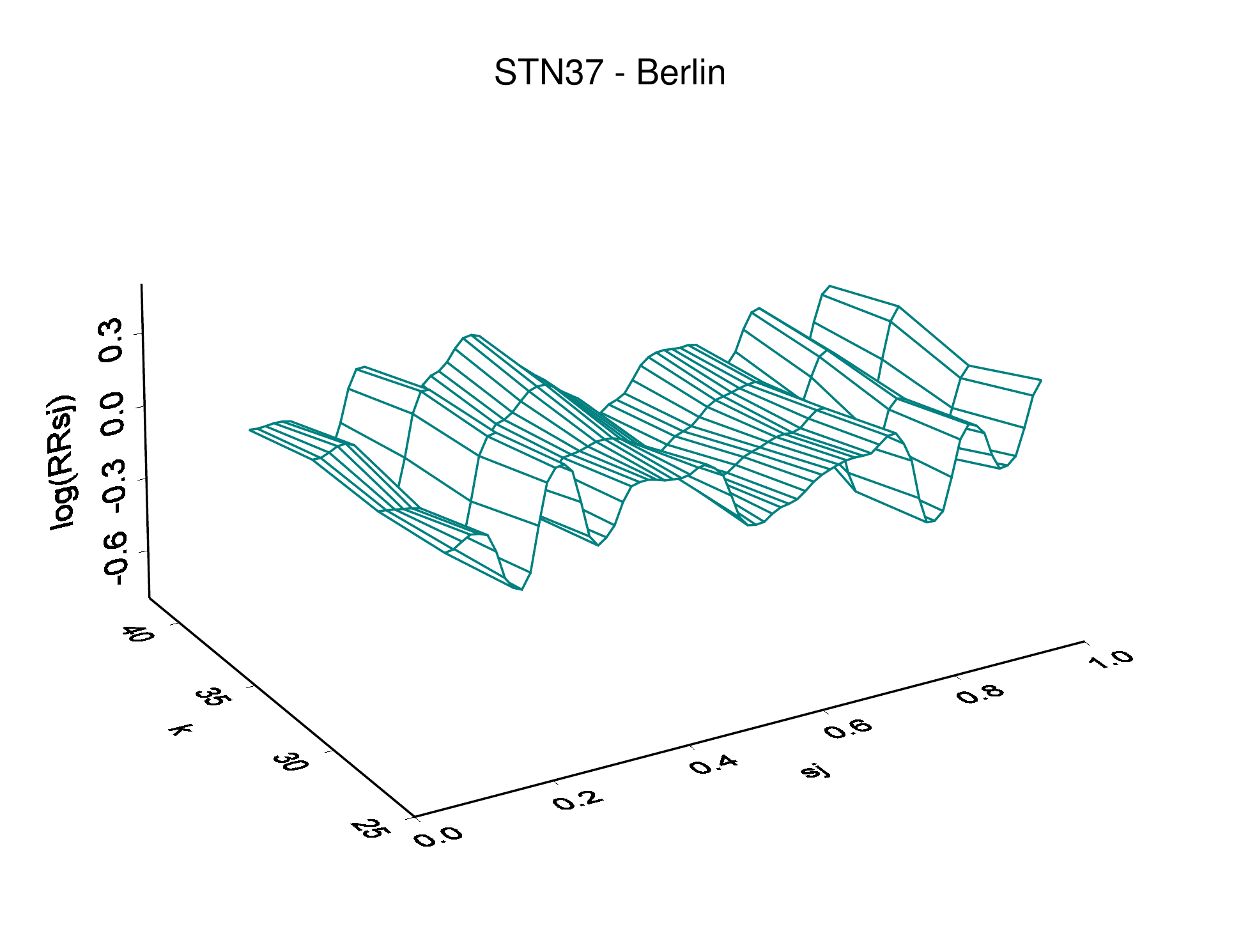}
\includegraphics*[scale=0.4]{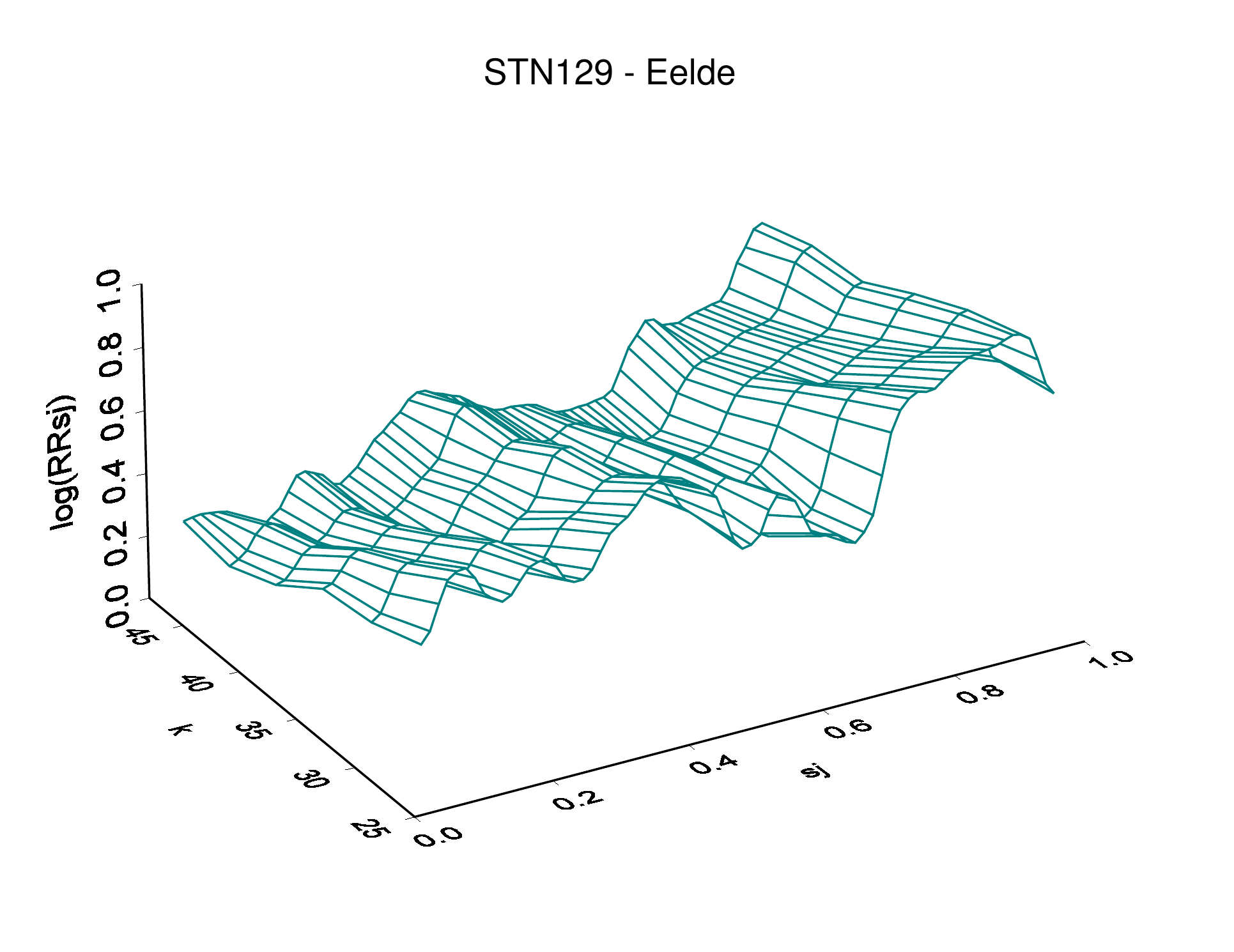}
 \label{Fig.LogRR}
\end{center}
\end{figure}

Similarly if $c = 0.2$,  which coincides with the estimate for the trend at \emph{STN40--Frankfurt} and  \emph{STN123--Westdorpe}, and it is approximately the case in \emph{STN49--M\"{u}nster}, then the probability of heavy rainfall increases in the same period
approximately by $2\%$. These results are in good agreement with the positive trend of 3\% per decade found by \citet{Zolinaetal:09} for the second half of the century (1950-2000) using a
different metric.

\section{Proofs}
\label{SecProofs}

We shall use the following representation:
\begin{equation*}
    \bigl\{X_{n-i,n}(s_j)\bigr\}_{i=1}^n \id \bigl\{U_{s_j}(Y_{n-i,n}(s_j))\bigr\}_{i=1}^n
\end{equation*}
where $\{Y_{n-i,n}(s_j)\}_{i=1}^n$ are the $n-$th order statistics from
the distribution function $1-x^{-1}$, $x\geq 1$, independently for $j=1,2,\ldots,m$.

\subsubsection*{Proof of consistency}

\noindent For the consistency of $\hat{c}^{(1)}$ note that
$(k/n)\,Y_{n-k,n}(s_j)\, \displaystyle {\mathop {\longrightarrow}^{P}} \,1$,
$n\rightarrow \infty$ (cf. \citet{deHaanFerreira:06}, Corollary
2.2.2) and that $\lim_{t \rightarrow \infty}
U_s(tx)/U_s(t)= x^{\gamma^+}$  locally uniformly for $x>0$. Hence for
$j=1,\,2,\ldots,m$
\begin{equation*}
    \log X_{n-k,n}(s_j)-\log U_{s_j}\bigl( \ndivk\bigr) \id   -\log U_{s_j}\bigl( \ndivk \bigl\{ \kdivn Y_{n-k,n}(s_j) \bigr\}\bigr) -\log U_{s_j}\bigl( \ndivk\bigr) \conv{P} 0.
\end{equation*}
The rest is easy.

\noindent Similarly with respect to $\hat{c}^{(2)}$ we get that
\begin{equation}\label{RandLev2Us}
    \frac{X_{n-k,n}(s_j)-U_{s_j}\bigl(
    \ndivk\bigr)}{a_{s_j}\bigl(\ndivk\bigr)}\conv{P}0.
\end{equation}
This limit relation combined with relation \eqref{UsU0Dif} leads
directly to the consistency of $\hat{c}^{(2)}$.

\noindent With respect to the consistency of $\hat{c}^{(3)}$, we
begin by noting that the domain of attraction condition
\begin{equation*}
\limit{n} \, \ndivk\, P\Bigl\{X_l(0)> U_0\bigl(\ndivk \bigr)+ x\,
a_0\bigl(\ndivk \bigr)\Bigr\}= (1+\gamma x)^{-1/\gamma}
\end{equation*}
for all $l=1,2,\ldots$, combined with \eqref{TrendPropag} implies
\begin{equation*}
\limit{n} \, e^{-cs_j}\ndivk\, P\Bigl\{X_l(s_j)> U_0\bigl(\ndivk
\bigr)+ x\, a_0\bigl(\ndivk \bigr)\Bigr\}= (1+\gamma x)^{-1/\gamma},
\end{equation*}
$j=1,2,\ldots, m$. Hence the characteristic functions converge:
\begin{eqnarray*}
    & &E \exp\Bigl\{i \frac{t}{k}\sumab{l=1}{n}I_{\{X_l(s_j)> \,U_0(n/k)+x\,a_0(n/k)\}}\Bigr\}\\
    &=& \biggl(E \exp \Bigl\{i \frac{t}{k}\, I_{\{X_1(s_j)> \,U_0(n/k)+x\,a_0(n/k)\}} \Bigr\}\biggr)^n\\
    &=& \biggl( e^{i\frac{t}{k}}\,P\Bigl\{X_1(s_j)> U_0\bigl(\ndivk\bigr)+x\,a_0\bigl(\ndivk\bigr) \Bigr\}\\
    & & \mbox{\hspace{1.5cm}} + 1-P\Bigl\{X_1(s_j)> U_0\bigl(\ndivk\bigr)+x\,a_0\bigl(\ndivk\bigr) \Bigr\}\biggr)^n\\
    &=& \biggl(1+e^{cs_j}k\bigl(e^{it/k}-1\bigr)\,\frac{e^{-cs_j}(n/k)\bigl[1-F_{s_j}(U_0(n/k)+xa_0(n/k))\bigr]}{n} \biggr)^n \\
    &\arrowf{n} & \exp\Bigl\{it\, e^{cs_j}(1+\gamma x)^{-1/\gamma}\Bigr\},
\end{eqnarray*}
for every $t \in \real$. Owing to L\'{e}vy's continuity theorem, the
latter implies
\begin{equation*}
    \frac{1}{k}\sumab{i=1}{n}I_{\{X_i(s_j)> \,U_0(n/k)+x\,a_0(n/k)\}} \conv{P} e^{cs_j}(1+\gamma\, x)^{-1/\gamma}.
\end{equation*}
Next use \eqref{RandLev2Us}.\hfill \ding{111}
\\

For the proof of the asymptotic normality we need an auxiliary result.

\begin{lem}\label{LemERVUs}
Assume conditions \eqref{2ndERVU0} and \eqref{2ndUs2U0Dif}. Define
\begin{eqnarray}\label{2ndRVaAs}
\nonumber    \alpha_s(t)&:=& e^{cs\rho}\alpha_0(t),\\
             a_s(t)&:=& e^{cs\gamma}a_0(t)\Bigl(1+\alpha_0(t)\,\frac{e^{cs\rho}-1}{\rho}
             \Bigr).
\end{eqnarray}
Then for $s\in \real$ and $x>0$
\begin{equation}\label{2ndRVas2a0}
    \limit{t}\frac{\frac{U_s(tx)-U_s(t)}{a_s(t)}-\frac{x^\gamma-1}{\gamma}}{\alpha_s(t)}
    =H_{\gamma,\rho}(x).
\end{equation}
\end{lem}
\begin{proof}
For simplicity we write $d$ for $e^{cs}$ in this proof.

\noindent Relation \eqref{2ndUs2U0Dif} implies \citep[cf.][p.44]{deHaanFerreira:06} that for $x>0$
\begin{equation*}
\limit{t}\, \frac{\frac{a_0(tx)}{a_0(t)}-x^{\gamma}}{\alpha_0(t)}= x^{\gamma}\frac{x^{\rho}-1}{\rho}
\end{equation*}
and
\begin{equation*}
\limit{t}\, \frac{\alpha_0(tx)}{\alpha_0(t)}=x^{\rho}.
\end{equation*}

\noindent First note that \eqref{2ndERVU0} and \eqref{2ndUs2U0Dif} imply that
\begin{eqnarray*}
  \frac{\frac{U_s(tx)-U_s(t)}{a_0(t)}-d^{\gamma}\frac{x^{\gamma}-1}{\gamma}}{\alpha_0(t)} &=& \frac{\frac{U_0(tx)-U_0(t)}{a_0(t)}-\frac{x^{\gamma}-1}{\gamma}}{\alpha_0(t)}+\frac{\frac{a_0(tx)}{a_0(t)}-x^{\gamma}}{\alpha_0(t)}\,\frac{d^{\gamma}-1}{\gamma} \\
  & & \mbox{\hspace{0.8cm}} +\Bigl(\frac{\alpha_0(tx)\,a_0(tx)}{\alpha_0(t)\,a_0(t)}-1\Bigr)H_{\gamma,\rho}(d)\bigl( 1+o(1)\bigr)
\end{eqnarray*}
converges to
\begin{equation*}
    H_{\gamma,\rho}(x)+x^{\gamma}\frac{x^{\rho}-1}{\rho}\,\frac{d^{\gamma}-1}{\gamma}+ (x^{\gamma+\rho}-1)H_{\gamma,\rho}(d).
\end{equation*}
\noindent Next write
\begin{equation*}
  \frac{\frac{U_s(tx)-U_s(t)}{a_s(t)}-\frac{x^{\gamma}-1}{\gamma}}{\alpha_s(t)}
 = \frac{\alpha_0(t)}{\alpha_s(t)}\biggl( \frac{\frac{U_s(tx)-U_s(t)}{a_0(t)}-d^{\gamma}\frac{x^{\gamma}-1}{\gamma}}{\alpha_0(t)}\,\frac{a_0(t)}{a_s(t)}+\frac{x^{\gamma}-1}{\gamma}\frac{ \frac{d^{\gamma}a_0(t)}{a_s(t)}-1}{\alpha_0(t)}\biggr).
\end{equation*}
This converges to
\begin{equation*}
    d^{-(\gamma+\rho)}\Bigl(H_{\gamma,\rho}(x)+x^{\gamma}\frac{x^{\rho}-1}{\rho}\,\frac{d^{\gamma}-1}{\gamma}+(x^{\gamma+\rho}-1)H_{\gamma,\rho}(d)\Bigr)-\frac{x^{\gamma}-1}{\gamma}\,\frac{1-d^{-\rho}}{\rho}
\end{equation*}
which is equal to $H_{\gamma,\rho}(x)$.
\end{proof}

\begin{rem}
The analogue of Lemma \ref{LemERVUs} stemming from conditions \eqref{2ndRVU0} and \eqref{2ndUs2U0} -- i.e. $\gamma >0$ -- holds with the auxiliary function $\beta_s(t):= e^{c s \widetilde{\rho}}\beta(t)$. This
leads to the following relation for every $s \in \real$,
\begin{equation*}
    \limit{t} \frac{\frac{U_s(tx)}{U_s(t)}-x^{\gamma^+}}{\beta_s(t)}= x^{\gamma^+}\frac{x^{\widetilde{\rho}}-1}{\widetilde{\rho}}, \quad  x>0.
\end{equation*}
\end{rem}

\subsubsection*{Proof of asymptotic normality}

We write
\begin{eqnarray*}
    \sqrt{k}\, \bigl(\hat{c}^{(1)}-c \bigr)
    &=&\sqrt{k}\, \Bigl(
        \frac{1}{\hat{\gamma}^+_{n,k}}-\frac{1}{\gamma^+}\Bigr)\,\hat{\gamma}^+_{n,k}\,\hat{c}^{(1)}+
        \Bigl(\gamma^+ \sumab{j=1}{m}s^2_j \Bigr)^{-1}\\
    & & \mbox{\hspace{1.0cm}} \times  \biggl\{\sumab{j=1}{m}s_j\Bigl[\sqrt{k}\, \Bigl(\log X_{n-k,n}(s_j)- \log U_{s_j}\bigl(\ndivk\bigr) \Bigr)\\
    & & \mbox{\hspace{1.5cm}} -\sqrt{k}\, \Bigl(\log X_{n-k,n}(0)- \log U_{0}\bigl(\ndivk\bigr) \Bigr)\\
    & & \mbox{\hspace{2.0cm}}+\sqrt{k}\, \beta\bigl(\ndivk \bigr)\, \frac{\log U_{s_j}(n/k)-\log
        U_0(n/k)-c\gamma^+s_j}{\beta(n/k)}\Bigr]\biggr\}.
\end{eqnarray*}
The result follows from \eqref{BivANgammaPlus}, \eqref{2ndUs2U0} and \eqref{ConstBias}.\\

\noindent For $\hat{c}^{(2)}$ it is sufficient to consider
\begin{equation*}
    \sqrt{k}\, \Bigl\{\log \Bigl(1+\hat{\gamma}_{n,k}\frac{X_{n-k,n}(s_j)-X_{n-k,n}(0)}{\hat{a}_0\bigl(\ndivk \bigr)} \Bigr)^{\frac{1}{\hat{\gamma}_{n,k}}}-cs_j\Bigr\}
\end{equation*}
for $j=1,\,2,\ldots,m$ where $\hat{\gamma}_{n,k}= 1/m\sum_{j=1}^{m}\hat{\gamma}_{n,k}(s_j)$. We use Cram\'{e}r's delta method.
\begin{equation*}
    \frac{\partial}{\partial \gamma}\, \log (1+\gamma
    x)^{\frac{1}{\gamma}}= \frac{1}{\gamma^2}\Bigl(\frac{\gamma x}{1+\gamma x}-\log(1+\gamma x) \Bigr)
\end{equation*}
(which is $x^2/2$ for $\gamma=0$) and
\begin{equation*}
    \frac{\partial}{\partial x}\, \log (1+\gamma
    x)^{\frac{1}{\gamma}}= \frac{1}{1+\gamma x}
\end{equation*}
($=1$ for $x=0$). Further we write
\begin{eqnarray*}
    & &\frac{X_{n-k,n}(s_j)-X_{n-k,n}(0)}{\hat{a}_0\bigl(\ndivk \bigr)}\\
    &=&\frac{a_0\bigl( \ndivk\bigr)}{\hat{a}_0\bigl(\ndivk \bigr)}\,\biggl\{\frac{a_{s_j}\bigl(\ndivk \bigr)}{a_0\bigl(\ndivk \bigr)} \frac{X_{n-k,n}(s_j)-U_{s_j}\bigl(\ndivk \bigr)}{a_{s_j}\bigl(\ndivk \bigr)}\\
    & & \mbox{\hspace{3.0cm}}-\frac{X_{n-k,n}(0)-U_0\bigl(\ndivk \bigr)}{a_0\bigl( \ndivk\bigr)}\,+\frac{U_{s_j}\bigl(\ndivk \bigr)-U_{0}\bigl(\ndivk \bigr)}{a_0\bigl(
    \ndivk\bigr)}\biggr\},
\end{eqnarray*}
implying (cf. \eqref{ScaleFct})
\begin{eqnarray*}
    & &\sqrt{k}\biggl(\frac{X_{n-k,n}(s_j)-X_{n-k,n}(0)}{\hat{a}_0\bigl(\ndivk \bigr)}-\frac{e^{c \gamma
    s_j}-1}{\gamma}\biggr)\\
    &=&
    \frac{e^{c \gamma s_j}-1}{\gamma}\,\sqrt{k}\,\biggl(\frac{a_0\bigl( \ndivk\bigr)}{\hat{a}_0\bigl(\ndivk \bigr)}-1\biggr)+ \frac{a_0\bigl( \ndivk\bigr)}{\hat{a}_0\bigl(\ndivk \bigr)}\biggl\{\frac{a_{s_j}\bigl(\ndivk \bigr)}{a_0\bigl(\ndivk \bigr)}\,\sqrt{k}\, \frac{X_{n-k,n}(s_j)-U_{s_j}\bigl(\ndivk \bigr)}{a_{s_j}\bigl(\ndivk \bigr)}\\
    & & \mbox{\hspace{1.0cm}}-\sqrt{k}\,\frac{X_{n-k,n}(0)-U_0\bigl(\ndivk \bigr)}{a_0\bigl( \ndivk\bigr)}\,+\sqrt{k}\,\biggl(\frac{U_{s_j}\bigl(\ndivk \bigr)-U_{0}\bigl(\ndivk \bigr)}{a_0\bigl(\ndivk\bigr)}-\frac{e^{c\gamma s_j}-1}{\gamma}\biggr)\biggr\}.
\end{eqnarray*}
Hence by \eqref{ScaleFct}
\begin{eqnarray*}
    & & \sqrt{k}\biggl(\frac{X_{n-k,n}(s_j)-X_{n-k,n}(0)}{\hat{a}_0\bigl(\ndivk \bigr)}-\frac{e^{c \gamma s_j}-1}{\gamma}\biggr)\\
    &\conv{d}& - \frac{e^{c \gamma
    s_j}-1}{\gamma}\,A(0)+ e^{c\gamma s_j}B(s_j)-B(0)+\lambda \,
    H_{\gamma, \rho}(e^{cs_j}).
\end{eqnarray*}
Next we apply the delta method:
\begin{eqnarray*}
    & &\sqrt{k}\Bigl\{\log\Bigl( 1+\hat{\gamma}_{n,k}\frac{X_{n-k,n}(s_j)-X_{n-k,n}(0)}{\hat{a}_0\bigl(\ndivk \bigr)}\Bigr)^{\frac{1}{\hat{\gamma}_{n,k}}} \\
    & & \mbox{\hspace{5.0cm}}-\log \Bigl(1+\gamma \, \frac{e^{c\gamma s_j}-1}{\gamma} \Bigr)^{\frac{1}{\gamma}}
    \Bigr\}\\
    & \conv{d}& \left[ \frac{\partial}{\partial \gamma}\, \log (1+\gamma
    x)^{\frac{1}{\gamma}}\right]_{x= \frac{e^{c \gamma
    s_j}-1}{\gamma}}\cdot \frac{1}{m}\sumab{i=1}{m}\Gamma(s_i) \\
    & & \mbox{\hspace{0.5cm}} + \left[ \frac{\partial}{\partial x}\, \log (1+\gamma
    x)^{\frac{1}{\gamma}}\right]_{x= \frac{e^{c \gamma
    s_j}-1}{\gamma}}\Bigl\{-\frac{e^{c\gamma s_j}-1}{\gamma}\, A(0)\\
    & & \mbox{\hspace{4.0cm}}+ e^{c\gamma s_j}B(s_j)-B(0)+\lambda \, H_{\gamma,
    \rho}(e^{cs_j})\Bigr\}\\
    &=& \frac{1-e^{-c \gamma s_j}-c \gamma s_j
    }{\gamma^2}\,\frac{1}{m}\sumab{i=1}{m}\Gamma(s_i)  + e^{-c\gamma s_j}\Bigl\{e^{c\gamma s_j}B(s_j)\\
    & & \mbox{\hspace{4.0cm}}-B(0)-\frac{e^{c\gamma s_j}-1}{\gamma}A(0)+\lambda \,H_{\gamma,
    \rho}(e^{cs_j})\Bigr\}.
\end{eqnarray*}
The result follows.\\

\noindent Finally for $\hat{c}^{(3)}$ consider first
\begin{equation}\label{ReducedProcess}
    \frac{1}{k}\sumab{i=1}{n}I_{\bigl\{X_i(s)>X_{n-k,n}(0)\bigr\}}= \frac{1}{k}\sumab{i=1}{n} I_{\bigl\{\frac{X_i(s)-U^{\star}_s(n/k)}{a^{\star}_s(n/k)}>\frac{X_{n-k,n}(0)-U^{\star}_s(n/k)}{a^{\star}_s(n/k)}\bigr\}}.
\end{equation}
We write
\begin{equation*}
    x_n(s):= \frac{X_{n-k,n}(0)-U^{\star}_s\bigl( \ndivk\bigr)}{a^{\star}_s\bigl( \ndivk\bigr)},
\end{equation*}
with $U^{\star}$ and $a^{\star}$ from Corollary 2.3.7 (but with a different notation since we use the subscript 0 for another purpose here) of
\citet{deHaanFerreira:06}, coupled with Lemma
\ref{LemERVUs},
\begin{eqnarray*}
    \alpha_s^{\star}(t)&:=& \left\{
                     \begin{array}{ll}
                       \alpha_0(t)\frac{e^{cs\rho}}{\rho}, & \mbox{ } \rho <0, \\
                       \alpha_0(t), & \mbox{   } \rho=0,
                     \end{array}
                   \right.\\
    a_s^{\star}(t)&:=& \left\{
                         \begin{array}{ll}
                           e^{cs\gamma}a_0(t)\bigl( 1-\alpha_0(t)\frac{1}{\rho}\bigr), & \mbox{ } \rho<0,\\
                           e^{cs\gamma}a_0(t)\bigl(1+\alpha_0(t)\bigl\{cs-\frac{1}{\gamma}\bigr\} \bigr), & \hbox{ }\rho=0\neq \gamma, \\
                           a_0(t)\bigl(1+\alpha_0(t) cs\bigr), & \hbox{ }\gamma=\rho=0,
                         \end{array}
                       \right.\\
    U_s^{\star}(t)&:=& \left\{
                     \begin{array}{ll}
                       U_s(t)-\frac{e^{cs(\gamma+\rho)}}{\gamma+\rho}\,a_0^{\star}(t)\alpha^{\star}_0(t), & \mbox{ }\gamma+\rho \neq 0, \, \rho <0, \\
                       U_s(t), & \mbox{   otherwise},
                     \end{array}
                   \right.
\end{eqnarray*}
and write \eqref{ReducedProcess} as
\begin{equation*}
    \ndivk \, \Bigl\{1-F_n^{(s)}\Bigl( U^{\star}_s\bigl(\ndivk\bigr)+x_n(s)\,a^{\star}_s\bigl(\ndivk \bigr)\Bigr)\Bigr\}.
\end{equation*}
where $F_n^{(s)}$ is the empirical distribution function of the random sample $X_1(s), X_2(s), \ldots, X_n(s)$.
\noindent We consider $x_n(s)$ first. We use
\eqref{2ndUs2U0Dif} and Theorem 2.4.2 of \citet{deHaanFerreira:06}.
\begin{eqnarray}\label{ANxn}
\nonumber & &  \sqrt{k}\Bigl( x_n(s)-\frac{e^{-cs\gamma}-1}{\gamma}\Bigr)\\
 \nonumber &=& \sqrt{k}\Bigl( \frac{X_{n-k,n}(0)-U^{\star}_s\bigl(\ndivk\bigr)}{a^{\star}_s\bigl(\ndivk\bigr)}-\frac{e^{-cs\gamma}-1}{\gamma}\Bigr)\\
\nonumber  &=& \frac{a^{\star}_0\bigl(\ndivk\bigr)}{a^{\star}_s\bigl(\ndivk\bigr)}\, \sqrt{k} \,\frac{X_{n-k,n}(0)-U^{\star}_0\bigl(\ndivk\bigr)}{a^{\star}_0\bigl(\ndivk\bigr)}- \frac{a^{\star}_0\bigl(\ndivk\bigr)}{a^{\star}_s\bigl(\ndivk\bigr)}\,\sqrt{k}\Bigl( \frac{U^{\star}_s\bigl(\ndivk\bigr)-U^{\star}_0\bigl(\ndivk\bigr)}{a^{\star}_0\bigl(\ndivk\bigr)}-\frac{e^{cs\gamma}-1}{\gamma}\Bigr)\\
\nonumber    & &\qquad  -\sqrt{k}\Bigl(\frac{a^{\star}_0\bigl(\ndivk\bigr)}{a^{\star}_s\bigl(\ndivk\bigr)}-e^{-cs\gamma}\Bigr)\frac{e^{cs\gamma}-1}{\gamma}\\
\nonumber    &=& \frac{a^{\star}_0\bigl(\ndivk\bigr)}{a^{\star}_s\bigl(\ndivk\bigr)}\, \sqrt{k}\,  \frac{X_{n-k,n}(0)-U_0\bigl(\ndivk\bigr)}{a^{\star}_0\bigl(\ndivk\bigr)}-\frac{a^{\star}_0\bigl(\ndivk\bigr)}{a^{\star}_s\bigl(\ndivk\bigr)}\,\sqrt{k}\Bigl(\frac{U^{\star}_0\bigl(\ndivk\bigr)-U_0\bigl(\ndivk\bigr)}{a^{\star}_0\bigl(\ndivk\bigr)} +\frac{U^{\star}_s\bigl(\ndivk\bigr)-U^{\star}_0\bigl(\ndivk\bigr)}{a^{\star}_0\bigl(\ndivk\bigr)}-\frac{e^{cs\gamma}-1}{\gamma}\Bigr)\\
    & &\qquad  -\sqrt{k}\Bigl(\frac{a^{\star}_0\bigl(\ndivk\bigr)}{a^{\star}_s\bigl(\ndivk\bigr)}-e^{-cs\gamma}\Bigr)\frac{e^{cs\gamma}-1}{\gamma}.
\end{eqnarray}
The first term converges in distribution to $e^{-cs\gamma} W^{(0)}(1)$. The second term converges to $\lambda^\star b(s)$ defined by
\begin{equation*}
    b(s):= \left\{
                 \begin{array}{ll}
                    \frac{e^{-cs\gamma}}{\gamma+\rho}, & \, \gamma+\rho\neq 0, \, \rho <0,\\
                    -e^{-cs\gamma}cs, & \,  \gamma+\rho= 0, \, \rho <0, \\
                    -\frac{cs}{\gamma}, & \,\rho=0\neq \gamma, \\
                    -\frac{1}{2}(cs)^2, & \, \gamma=\rho=0
                 \end{array}
               \right.
\end{equation*}
and the third term converges to 
\begin{equation*}
\lambda^\star e^{-cs\gamma}\, \frac{e^{cs\gamma}-1}{\gamma}cs\, I_{\{\rho=0\}}.
\end{equation*}
Here
\begin{equation*}
\lambda^{\star}:= \lim_{n\rightarrow \infty} \sqrt{k}\, \alpha^{\star}_0(n/k)=\left\{
                 \begin{array}{ll}
                    \lambda/\rho, & \, \rho <0\\
                     \lambda, & \, \rho=0.
                 \end{array}
               \right.
\end{equation*}
\noindent Further by Theorem 5.1.2 of \citet{deHaanFerreira:06},
since $x=x_n$ is asymptotically constant,
\begin{eqnarray}\label{TrendProcess}
\nonumber& &\sqrt{k}\biggl\{\ndivk \, \biggl(1-F_n^{(s)}\Bigl( U^{\star}_s\bigl(\ndivk\bigr)+x_n(s)\,a^{\star}_s\bigl(\ndivk \bigr)\Bigr)\biggr)- \bigl(1+\gamma\,x_n(s) \bigr)^{-\frac{1}{\gamma}}\biggr\} \\
& & \mbox{\hspace{1.0cm}} -W_n^{(s)}\Bigl(\bigl(1+\gamma\,x_n(s) \bigr)^{-\frac{1}{\gamma}} \Bigr) \\
\nonumber& & \mbox{\hspace{2.0cm}} -\sqrt{k}\,\alpha^{\star}_s\bigl(\ndivk\bigr)\bigl(1+\gamma\,x_n(s)
\bigr)^{-\frac{1}{\gamma}-1}\overline{\Psi}_{\gamma,\rho}\Bigl(\bigl(1+\gamma\,x_n(s)
\bigr)^{\frac{1}{\gamma}}\Bigr)\conv{P}0
\end{eqnarray}
for a sequence of standard Brownian motions $\{W^{(s)}_n(t)\}_{t\geq
0}$ and with
\begin{equation*}
\overline{\Psi}_{\gamma,\rho}(x):= \left\{
                              \begin{array}{ll}
                                \frac{x^{\gamma+\rho}}{\gamma+\rho}, & \mbox{ } \gamma+\rho \neq 0, \, \rho <0, \\
                                \log x, & \mbox{ }  \, \gamma+\rho= 0, \, \rho <0,\\
                                \frac{1}{\gamma}\,x^{\gamma} \log x, & \mbox{ }\rho =0\neq \gamma ,\\
                                \frac{1}{2}\, (\log x)^2, & \mbox{ } \gamma= \rho=0.
                              \end{array}
                            \right.
\end{equation*}
Finally we write 
\begin{eqnarray*}
    & & \sqrt{k}\,\biggl\{\ndivk \, \biggl(1-F_n^{(s)}\Bigl( U^{\star}_s\bigl(\ndivk\bigr)+x_n(s)\,a^{\star}_s\bigl(\ndivk \bigr)\Bigr)\biggr)-e^{cs} \biggr\} \\
    &=& \sqrt{k}\biggl\{\ndivk \, \biggl(1-F_n^{(s)}\Bigl( U^{\star}_s\bigl(\ndivk\bigr)+x_n(s)\,a^{\star}_s\bigl(\ndivk \bigr)\Bigr)\biggr)- \bigl(1+\gamma\,x_n(s) \bigr)^{-\frac{1}{\gamma}}\biggr\}\\
    & &  \mbox{\hspace{1.0cm}} +\sqrt{k} \biggl\{ \bigl(1+\gamma\,x_n(s) \bigr)^{-\frac{1}{\gamma}}- \Bigl(1+\gamma\,\frac{e^{-cs\gamma}-1}{\gamma}\Bigr)^{-\frac{1}{\gamma}} \biggr\}.
\end{eqnarray*}
Since $x_n(s)\conv{P}(e^{-cs\gamma}-1)/\gamma$, by \eqref{TrendProcess} the first term converges in distribution to
\begin{equation*}
W^{(s)}(e^{cs})-\lambda^{\star} e^{cs(\gamma+1)}\overline{\Psi}_{\gamma,\rho}(e^{-cs}).
\end{equation*}
By \eqref{ANxn} and the delta method the second term converges to
\begin{equation*}
-e^{cs} W^{(0)}(1)-\lambda^\star e^{cs(\gamma+1)}\Bigl(b(s)+\frac{1-e^{-cs\gamma}}{\gamma}\, cs\, I_{\{\rho=0\}}\Bigr).
\end{equation*}
The result follows by Carm\'er's delta method again,
\begin{eqnarray*}
& & \sqrt{k}\,\biggl\{\log\biggl[\ndivk \, \biggl(1-F_n^{(s)}\Bigl( U^{\star}_s\bigl(\ndivk\bigr)+x_n(s)\,a^{\star}_s\bigl(\ndivk \bigr)\Bigr)\biggr)\biggr]-cs \biggr\} \\
&\conv{d}& e^{-cs}W^{(s)}(e^{cs})-W^{(0)}(1)-\lambda^{\star}e^{cs\gamma}\Bigl( \overline{\Psi}(e^{-cs})+b(s)+ \frac{1-e^{-cs\gamma}}{\gamma}\, cs\, I_{\{\rho=0\}}\Bigr).
\end{eqnarray*}
\hfill \ding{111}
\\

\begin{proof}{[of Corollary \ref{CorTest}]}
By virtue of Rao's theorem (\citet{Rao:73}, Section 3.b.4; see also
\citet{Serfling:02}, p.128) pertaining to quadratic forms of
asymptotically normal random vectors, statements
\eqref{TestStat1} and \eqref{TestStat2} follow immediately from the
theorem.
\end{proof}



\appendix
\section{Proofs concerning relations \eqref{TrendPropag}, \eqref{UsU0Dif}, \eqref{Us2U0}, \eqref{ScaleFct}}\label{AppRel}

\underline{Proof} of \eqref{TrendPropag}$\Leftrightarrow$  \eqref{UsU0Dif}.

We use the following result of S.I. Resnick \citep[][Lemma 1.2.12 p.23]{Resnick:71, deHaanFerreira:06}. Let $F_1$ and $F_2$ be two probability distribution functions and let $F_1\in \mc{D}(G_{\gamma})$ for some $\gamma \in \real$. The following two statements are equivalent (with $x^{*}$ the right endpoint of $F_1$).
\begin{description}
\item[(i)] $\qquad \displaystyle{\lim_{x \uparrow x^*}\frac{1-F_2(x)}{1-F_1(x)}=1}$
\item[(ii)] $\qquad \displaystyle{\limit{t} \frac{U_2(t)-U_1(t)}{a_1(t)}=0}$ 
\end{description}
where $U_i:=\bigl(1/(1-F_i)\bigr)^{\leftarrow}$, $i=1,2$. Define the distribution function $F^*_s$ by
\begin{equation*}
1-F_s^*(x)= \min \bigl(1,\,e^{-cs}(1-F_s(x)\bigr).
\end{equation*}
Then (using (i), (ii))
\begin{equation*}
\eqref{TrendPropag} \; \Leftrightarrow \; \lim_{x \uparrow x^*}\frac{1-F^*_s(x)}{1-F_0(x)}=1  \; \Leftrightarrow \; \limit{t} \frac{U_s^*(t)-U_0(t)}{a_0(t)}=0
\end{equation*}
which holds if and only if (since $U_s^*(t)=\bigl(1/(e^{-sc}(1-F_s))\bigr)^{\leftarrow}(t)= U_s(t\,e^{-cs})$)
\begin{equation*}
\limit{t}\, \frac{U_s(t\,e^{-sc})-U_0(t)}{a_0(t)}=0.
\end{equation*}
Hence \eqref{TrendPropag} holds if and only if
\begin{equation*}
\frac{U_s(t)-U_0(t)}{a_0(t)}= \frac{U_s(t)-U_0(t\,e^{cs})}{a_0(t\,e^{cs})}\,\frac{a_0(t\, e^{cs})}{a_0(t)}+ \frac{U_0(t\,e^{cs})-U_0(t)}{a_0(t)} \longrightarrow 0+\frac{e^{\gamma c s}-1}{\gamma}
\end{equation*}
as  $t \rightarrow \infty$. Here we have used that $F_0\in \mc{D}(G_{\gamma})$ implies
\begin{equation*}
\limit{t}\, \frac{U_0(tx)-U_0(t)}{a_0(t)}= \frac{x^{\gamma}-1}{\gamma} \qquad \mbox{ and } \qquad \limit{t} \, \frac{a_0(tx)}{a_0(t)}=x^{\gamma}
\end{equation*}
for $x>0$.

\vspace{1.0cm}

\noindent \underline{Proof} of \eqref{TrendPropag}$\Leftrightarrow$  \eqref{Us2U0} for $\gamma >0$.

The proof of Lemma 1.2.12 in \citet{deHaanFerreira:06} shows that for $\gamma >0$ the following statements are equivalent.
\begin{description}
\item[(i)] $\qquad \displaystyle{\limit{x}\frac{1-F_2(x)}{1-F_1(x)}=1}$
\item[(ii)] $\qquad \displaystyle{\limit{t} \frac{U_2(t)}{U_1(t)}=1}.$ 
\end{description}
The rest of the proof is very similar to the previous one and is omitted.

\vspace{1.0cm}

\noindent \underline{Proof} of \eqref{TrendPropag}$\Leftrightarrow$ \eqref{ScaleFct}.

$F_s \in \mc{D}(G_{\gamma})$ implies
\begin{equation}\label{aux0}
	\limit{t} \, t\, \bigl\{1-F_s\bigl(U_s(t)+x a_s(t)\bigr)\bigr\}= (1+\gamma x)^{-1/\gamma}.
\end{equation}
Relation \eqref{TrendPropag} implies
\begin{equation}\label{aux1}
	\limit{t} \, t\, \bigl\{1-F_s\bigl(U_0(t)+x a_0(t)\bigr)\bigr\}= e^{cs}(1+\gamma x)^{-1/\gamma}.
\end{equation}
We combine this with relation \eqref{ERVUs} and
\begin{equation}\label{aux2}
	\limit{t}\, \frac{a_0(t\,e^{sc})}{a_0(t)}= e^{sc \gamma}.
\end{equation}
We then get
\begin{eqnarray*}
& & t\,\bigl\{ 1-F_s\bigl(U_0(t\,e^{sc}) -x a_0(t\, e^{sc})\bigr)\bigr\}\\
&=&  t\,\Bigl\{ 1-F_s\Bigl( U_0(t) +\frac{e^{s c \gamma -1}}{\gamma}\,a_0(t)\bigl(1+o(1)\bigr)+x\, a_0(t) e^{sc\gamma}\bigl(1+o(1)\bigr)\Bigr)\Bigr\}\\
&=& t\, \Bigl\{ 1-F_s\Bigl(U_0(t) + a_0(t) \bigl(\frac{e^{cs\gamma}-1}{\gamma}+x\,e^{sc\gamma}\bigr)\bigl(1+o(1)\bigr)\Bigr)\Bigr\}\\
&\longrightarrow & e^{cs}\Bigl( 1+\gamma \Bigl\{\frac{e^{sc\gamma}-1}{\gamma}+ x\, e^{sc\gamma} \Bigr\}\Bigr) \, =\, (1+\gamma x)^{-1/\gamma}
\end{eqnarray*}
as $t\rightarrow \infty$. Comparing this with \eqref{aux0} we conclude by Khinchine's convergence-to-types theorem
\begin{equation*}
	\limit{t}\, \frac{U_s(t)-U_0(t)}{a_0(t)}=0 \qquad \mbox{and} \qquad \limit{t}\, \frac{a_s(t)}{a_0(t\, e^{sc})}=1.
\end{equation*}
The latter relation gives the result by \eqref{aux2}.

\section{Sketch of alternative approaches}\label{app}

The subject of extreme value theory (EVT) is the study of the right (or left) tail of a probability distribution near the endpoint. Hence by nature EVT is an asymptotic theory. The basic assumption is
\begin{equation}\label{EVTn}
\limit{n} P\bigl\{\frac{\max(X_1,\, X_2, \ldots,X_n)-b_n}{a_n}\leq x\bigr\}= \exp\bigl\{-(1+\gamma\, x)^{-1/\gamma}\bigr\},
\end{equation}
where $X_1, X_2, \ldots$ are i.i.d. random variables. It follows that the limit distribution has only one parameter, the shift $b_n \in \real$ and scale $a_n>0$ are not parameters of the limit distribution. They depend essentially on the distribution of $X_1$.

When it comes to statistics there are three basic methods:\\

\textbf{1}. \underline{Yearly maxima} (or block maxima). Over a number of years one takes the yearly maximum. Since the yearly maximum is taken over many underlying random variables (albeit not i.i.d.) the assumption is that the yearly maximum $M_j$ can be considered the maximum over a large number $n$ of i.i.d random variables so that
\begin{equation*}
P\bigl\{M_j \leq x\bigr\}\approx \exp \Bigl\{-\Bigl(1+\gamma\,\frac{x-b_n }{a_n} \Bigr)^{-\frac{1}{\gamma}}\Bigr\}
\end{equation*}
where $n$ is unknown. The random variables $M_j$ are i.i.d.. The right hand-side can then be interpreted as a parametric model (GEV: Generalized Extreme Value distribution) so that e.g. the method of maximum likelihood can be applied.

The interpretation of $b_n$ is: the level that has a return period (the mean time between consecutive exceedances of  the level) of $e/(e-1)\approx 1.58$ years. Hence there is no direct intuitive meaning for $b_n$. Also the behavior of $b_n$ as $n \rightarrow \infty$ can not be found.
This method carries a bias stemming from replacing an approximate equality with a firm equality. In contrast to the next case it seems difficult to control that bias.\\

\textbf{2}. \underline{Peaks over threshold}. The basic assumption \eqref{EVTn} implies that with $b(t)=b_{[t]}$, $a(t)=a_{[t]}$ and $[t] $ the integer part of $t$ for $x>0$
\begin{equation}\label{EVTt}
\limit{t} P\Bigl\{\frac{X_1-b(t)}{a(t)}>x|X_{1}>b(t)\Bigr\}= (1+\gamma x)^{-1/\gamma}.
\end{equation}
Select out of $n$ i.i.d. observations the ones that are larger than $b(t)$. These are approximately i.i.d. and (when normalized) follow approximately the GPD distribution $1-(1+\gamma\,x)^{-1/\gamma}$ (\citet{Pickands:75}).

One can take for $b(t)$ one of the order statistics $X_{n-k,n}$. In order to get meaningful results we need to have $k=k_n$ and $k_n\rightarrow \infty$, $k(n)/n \rightarrow 0$ as $n\rightarrow \infty$. Then $X_{n-k,n}$ is close to $b(n/k)$, i.e., the quantile $F^{\leftarrow}(1-k/n)$.

Again, since for $x>b(t)$
\begin{equation*}
P\bigl\{X_1>x|X_1>b(t)\bigr\}\approx \Bigl( 1+\frac{x-b(t)}{a(t)}\Bigr)^{-1/\gamma},
\end{equation*}
one can consider the right hand-side as a parametric model so that e.g. the method of maximum likelihood can be applied.

Next one can prove that the obtained estimators are consistent and asymptotically normal as $n$, the number of observations, tends to infinity. That is, the vector
\begin{equation*}
\sqrt{k}\, \biggl( \frac{\hat{a}\bigl( \ndivk\bigr)}{a\bigl( \ndivk \bigr)}-1, \, \hat{\gamma}_{n,k}-\gamma\biggr)
\end{equation*}
has asymptotically a normal distribution( \citet{Smith:87, Dreesetal:04, Zhou:09}). There are also methods to minimize the bias by choosing the number $k$ appropriately.

\textbf{3.} \underline{Point process convergence}. Suppose that the basic assumption holds. Take the point process on $\real^2$ with points
\begin{equation}\label{PP}
\biggl\{\Bigl( \frac{i}{n}, \, \frac{X_i-b_n}{a_n}\Bigr)\biggr\}_{i=1}^n.
\end{equation}
This point process converges in distribution to a Poisson point process on $(0,1)\times \real$ with intensity measure $dt \cdot (1+\gamma x)^{-1/\gamma-1}dx$ (cf. \citet{Pickands:71}). Note that the intensity measure is unbounded. Those points in \eqref{PP} for which $(X_i-b_n)/a_n$ exceeds some threshold $u$ are approximately points from a Poisson point process with (finite) parametric intensity measure so that the method of maximum likelihood can be applied supplying estimators for $\gamma,\, b_n$ and $a_n$ (see \citet{Smith:89}, cf. \citet{Coles:01}, Chapter 7). No asymptotic behavior ($n\rightarrow \infty$, $u=u_n$ decreasing) seems to be known for these estimators.\\

The three methods have been explained in detail in the book of \citet{Coles:01}. A trend in the EVT analysis has been considered in all three methods.

$\mathbf{1'.}$ Chapter 6 of \citet{Coles:01} book treats trends in the block maxima / GEV setup. One considers time points $j=0,1,2,\ldots$ and assumes (for example)
\begin{equation*}
b_n(j)= b_n(0) + j\,c
\end{equation*}
or/and
\begin{equation*}
\log a_n(j)= \log a_n(0) + j\,c'.
\end{equation*}
As we saw before, $b_n$ is the level that has a return period of just $e/(e-1)$ years. The scale $a_n$ can be interpreted with some liberty as a derivative, i.e., speed of change of location. The interpretation of both seems less straightforward than that of \eqref{TrendPropag}.

There is also another complication. If one is interested in the location parameter over a longer period, say, of 2 years i.e. $n$ replaced with $2n$, the relation is
\begin{eqnarray*}
b_{2n}(j)&\approx& b_n(j)+a_n(j)\,\frac{2^{\gamma}-1}{\gamma}\\
	&\approx& b_{2n}(0)+j\,c+(a_n(j)-a_n(0))\,\frac{2^{\gamma}-1}{\gamma}.
\end{eqnarray*}
This is no longer a linear trend in general. Note that our framework, combining trends in  location and scale, is not bound to a certain period.

$\mathbf{2'.}$\underline{Peaks over threshold}. \citet{DavisonSmith:90} consider a linear trend in both $\gamma$ and $a(n/k)$. \citet{Coles:01}, p.119, considers a linear trend in $\log a(n/k)$. Estimation is done by maximum likelihood. No asymptotic analysis of the quality of the estimators as the number of observations tends to infinity is made. Again the interpretation of a trend in $a(n/k)$ or $\log a(n/k)$ seems difficult.

\citet{HallTajvidi:00} consider a nonlinear trend in $\gamma$. The method is likelihood based. No large sample results are given.\\

$\mathbf{3'.}$ \citet{Smith:89} (c.f. \citet{Coles:01} p.133 sqq.) considers a trend in the location
\begin{equation*}
b_n(j)=b_n(0)+j\,c
\end{equation*}
 (simplified) in the point process model. Estimation is by maximum likelihood. No asymptotic analysis ($n\rightarrow \infty$) is made. 
Another possible problem is that the trend for $b_n(j)$ could get out of range if $\gamma <0$.

In short: the present paper looks at changes in (tail) probabilities whereas in the literature changes in various quantiles have been considered. The two viewpoints are not equivalent.

\section{An auxiliary result}
\label{AppParetos}

\begin{lem}
Let $X_1,X_2,X_3,\ldots$ be i.i.d. random variables with distribution function $F$. Let $X_{1,n}\leq X_{2,n}\leq  \ldots \leq  X_{n,n}$ be the $n$-th order statistics. Define $U:= \bigl(1/(1-F)\bigr)^{\leftarrow}$. If $F\in \mc{D}(G_{\gamma})$ for some $\gamma \in \real$, then
\begin{equation*}
\frac{X_{n-k,n}}{U\bigl(\frac{n}{k}\bigr)} \conv{P}\, 1,
\end{equation*}
as $n \rightarrow \infty$ where $k=k_n$, $k_n \rightarrow \infty$, $k_n/n \rightarrow 0$ as  $n \rightarrow \infty$.
\end{lem}
\begin{proof}
First note that $\{X_1,X_2,X_3,\ldots\}\id \{U(Y_1), U(Y_2),U(Y_3),\ldots\}$ where $Y_1,Y_2,Y_3,\ldots$ are i.i.d. with distribution function $1-1/x$, $x\geq1$. Hence $X_{n-k,n}\id U(Y_{n-k,n})$ for all $n$. Next note that
\begin{equation}\label{ParetoThresh}
\frac{k}{n} Y_{n-k,n} \conv{P}\, 1
\end{equation}\citep[cf. e.g. Corollary 2.2.2 p.41][]{deHaanFerreira:06}.

The regular variation of $U$ \citep[cf. Lemma 1.2.9 p.22][]{deHaanFerreira:06} implies 
\begin{equation}\label{RVt}
	\limit{t}\, \frac{U\bigl(t\,x(t)\bigr)}{U(t)} = 1
\end{equation}
provided $\lim_{t \rightarrow \infty} x(t)=1$. Hence
\begin{equation*}
	\frac{X_{n-k,n}}{U\bigl(\frac{n}{k}\bigr)}\,\id \, \frac{U(Y_{n-k,n})}{U\bigl(\frac{n}{k}\bigr)} \,\id\, \frac{U\bigl(\frac{n}{k}(\frac{k}{n}Y_{n-k,n})\bigr)}{U\bigl(\frac{n}{k}\bigr)}\conv{P}\, 1
\end{equation*}
where in the last step we have used \eqref{ParetoThresh} and \eqref{RVt}.

\end{proof}


\clearpage
\bibliography{bibTrend}
\bibliographystyle{apalike}

\end{document}